\documentclass[11pt]{amsart}
\usepackage{a4wide}
\usepackage{amsmath}
\usepackage[utf8]{inputenc}
\usepackage{amssymb}
\usepackage{amsopn}
\usepackage{epsfig}
\usepackage{amsfonts}
\usepackage{latexsym}
\usepackage{graphicx}
\usepackage{enumerate}
\usepackage[dvipsnames]{xcolor}
\usepackage{ mathrsfs }
\usepackage{tikz}
\usepackage{extarrows}
\usetikzlibrary{arrows,automata}
\usetikzlibrary{cd}
\usepackage{bm}
\usepackage{longtable}

\newtheorem{theorem}{Theorem}[section]
\newtheorem{lemma}[theorem]{Lemma}
\newtheorem{proposition}[theorem]{Proposition}

\theoremstyle{definition}

\newtheorem{example}[theorem]{Example}

\theoremstyle{remark}
\newtheorem{remark}[theorem]{Remark}

\numberwithin{equation}{section}

\newcommand{\R}{\ensuremath{\mathbb{R}}}
\newcommand{\N}{\ensuremath{\mathbb{N}}}
\renewcommand{\L}{\ensuremath{\mathbb{L}}}
\renewcommand{\S}{\ensuremath{\mathbb{S}}}
\newcommand{\A}{\ensuremath{\mathcal{A}}}

\newcommand{\cL}{{\mathcal L}}
\renewcommand{\P}{\ensuremath{\mathcal{P}}}

\renewcommand{\c}{ {\mathbf{c}}}
\renewcommand{\a}{ {\mathbf{a}}}
\renewcommand{\d}{ {\mathbf{d}}}
\renewcommand{\b}{ {\mathbf{b}}}
\renewcommand{\l}{ {\bm{\ell}}}
\newcommand{\p}{ {\mathbf{p}}}

\newcommand{\s}{ {\mathbf{s}}}
\newcommand{\cs}{ {\mathbf{S}}}

\renewcommand{\r}{ {\mathbf{r}}}

\renewcommand{\u}{\mathbf{u}}
\renewcommand{\v}{\mathbf{v}}
\newcommand{\w}{\mathbf{w}}

\newcommand{\set}[1]{\left\{#1\right\}}

\newcommand{\ga}{\gamma}

\newcommand{\Ga}{\Gamma}
\newcommand{\ep}{\varepsilon}
\newcommand{\f}{\infty}
\newcommand{\de}{\delta}

\newcommand{\lle}{\preccurlyeq}
\newcommand{\lge}{\succcurlyeq}
\renewcommand{\a}{ \mathbf{a}}
\newcommand{\si}{\sigma}

\newcommand{\K}{\mathbf K}

\begin{document}
\title[Periodic orbits in  $\beta$-transformation]{Phase transitions on periodic orbits in $\beta$-transformation with a hole at zero}

\author[D. Kong]{Derong Kong}
\address[D. Kong]{College of Mathematics and Statistics, Center of Mathematics,  Chongqing University, Chongqing 401331, People's Republic of China.}
\email{derongkong@126.com}

\author[D. Pu]{Dantong Pu}
\address[D. Pu]{College of Mathematics and Statistics, Chongqing University, Chongqing 401331, People's Republic of China.}
\email{dantongpu@stu.cqu.edu.cn}

\dedicatory{Dedicated to Professor F.~Michel Dekking on the occasion of his 80th birthday.}



\begin{abstract}
 Given $\beta\in(1,2]$, let $T_\beta: [0,1)\to[0,1);~x\mapsto\beta x\pmod 1$. For $m\in\mathbb N$ let 
 \[
 \tau_m(\beta):=\sup\left\{t\in[0,1): K_\beta(t)\textrm{ {contains  a periodic orbit}  of smallest period }m \right\},
 \]
 where $K_\beta(t)=\{x\in[0,1): T_\beta^n(x)\notin(0,t)~\forall n\ge 0\}$ is the survivor set of the open dynamical system $(T_\beta, [0,1), H)$ with a hole $H=(0,t)$. In this paper we give a complete characterization of $\tau_m$, and show that $\tau_m$ is piecewise continuous  with precisely $\psi(m)$ discontinuity points, where $\psi(m)$ is the number of bulbs of   period $m$ in the Mandelbrot set. To describe the critical value function $\tau_m$ we construct a finite butterfly tree $\mathcal T_m$, from which we are able to determine the discontinuity points and the  analytic formula of $\tau_m$ based on Farey words and  substitution operators. As a by product, we characterize the extremal Lyndon words and extremal Perron words. Since we are working in the symbolic space, our result can be applied to study phase transitions for periodic orbits in topologically expansive Lorenz maps, doubling map with an asymmetric hole, intermediate $\beta$-transformations, unique expansions in double bases, and so on.
\end{abstract}


\maketitle
\tableofcontents

\section{Introduction} \label{sec: Introduction}
Open dynamical {system}  was first proposed by Pianigiani and Yorke \cite{Pianigiani-Yorke-1979} in 1979. It focuses on
the study of dynamical systems with holes. In recent years open dynamical systems have received considerable attention from both theoretical and applied perspectives (cf.~\cite{Demers-Wright-Young-2010, Demmers-2005, Demers-Young-2006}). In the general setting,   let $X$ be a compact metric space, and let $T: X\to X$ be a continuous map with positive topological entropy. Take an open connected set $H\subset X$, called a \emph{hole}. We are interested in the \emph{survivor set}
\[
K(H)=X\setminus\bigcup_{n=0}^\f T^{-n}(H)=\set{x\in X: T^n(x)\notin H~\forall n\ge 0}.
\]
Note that the size of ${K}(H)$ depends not only on the size but also on the position of the hole $H$ (cf.~\cite{Bunimovich-Yurchenko-2011}).

 In \cite{Urbanski_1986, Urbanski-87} Urba\'nski considered $C^2$-expanding, orientation-preserving circle maps with a hole of the form $(0,t)$. In particular,  he proved that for the doubling map  $T_2: [0,1)\to[0,1);~x\mapsto 2x\pmod 1$, the Hausdorff dimension of the survivor set
\[
 K_2(t):=\set{x\in[0,1): T_2^n(x)\notin(0,t)~\forall n\ge 0}
\]
 depends continuously on the parameter $t\in[0,1)$. Furthermore, he showed that  the dimension function $\eta_2: t\mapsto \dim_H K_2(t)$ is a non-increasing devil's staircase, and studied its bifurcation set.  Carminati and Tiozzo \cite{Carminati-Tiozzo-2017} showed that the function $\eta_2$ has an interesting analytic property: the local H\"older exponent of $\eta_2$ at any bifurcation point $t$ is equal to $\eta_2(t)$. 

Motivated by the   works of Urba\'nski \cite{Urbanski_1986, Urbanski-87}, Kalle et al.~\cite{Kalle-Kong-Langeveld-Li-18} considered the survivor set in the $\beta$-dynamical system $([0, 1), T_\beta)$ with a hole at zero, where $\beta\in(1,2]$ and $T_\beta: [0,1)\to[0,1);~x\mapsto \beta x\pmod 1$.  More precisely, for $t\in[0,1)$ they determined the Hausdorff dimension of the survivor set
\begin{equation}\label{eq:K-beta-t}
K_\beta(t)=\set{x\in[0,1): T_\beta^n (x)\notin (0,t)~\forall n\ge 0},
\end{equation}
and showed that the dimension function $\eta_\beta: t\mapsto \dim_H K_\beta(t)$ is a non-increasing devil's staircase.  In particular, $\eta_\beta$ is continuous. So there exists a critical value $\tau(\beta)\in[0,1)$ such that
$\dim_H K_\beta(t)>0$ if and only if $t<\tau(\beta)$.
Kalle et al.~\cite{Kalle-Kong-Langeveld-Li-18}   showed that $\tau(\beta)\le 1-\frac{1}{\beta}$ for all $\beta\in(1,2]$, and the equality $\tau(\beta)=1-\frac{1}{\beta}$ holds for infinitely many $\beta\in(1,2]$. The complete description of the critical value $\tau(\beta)$ was recently given by Allaart and the first author in \cite{Allaart-Kong-2023} {(see \cite{Allaart-Kong-2025} for the general case $\beta>2$). For some further topological properties of $K_\beta(t)$ we   refer to \cite{Allaart-Kong-2026}.} 

Periodic orbits play an important role in the study of chaotic dynamical systems (cf.~\cite{Li-Yorke-1975, Sharkovskii-1964, Thunberg-2001}). Furthermore, periodic orbits are  pivotal  in the study of ergodic optimization (cf.~\cite{Contreras-2016, Li-Zhang-2025}). In number theory, Allouche et al.~\cite{Allouche-Clarke-Sidorov-2009} considered the periodic orbits in  unique non-integer base expansions, and showed that the critical bases for the periodic orbits obey the Sharkovskii ordering (see also \cite{Ge-Tan-2017}). Recently, the first author and Zhang \cite{Kong-Zhang-2025} studied the periodic orbits of unique codings in fat Sierpinski gaskets. 

 In this paper we consider the periodic orbits in the survivor set $K_\beta(t)$ {in (\ref{eq:K-beta-t})}. Note that the symbolic {spaces} of unique codings studied in \cite{Allouche-Clarke-Sidorov-2009, Ge-Tan-2017, Kong-Zhang-2025} are symmetric. While the symbolic {space} of $K_\beta(t)$  is in general \emph{not} symmetric {(see Equation (\ref{eq:symbolic-K}) below)}. This makes our study of periodic orbits in $K_\beta(t)$  more involved. 
Given $m\in\N$ and $\beta\in(1,2]$, a periodic point $x\in K_\beta(t)$ is said to have \emph{smallest period $m$} if $T_\beta^m(x)=x$ and $T_\beta^i(x)\ne x$ for $1\le i<m$. Equivalently, $\{T_\beta^n(x)\}_{n=0}^\f$ is a periodic orbit of smallest period $m$. Note that the set-valued map $t\mapsto K_\beta(t)$ is {non-increasing} with respect to the set inclusion. This implies that if $K_\beta(t)$ contains a period orbit of smallest period $m$, then so does $K_\beta(t')$  for any $t'<t$. Define
\begin{equation}\label{eq:def-tau-m}
\tau_m(\beta):=\sup\set{t\in[0,1): K_\beta(t)\textrm{ contains a periodic orbit of smallest period }m},
\end{equation}
where we set $\sup\emptyset=0$. In fact,   the supremum in (\ref{eq:def-tau-m}) can be replaced by maximum.  Then $K_\beta(t)$ contains a periodic orbit of smallest period $m$ if and only if $t\le \tau_m(\beta)$. Clearly, for $m=1$ we have ${\tau_1(\beta)=1}$ for all $\beta\in(1,2]$, since  $T_\beta(0)=0$ for all $\beta\in(1,2]$. So, in the following we always assume $m\ge 2$. 
{We will} give a complete description of $\tau_m$ based on a finite butterfly tree $\mathcal T_m$ and the associated Farey words. We show that $\tau_m: (1,2]\to[0,1)$ is piecewise analytic with $\psi(m)$ discontinuity points, where $\psi(m)$ is the number of leaves in  $\mathcal T_m$.  
Interestingly, $\psi(m)$ is also the number of periodic  bulbs   in the Mandelbrot set (cf.~\cite{Devaney-1995, Devaney-1999, Milnor-2000}).

\subsection{The butterfly tree $\mathcal T_m$ and the critical value $\tau_m$} 
We will describe the critical value $\tau_m$ by constructing  a butterfly tree $\mathcal T_m$. First we introduce some notation.
Let $\N:=\set{1,2,\ldots}$, and  $\N_0:=\N\cup\set{0}$.  For $k\in\N$ let $\N_{\ge k}:=\N\cap[k,+\f)$.  
Write  $\N^*:=\bigcup_{n=1}^\f\N^n$.
 Given $m\in\N_{\ge 2}$, a  vector $(m,k_1,k_2,\ldots, k_\ell)\in\N^*$ is called an \emph{admissible $m$-chain} if 
\[
   k_i<m_i:=\gcd(m,k_1,\ldots, k_{i-1})\quad\forall 1\le i\le \ell,\quad\textrm{and}\quad \gcd(m,k_1,\ldots, k_\ell)=1,
\]
where   $m_1=m$, and  $\gcd(n_1,n_2,\ldots,n_q)$ denotes the \emph{greatest common divisor} of $n_1,n_2,\ldots, n_q$. 
  Let $\A_m$ be the set of all admissible $m$-chains (see Example \ref{ex:T-8} for $m=8$). All of these admissible $m$-chains in $\A_m$ can be represented by the directed paths in  the \emph{butterfly tree} $\mathcal T_m$ which we construct  in the following way. 
 Let $m$ be the root of $\mathcal T_m$, and it has $m-1$ children named $1,2,\ldots, m-1$ from the left to the right. Choose a child $k_1\in\set{1,2,\ldots, m-1}$, if $\gcd(m,k_1)=1$ then the vertex $k_1$ has no offsprings; otherwise, let $m_2=\gcd(m,k_1)$, and then the vertex $k_1$ has $m_2-1$ offsprings named $1,2,\ldots, m_2-1$ from the left to the right. {This procedure will stop after  finitely many steps. Eventually,} we    construct a finite tree $\mathcal T_m$ (see Figure \ref{fig:tree-8}    for   $\mathcal T_8$). Since each $\mathcal T_m$ is symmetric and looks like a butterfly, this explains why we call $\mathcal T_m$ a  butterfly tree. 
   \begin{figure}[h!]
   \begin{tikzpicture}[level distance=12mm]
   \node{$8$}
   child{node{$1$}}
   child{node{$2$}
 child{node{$1$}}}
 child{node{$3$}}
child{node{$4$}
child{node{$1$}}
child{node{$2$}
child{node{$1$}}}
child{node{$3$}}}
child{node{$5$}}
child{node{$6$}
child{node{$1$}}}
child{node{$7$}}
;
\end{tikzpicture} 
\caption{The butterfly tree $\mathcal T_8$ {has} $9$ leaves. Each directed path from the root to a leaf corresponds to an admissible $8$-chain in $\mathcal A_8$.}\label{fig:tree-8}
 \end{figure}
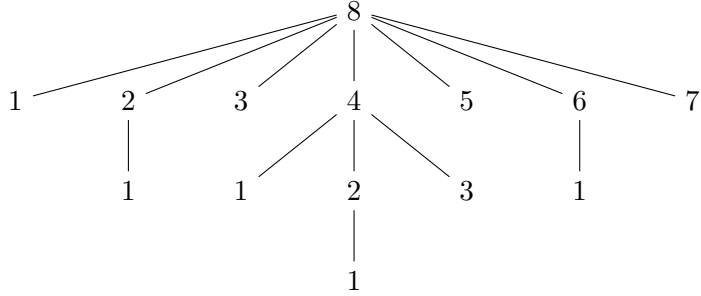
\begin{example}\label{ex:T-8}
  Let $m=8$. Then $\mathcal A_8$ consists of $9$ admissible $8$-chains: \[(8,1),\quad (8,2,1),\quad (8,3),\quad (8,4,1),\quad (8,4,2,1),\quad (8,4,3), \quad(8,5), \quad(8,6,1),\quad (8,7).\]
  These admissible $8$-chains {from the left to the right} are listed in the lexicographically increasing order. Note that each admissible $8$-chain determines a unique directed path from the root to a leaf in the butterfly tree $\mathcal T_8$ (see Figure \ref{fig:tree-8}).   
\end{example}
\begin{proposition}
  \label{prop:property-butterfly-tree}
 Each directed path from the root $m$ to   a leaf in $\mathcal T_m$ is bijectively mapped  to an admissible  $m$-chain in $\A_m$. Furthermore, for any $m\ge 2$ we have $\#\A_m=\psi(m)$ which satisfies
  \begin{equation}\label{eq:psi}
  \psi(n)=\sum_{k=1}^{n-1}\psi(\gcd(n,k))\quad \textrm{for all } n\ge 2,
\end{equation}  
where $\psi(1)=1.$
    \end{proposition}
    \begin{proof}
    First we show that  each admissible $m$-chain in $\A_m$ can be bijectively mapped to a directed path from the root $m$ to a leaf in $\mathcal T_m$. Take a directed path $e_1e_2\ldots e_j$ from the root $m$ to a leaf in $\mathcal T_m$. Then the terminal vertex of $e_j$ has no offsprings. Let {$k_i\in\N$} be the terminal vertex of $e_i$ for $i=1,2,\ldots, j$. By the construction of $\mathcal T_m$ it follows that {(i) $k_1<m_1$; (ii) $k_i<m_i=\gcd(m,k_1,\ldots, k_{i-1})$} for $2\le i\le j$; (iii) $\gcd(m,k_1,\ldots,k_j)=1$. So, $(m,k_1,\ldots, k_j)$ is an admissible $m$-chain, i.e., $(m,k_1,\ldots,k_j)\in\A_m$. 
    
    On the other hand, let $(m,k_1,\ldots, k_j)$ be an admissible $m$-chain. Then by our construction of $\mathcal T_m$ there exists a directed path $e_1e_2\ldots e_j$ such that 
    \[e_1: m\to k_1;\quad e_2: k_1\to k_2;\quad \cdots; \quad e_j: k_{j-1}\to k_j;\]
    and the vertex $k_j$ has no offsprings since $\gcd(m,k_1,\ldots,k_j)=1$. So, each admissible $m$-chain in $\A_m$ can be bijectively mapped to a directed path in $\mathcal T_m$ from the root $m$ to a leaf. 
    
    Next we prove by induction on $m$ that  $\#\A_m=\psi(m)$. Clearly, for $m=2$ we have $\A_2=\set{(2,1)}$, and thus $\#\A_2=1=\psi(2)$.   Now suppose $\#\A_n=\psi(n)$ for all $n\le m$ with $m\in\N_{\ge 2}$, and we consider   $\A_{m+1}$. 
{Take  $(m+1, k_1,k_2,\ldots, k_j)\in\A_{m+1}$. If $\gcd(m+1,k_1)>1$, then 
   \[
   (\gcd(m+1,k_1),k_2,k_3,\ldots, k_j)\in\A_{\gcd(m+1,k_1)}.
    \]
    Otherwise, $\gcd(m+1,k_1)=1$ and then $(m+1,k_1)\in\A_{m+1}$.
  Since $k_1\in\set{1,2,\ldots, m}$, we have $\gcd(m+1,k_1)\le m$. By the induction hypothesis and our construction of $\mathcal T_m$ we obtain that
   \begin{align*}
   \#\A_{m+1}&=\sum_{\substack{k_1\in\set{1,2,\ldots,m}\\ \gcd(m+1,k_1)>1}}\#\A_{\gcd(m+1,k_1)}+\#\set{k_1\in\set{1,2,\ldots,m}:\gcd(m+1,k_1)=1}\\
   &=\sum_{\substack{k_1\in\set{1,2,\ldots,m}\\ \gcd(m+1,k_1)>1}}\psi(\gcd(m+1,k_1))+\sum_{\substack{k_1\in\set{1,2,\ldots,m}\\ \gcd(m+1,k_1)=1}}\psi(1)\\
   &=\sum_{k_1=1}^m\psi(\gcd(m+1,k_1))=\psi(m+1),
    \end{align*}
    where the last equality holds by (\ref{eq:psi}). By induction this proves $\#\A_m=\psi(m)$ for all $m\ge 2$. } 
    \end{proof}

\begin{figure}[h!]
  \centering
  \includegraphics[width=12cm]{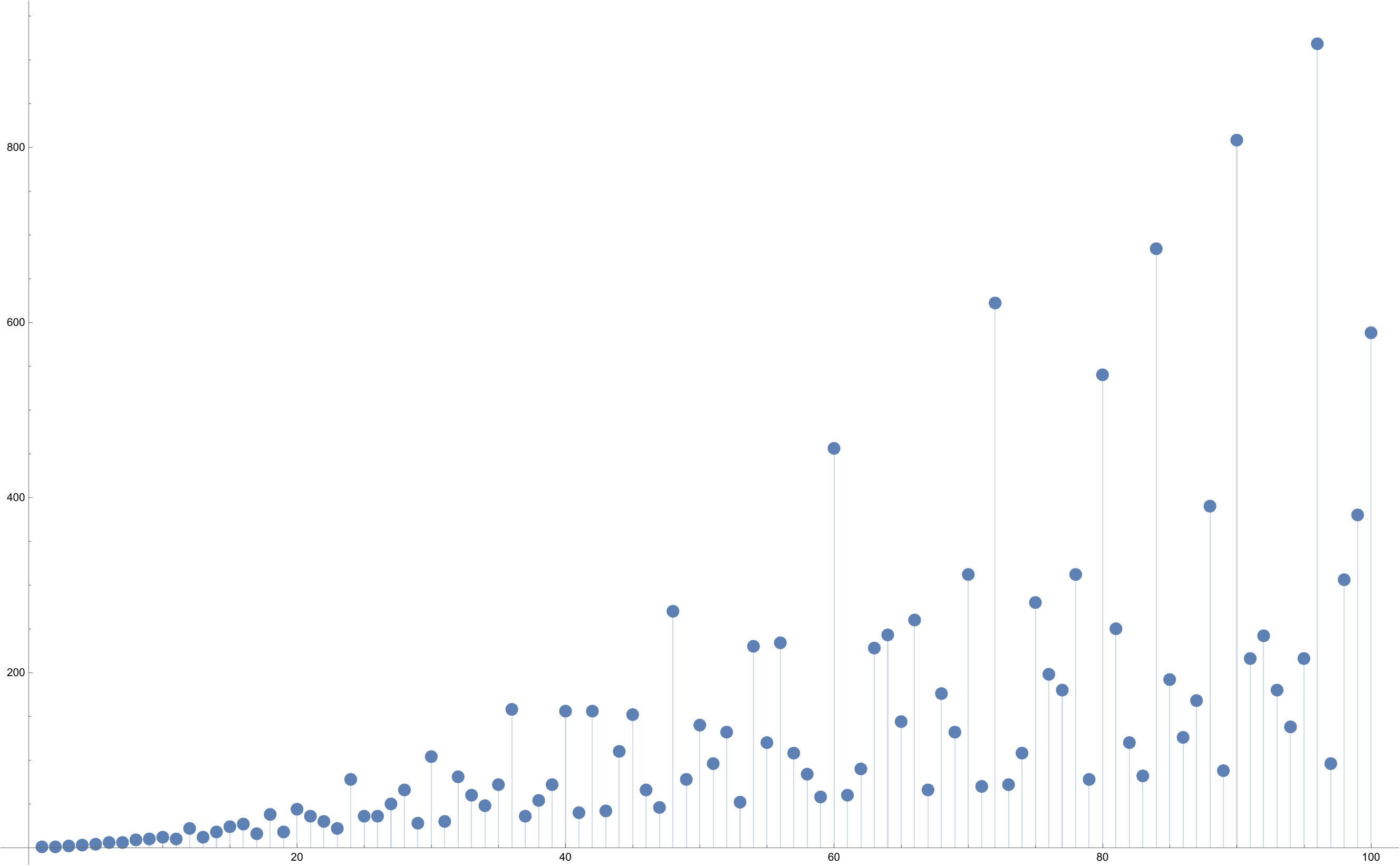}
  \caption{The graph of the arithmetic function $\psi(n)$ for $1\le n\le 100$.}\label{fig:partion-numbers-psi}
\end{figure}

By (\ref{eq:psi}) we can easily calculate the sequence $\set{\psi(n)}$, see Figure \ref{fig:partion-numbers-psi} for the first $100$ terms. We will show that the critical value function {$\tau_m$} has precisely $\psi(m)$ discontinuity points.  Note that the sequence $\set{\psi(n)}$ coincides with the integer sequence defined in A006874-OEIS, from which we found that $\psi(n)$ is also the number of bulbs of   period $n$ in the Mandelbrot set. More precisely, for $c\in\mathbb C$ let $f_c(z)=z^2+c$ be the   quadratic map on the complex plain $\mathbb C$. The Mandelbrot set $\mathcal M$ consists of all $c\in\mathbb C$ in which the critical orbit $\set{f_c^n(0)}_{n=1}^\f$ is bounded. $\mathcal M$ is also the set of parameters $c\in\mathbb C$ such that the Julia set $J_c$ is connected, where $J_c$ is the attractor of the quadratic map $f_c$. In the Mandelbrot set $\mathcal M$ there are many bulbs, and each bulb consists of those parameters $c\in\mathbb C$ in which the critical orbit $\set{f_c^n(0)}$ is eventually periodic of the same period. We call a bulb of \emph{period $n$} if for any $c$ in this bulb the critical orbit $\set{f_c^n(0)}$ is eventually periodic of period $n$. It is known that the number of bulbs of period $n$ in $\mathcal M$ is indeed $\psi(n)$ for {any} $n\in\N$ (see, e.g., \cite{Devaney-1995, Devaney-1999, Milnor-2000}). 

To state our main result on the critical value $\tau_m$ we need some notation from symbolic dynamics. 
For any  $p/q\in\mathbb Q\cap(0,1)$ with $\gcd(p,q)=1$, there exists a unique \emph{Farey word} $\w_{p/q}\in\set{0,1}^*$ of length $q$ with precisely $p$ ones (see Section \ref{subsec:Farey-words} for more details). Let $\mathcal F^*$ be the set of all Farey words  of length at least two. Then  there is a bijection between   $\mathbb Q\cap(0,1)$ and $\mathcal F^*$ (see Lemma \ref{lem:bijection-Fareyword-Rational} below). For a word $\w=w_1w_2\ldots w_n\in\set{0,1}^*$ let $\L(\w)$ be the lexicographically largest cyclic permutation of $\w$, i.e., 
\[\L(\w)=\max\set{w_{i+1}\ldots w_nw_1\ldots w_i: i=0,1,\ldots,n-1}.\]
 If $\w=w_1\ldots w_n$ ends with digit $0$, then we write $\w^+:=w_1\ldots w_{n-1}(w_n+1)$; and if $\w$ ends with digit $1$, then we write $\w^-:=w_1\ldots w_{n-1}(w_n-1)$. Furthermore, let $\w^\f=\w\w\cdots\in\set{0,1}^\N$ be the periodic sequence constructed  by concatenating $\w$ with itself infinitely many times.   

Next we define a substitution operator $\bullet$  as in \cite{Allaart-Kong-2023} (see Section \ref{subsec:Lyndon-Perron-words} for more explanation).  For $\s\in\mathcal F^*$ and a word $\r=r_1\ldots r_n\in\set{0,1}^*$, we define
\begin{equation}\label{eq:bullet-operator} 
\s\bullet\r=\w_1\w_2\ldots \w_n
\end{equation}
such that   
{\begin{enumerate} [{\rm(i)}]
\item  $\w_1=\s^-$ if $r_1=0$, and  $\w_1=\L(\s)^+$ if $r_1=1$; 
\item for any $k\in\set{1,2,\ldots,n-1}$ we have 
\[
\w_{k+1}=\left\{
\begin{array}{lll}
  \L(\s) & \textrm{if} & r_kr_{k+1}=00, \\
  \L(\s)^+ & \textrm{if} & r_kr_{k+1}=01,\\
  \s^-&\textrm{if}& r_kr_{k+1}=10,\\
  \s&\textrm{if}& r_kr_{k+1}=11.
\end{array}\right.
\]
\end{enumerate}}

For $\beta\in(1,2]$ and a sequence $(d_i)=d_1d_2\ldots\in\set{0,1}^\N$, let
 \begin{equation}\label{eq:beta-expansion}
((d_i))_\beta:=\frac{d_1}{\beta}+\frac{d_2}{\beta^2}+\frac{d_3}{\beta^3}+\cdots\in\left[0,\frac{1}{\beta-1}\right].
\end{equation} 
The infinite sequence $(d_i)$ is called a \emph{$\beta$-expansion} of $x=((d_i))_\beta$ (see Section \ref{subsec:beta-expansions} for more details).
  Given $m\in\N_{\ge 2}$, for each admissible $m$-chain  $(m,k_1,\ldots, k_j)\in\A_m$ let $\beta_{m,k_1,\ldots, k_j}$ be the largest $\beta \in(1,2)$ satisfying 
\begin{equation}\label{eq:partition-base}
(\L(\w_{k_1/m_1}\bullet\w_{k_2/m_2}\bullet\cdots\bullet\w_{k_j/m_j})^\f)_\beta=1,
\end{equation}
where $m_1=m$ and $m_i=\gcd(m,k_1,\ldots, {k_{i-1}})$ for $2\le i\le j$. Note that $\gcd(m_i,k_i)=m_{i+1}$ for all $1\le i\le j$, where $m_{j+1}=1$. Then $\w_{k_i/m_i}$ stands for the Farey word generated by the rational number $\frac{k_i}{m_i}=\frac{k_i/m_{i+1}}{m_i/m_{i+1}}$. Observe that  these bases $\beta_{m,k_1,\ldots,k_j}$ with $(m,k_1,\ldots,k_j)\in\A_m$ form a  $m$-partition of $(1,2]$:
\begin{equation}\label{eq:partition-intervals}
(1,2]=(1,\beta_{m,1}]\cup(\beta_{m,m-1},2]\cup\bigcup_{(m,k_1,\ldots,k_j)\in\A_m\setminus\set{(m,m-1)}} I_{m,k_1,\ldots,k_j}
\end{equation}
with the union pairwise disjoint, where $I_{m,k_1,\ldots,k_j}=(\beta_{m,k_1,\ldots,k_j}, \beta_{m,k_1',\ldots,k_l'}]$ is a $m$-partition interval. {Note by Proposition \ref{prop:property-butterfly-tree} that $\#\A_m=\psi(m)$. Then we have $\psi(m)+1$ partition intervals in (\ref{eq:partition-intervals}), and}  the first and the last $m$-partition intervals are $(1,\beta_{m,1}]$ and $(\beta_{m,m-1},2]$, respectively. 

 {For two vectors $(a_1,\ldots,a_j), (b_1,\ldots,b_l)\in\N^*$ we write $(a_1\ldots a_j)\prec (b_1,\ldots, b_l)$ if {the two periodic sequences $(a_1\ldots a_j)^\f\prec(b_1\ldots b_l)^\f$ in lexicographical ordering (see Section \ref{sec:preliminaries} for its definition)}. Now we state our main result on the critical value $\tau_m$.}

\begin{theorem}
  \label{main-3:critical-value-details}
 For any $m\in\N_{\ge 2}$,  the critical value function $\tau_m$ has $\psi(m)$ discontinuity points   $\beta_{m,k_1,\ldots, k_j}$ with $(m,k_1,\ldots, k_j)\in\A_m$, and 
  \[
  \beta_{m,k_1,\ldots, k_j}<\beta_{m,k_1',\ldots, k_l'}\quad\Longleftrightarrow\quad (m,k_1,\ldots, k_j)\prec(m,k_1', \ldots, k_l').
  \]
  Furthermore, the critical value function $\tau_m$  is   determined as follows. 
  \begin{enumerate}[{\rm(i)}]
     \item If $1<\beta\le \beta_{m,1}$, then $\tau_m(\beta)=0$.
     \item If $\beta_{m,m-1}<\beta\le 2$, then $\tau_m(\beta)=(\w_{(m-1)/m}^\f)_\beta$.
     
    \item If $\beta\in I_{m,k_1,\ldots,k_j}$ for some other $m$-partition interval $I_{m,k_1,\ldots,k_j}$, then  
    \[
    \tau_m(\beta)=((\w_{k_1/m_1}\bullet\w_{k_2/m_2}\bullet\cdots\bullet\w_{k_j/m_j})^\f)_\beta,
    \]
    where $m_1=m$ and $m_i=\gcd(m,k_1,\ldots,{k_{i-1}})$ for $2\le i\le j$.

  \end{enumerate}
  
\end{theorem}

\begin{remark}
\begin{enumerate}[{\rm(i)}]
\item In the proof of Theorem \ref{main-3:critical-value-details} we determine the extremal Lyndon words and extremal Perron words. More precisely, let $\cL_{m,k}$ be the set of all Lyndon words of length $m$ with precisely $k$ ones, and let $\P_{m,k}=\set{\L(\w): \w\in\cL_{m,k}}$ be the set of all Perron words of length $m$ with $k$ ones (see Section \ref{subsec:Lyndon-Perron-words} for more details). It is clear that $\min\cL_{m,k}=0^{m-k}1^k$ and $\max\P_{m,k}=1^k0^{m-k}$. We prove in Theorem \ref{main-2:L-P-mk} that 
    {\[
    \max\cL_{m,k}=\left\{\begin{array}{lll}
                           \w_{{k}/{m}} & \textrm{if} & \gcd(m,k)=1, \\
                           \w_{{k}/{m}}\bullet 01^{d-1} & \textrm{if} & \gcd(m,k)=d>1;
                         \end{array}\right.
    \]
    and
    \[
     \min\P_{m,k}=\left\{\begin{array}{lll}
                           \L(\w_{{k}/{m}}) & \textrm{if} & \gcd(m,k)=1, \\
                           \w_{{k}/{m}}\bullet 10^{d-1} & \textrm{if} & \gcd(m,k)=d>1. 
                         \end{array}\right.
    \]}

\item   {For any $\beta\in(1,2]$ and any $m\in\N$},  we   give an algorithm to determine   $\tau_m(\beta)$  {(see Proposition \ref{prop:algorithm})}. For $m, n\in\N_{\ge 2}$ we give  {in Proposition \ref{prop:critical-m>n}} an efficient way to determine whether $\tau_m(\beta)>\tau_n(\beta)$. Furthermore, for any sequence $(r_i)\subset\mathbb Q\cap(0,1)$ we   determine the unique $\beta\in(1,2]$ such that $\beta$ belongs to  {the partition intervals $I_{r_1,\ldots, r_n}$ for all $n\in\N$ (see Proposition \ref{prop:beta-infinite-intervals})}. 
    
    \item {By (\ref{eq:partition-base}) and \cite{Blanchard-1989} it follows that the endpoints of all non-extremal $m$-partition intervals are   Perron numbers.} 
    \end{enumerate}
\end{remark}

At the end of this part we give an example  to illustrate Theorem \ref{main-3:critical-value-details}.
\begin{example}

 Let $m=8$. In terms of the butterfly tree $\mathcal T_8$ in Figure \ref{fig:tree-8}, we   replace each vertex $k_i$ (except the root) in $\mathcal T_8$ by a Farey word $\w_{k_i/m_i}$ (see Figure \ref{fig:Farey-tree-8}).  {Based on this new butterfly tree we}  describe the  partition points $\beta_{m,k_1,\ldots,k_j}$ and the critical value $\tau_m$.  
   \begin{figure}
   [h!]
   \begin{tikzpicture}[level distance=15mm]
   \node{$8$}
   child{node{$\w_{1/8}$}}
   child{node{$\w_{2/8}$}
 child{node{$\w_{1/2}$}}}
 child{node{$\w_{3/8}$}}
child{node{$\w_{4/8}$}
child{node{$\w_{1/4}$}}
child{node{$\w_{2/4}$}
child{node{$\w_{1/2}$}}}
child{node{$\w_{3/4}$}}}
child{node{$\w_{5/8}$}}
child{node{$\w_{6/8}$}
child{node{$\w_{1/2}$}}}
child{node{$\w_{7/8}$}}
;

\end{tikzpicture}
 
\caption{The butterfly  tree $\mathcal T_8$ with Farey words as its vertices.}\label{fig:Farey-tree-8}
 \end{figure}
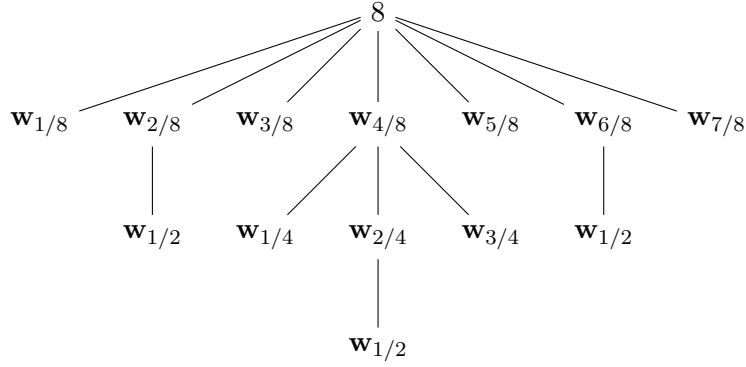
  By our construction of Farey words in (\ref{eq:farey-word}) we obtain that 
 \begin{align*}
&\w_{{4}/{8}}=\w_{{2}/{4}}=\w_{{1}/{2}}=01,\quad \w_{{2}/{8}}=\w_{{1}/{4}}=0001,\quad \w_{{6}/{8}}=\w_{{3}/{4}}=0111,\\
&  \w_{{1}/{8}}=00000001, \quad \w_{{3}/{8}}=00100101,\quad \w_{{5}/{8}}=01011011, \quad \w_{{7}/{8}}=01111111.
 \end{align*}
 Then by  (\ref{eq:partition-base}) it follows that  
 \begin{align*}
   \beta_{8,1} & \sim \L(\w_{ {1}/{8}})^\f=(10000000)^\f, \quad
    \beta_{8,2,1}    \sim \L(\w_{2/8}\bullet \w_{1/2})^\f  =(10010000)^\f,\\
  \beta_{8,3} &\sim \L(\w_{3/8})^\f=(10100100)^\f,\quad
    \beta_{8,4,1}  \sim \L(\w_{4/8}\bullet\w_{1/4})^\f  =(11001010)^\f,\\
    \beta_{8,4,2,1} &\sim \L(\w_{4/8}\bullet\w_{2/4}\bullet\w_{1/2})^\f =(11010010)^\f,\\
        \beta_{8,4,3}  &\sim \L(\w_{4/8}\bullet\w_{3/4})^\f =(11010100)^\f,\quad
    \beta_{8,5}  \sim \L(\w_{5/8})^\f=(11011010)^\f,\\
    \beta_{8,6,1} &\sim \L(\w_{6/8}\bullet\w_{1/2})^\f=(11110110)^\f,\quad
    \beta_{8,7}  \sim \L(\w_{7/8})^\f=(11111110)^\f,
 \end{align*}
 which implies 
 \[
 0<\beta_{8,1}<\beta_{8,2,1}<\beta_{8,3}<\beta_{8,4,1}<\beta_{8,4,2,1}<\beta_{8,4,3}<\beta_{8,5}<\beta_{8,6,1}<\beta_{8,7}<2. 
 \]
 {Therefore, by  Theorem \ref{main-3:critical-value-details} we have}
 \[
 \tau_8(\beta)=\left\{
 \begin{array}{lll}
   0 & \textrm{if} & 1<\beta\le \beta_{8,1};\\
    (\w_{1/8}^\f)_\beta=((00000001)^\f)_\beta & \textrm{if} & \beta_{8,1}<\beta\le \beta_{8,2,1};\\
    ((\w_{2/8}\bullet\w_{1/2})^\f)_\beta=((00001001)^\f)_\beta & \textrm{if} & \beta_{8,2,1}<\beta\le \beta_{8,3};\\
    (\w_{3/8}^\f)_\beta=((00100101)^\f)_\beta & \textrm{if} & \beta_{8,3}<\beta\le \beta_{8,4,1};\\
    ((\w_{4/8}\bullet\w_{1/4})^\f)_\beta =((00101011)^\f)_\beta & \textrm{if} & \beta_{8,4,1}<\beta\le \beta_{8,4,2,1};\\
    ((\w_{4/8}\bullet\w_{2/4}\bullet\w_{1/2})^\f)_\beta= ((00101101)^\f)_\beta & \textrm{if} & \beta_{8,4,2,1}<\beta\le \beta_{8,4,3};\\
    ((\w_{4/8}\bullet\w_{3/4})^\f)_\beta=((00110101)^\f)_\beta & \textrm{if} & \beta_{8,4,3}<\beta\le \beta_{8,5};\\
    (\w_{5/8}^\f)_\beta =((01011011)^\f)_\beta & \textrm{if} & \beta_{8,5}<\beta\le \beta_{8,6,1};\\
    ((\w_{6/8}\bullet\w_{1/2})^\f)_\beta=((01101111)^\f)_\beta  & \textrm{if} & \beta_{8,6,1}<\beta\le \beta_{8,7};\\
    (\w_{7/8}^\f)_\beta=((01111111)^\f)_\beta & \textrm{if} & \beta_{8,7}<\beta\le 2.
 \end{array}\right.
 \]
 By numerical calculation we plot the graph of $\tau_8$ in Figure \ref{fig:critical-8}. 
\begin{figure}[h!]
  \centering
  \includegraphics[width=12cm]{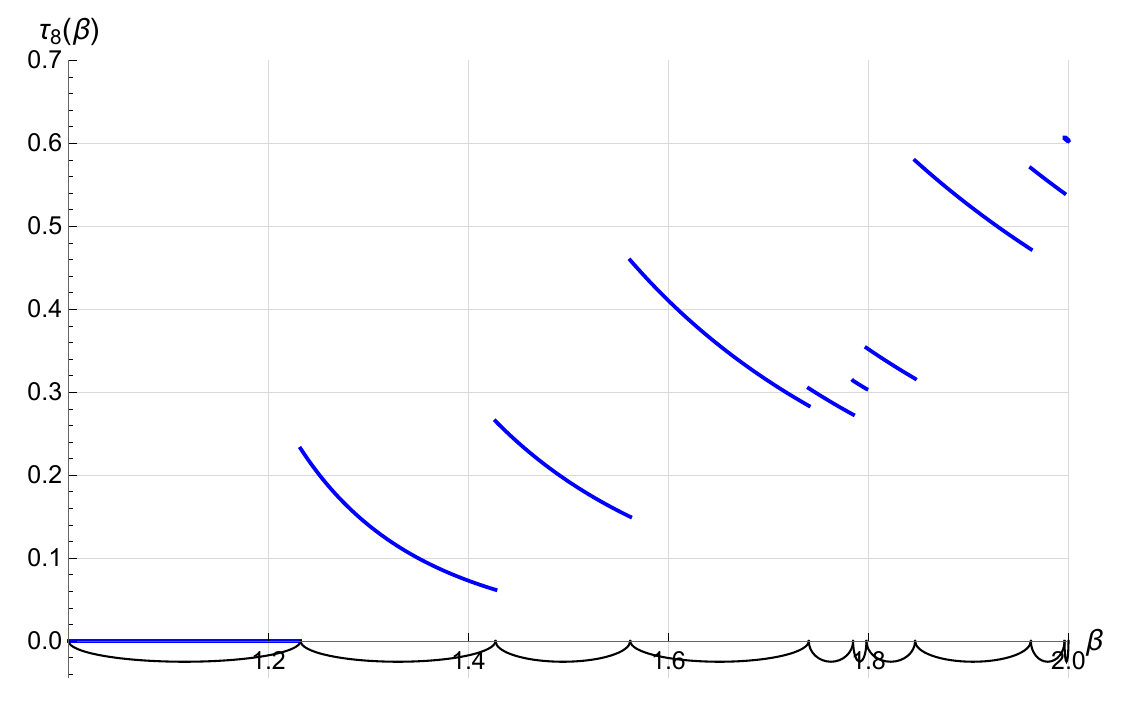}
  \caption{The graph of the critical value function $\tau_8(\beta)$ for $\beta\in(1,2]$, {which}  has $\psi(8)=9$ discontinuity points. {The curve fragments below the $\beta$-axis stand} {for the partition intervals of $\tau_8$}.}\label{fig:critical-8}
\end{figure}
 
\end{example}

 \subsection{Applications to  Lorenz maps and beyond} 
 Note that the survivor set $K_\beta(t)$  is closely related to the set of kneading sequences of Lorenz maps (cf.~\cite{Allaart-Kong-2026}). A function $f: [0,1]\to [0,1]$ is said to be a \emph{Lorenz map} if there exists $c\in(0,1)$ such that $f$ is continuous and strictly increasing on both $[0,c)$ and $(c,1]$ with $f(c-)=1, f(c+)=0$.   {Furthermore, a Lorenz map $f$ is called \emph{topologically expansive} if the set $\bigcup_{n=0}^\f f^{-n}(\{c\})$ of all preimages  of the critical point $c$ is dense in $[0,1]$.} For a point $x\in[0,1]\setminus\bigcup_{n=0}^\f f^{-n}(\set{c})$, there is a unique sequence 
 \[\mathbf k_f(x)=\ep_1\ep_2\ldots\in\set{0,1}^\N,\]
  called the \emph{kneading sequence} of $x$,  such that $\ep_n=0$ if $f^{n-1}(x)<c$; and $\ep_n=1$ if $f^{n-1}(x)>c$. If $x\in\bigcup_{n=0}^\f f^{-n}(\set{c})$, then we define two kneading sequences:
 \[
 \mathbf k_f^-(x)=\lim_{y\to x-}\mathbf k_f(y),\quad \mathbf k_f^+(x)=\lim_{y\to x+}\mathbf k_f(y),
 \]
 where $y$ runs over points that are not preimages of $c$ under $f$, and the limits are with respect to the product topology on $\set{0,1}^\N$. 

{ Let  \[
\Sigma_f:=\set{\mathbf k_f^-(x): x\in[0,1]}\cup\set{\mathbf k_f^+(x): x\in[0,1]}
\]
 be the set of all kneading sequences}. Hubbard and Sparrow \cite{Hubbard-Sparrow-1990} showed that $\Sigma_f$ is uniquely determined by a pair of sequences $(\a_f, \b_f):=(\mathbf k_f^+(0), \mathbf k_f^-(1))$, called the \emph{kneading invariants} of the Lorenz map $f$. More precisely, 
\[\Sigma_f=\set{(d_i)\in\set{0,1}^\N: \a_f\lle \si^n((d_i))\lle \b_f\quad\forall n\ge 0}.\] 
In \cite{Hubbard-Sparrow-1990} it was shown that  a pair of sequences $(\a, \b)\in\set{0,1}^\N\times\set{0,1}^\N$ is   kneading invariants of some topologically expansive Lorenz map if and only if 
\begin{equation}\label{eq:kneading-invariants}
\a\lle \si^n(\a)\prec \b\quad\textrm{and}\quad \a\prec\si^n(\b)\lle \b\quad\forall n\ge 0. 
\end{equation}
For some further study of topologically expansive Lorenz maps we refer to \cite{Oprocha-Potorski-Raith-2019, Cholewa-Oprocha-2024, Sun-Li-Ding-2024} and the references therein.  

In general, for $\a, \b\in\set{0,1}^\N$ we define the  subshift by 
\begin{equation}\label{eq:Sigma-ab}
\Sigma_{\a, \b}:=\set{(d_i)\in\set{0,1}^\N: \a\lle \si^n((d_i))\lle \b\quad\forall n\ge 0}.
\end{equation}
To avoid the trivial cases we assume that $\a$ begins with digit $0$ and $\b$ begins with digit $1$. 
It is known that (cf.~\cite{Labarca-Moreira-2006})
$
\Sigma_{\a, \b}=\Sigma_{\ell_{\a,\b}, r_{\a,\b}},
$
where $\ell_{\a,\b}:=\min\Sigma_{\a, \b}$ and $r_{\a,\b}:=\max\Sigma_{\a,\b}$. So, it suffices to study   $\Sigma_{\a,\b}$ with $\a, \b\in\Sigma_{\a,\b}$. {In terms of (\ref{eq:kneading-invariants}) and (\ref{eq:Sigma-ab}), this is nearly equivalent to study $\Sigma_f$ for some topologically expansive Lorenz map $f$. Note that for $\beta\in(1,2]$ and $t\in(0,1)$ the survivor set $K_\beta(t)$ defined in (\ref{eq:K-beta-t}) can be essentially symbolically written as 
\begin{equation}\label{eq:symbolic-K}
\mathbf K_\beta(t):=\set{(d_i)\in\set{0,1}^\N: (t_i)\lle \si^n((d_i))\prec\de(\beta)~\forall n\ge 0},
\end{equation}
where $(t_i)$ is the greedy $\beta$-expansion of $t$ and $\de(\beta)$ is the quasi-greedy $\beta$-expansion of $1$ (see Section \ref{sec:preliminaries} for more details). In view of (\ref{eq:Sigma-ab}) and (\ref{eq:symbolic-K}), the study of $\Sigma_{\a,\b}$   is also nearly equivalent to the  study of $K_\beta(t)$ for some $t\in(0,1)$ and $\beta\in(1,2]$ (cf.~\cite{Allaart-Kong-2026}). }

The subshift $\Sigma_{\a,\b}$ has connection with the doubling map $T_2$ with an asymmetric hole (cf.~\cite{Glendinning-Sidorov-2015}). Recall   the doubling map   $T_2: [0,1)\to[0,1);~ x\mapsto 2 x\pmod 1$. For $0\le a<b<1$ we define the survivor set $K_2(a,b)$ by
 \[
 K_2(a,b):=\set{x\in[0,1): T^n_2(x)\notin (a,b)~\forall n\ge 0}.
 \]
 Glendinning and Sidorov \cite{Glendinning-Sidorov-2015}  determined (i) when $K_2(a,b)$
 is   nonempty; (ii) when $K_2(a,b)$ is infinite; and (iii) when $K_2(a,b)$ has positive Hausdorff dimension. In particular,  they proved that when the size of the hole $(a,b)$ is strictly smaller than  $0.175092$, the survivor set $K_2(a,b)$ has positive Hausdorff dimension. The work of Glendinning and Sidorov was partially extended by Clark \cite{Lyndsey-2016} to {the} $\beta$-dynamical system $([0,1), T_\beta)$ with a hole $(a,b)$. The   survivor set $K_2(a,b)$ can be symbolically written as
 \[
 \Omega_{\a,\b}:=\set{(d_i)\in\set{0,1}^\N: \si^n((d_i))\lle\a~\textrm{ or }~\si^n((d_i)) \lge \b~\forall n\ge 0},
 \]
 where $\a,\b\in\set{0,1}^\N$ are the greedy $2$-expansions of $a$ and $b$ respectively. It is known \cite{Komornik-Steiner-Zou-2024} that for any $\a, \b\in\set{0,1}^\N$,
 \[
 \Omega_{0\b,1\a}=\bigcup_{n=0}^\f 0^n\Sigma_{\a,\b}\cup\bigcup_{n=0}^\f 1^n\Sigma_{\a,\b}\cup\set{0^\f, 1^\f},
 \]
 where $\Sigma_{\a,\b}$ is defined in (\ref{eq:Sigma-ab}). This implies that for any $m\in\N_{\ge 2}$, $\Omega_{0\b,1\a}$ contains a periodic sequence of smallest period $m$ if and only if $\Sigma_{\a,\b}$ contains a periodic sequence of smallest period $m$.
 
{In the literature,} the subshift $\Sigma_{\a,\b}$ is also related to admissible expansions in the   intermediate $\beta$-transformation  $T_{\beta,\alpha}: x\mapsto \beta x+\alpha\pmod 1$, where $\beta\in(1,2]$ and $\alpha\in(0, 2-\beta)$. For more details we refer to \cite{Dajani_Kraaikamp_2003, Dajani-Kraaikamp-2002, Li-Sahlsten-Samuel-2016, Li-Sahlsten-Samuel-Steiner-2019, Bruin-Carminati-Kalle-2017} and the references therein. 
Recently, Komornik, Steiner and Zou \cite{Komornik-Steiner-Zou-2024} studied the unique expansions in double bases. It turns out that the intrinsic symbolic setting is the same as $\Sigma_{\a, \b}$ (see \cite{Hu-AlcarzBarrera-Zou-2025} and the references {therein} for more details).

Given $m\in\N$ and $\b=b_1b_2\ldots\in\set{0,1}^\N$, we define the critical value
\begin{equation}\label{eq:critical-value-theta-m}
\theta_m(\b):=\sup\set{\a\in\set{0,1}^\N: \Sigma_{\a,\b} \textrm{ contains a periodic sequence of smallest period }m},
\end{equation}
where   $\set{0,1}^\N$ is equipped with the product topology and the lexicographical ordering. {We point out that if $\theta_m(\b)=0^\f$, then for any $\a\in\set{0,1}^\N$ the set $\Sigma_{\a,\b}$ contains no periodic sequence of smallest period $m$. This means the supremum  in (\ref{eq:critical-value-theta-m}) may not be  attainable.}   {Clearly,} $\Sigma_{\a,\b}=\set{0^\f}$ if $\b$ begins with digit $0$. To avoid the trivial case we always assume that $\b$ begins with digit $1$. 

For each admissible $m$-chain $(m,k_1,\ldots,k_j)$ let
\[
\b_{m,k_1,\ldots,k_j}:=\L(\w_{k_1/m_1}\bullet\w_{k_2/m_2}\bullet\cdots\bullet\w_{k_j/m_j})^\f\in\set{0,1}^\N,
\]
where $m_1=m$ and $m_i=\gcd(m,k_1,\ldots,k_{i-1})$ for $2\le i\le j$. Then these sequences $\b_{m,k_1,\ldots,k_j}$ with $(m,k_1,\ldots,k_j)\in\A_m$ form a $m$-partition of $\set{0,1}^\N$:
\[
\set{0,1}^\N={[0^\f, \b_{m,1})} \cup{[\b_{m,m-1}, 1^\f]}\cup\bigcup_{(m,k_1,\ldots,k_j)\in\A_m\setminus\set{(m,m-1)}}\mathbf I_{m,k_1,\ldots,k_j}
\]
with the union pairwise disjoint, where   $\mathbf I_{m,k_1,\ldots,k_j}={[\b_{m,k_1,\ldots,k_j},\b_{m,k_1',\ldots,k_l'})}$  is a symbolic $m$-partition interval consisting  of all sequences $\mathbf z\in\set{0,1}^\N$ satisfying {$\b_{m,k_1,\ldots,k_j}\lle \mathbf z\prec \b_{m,k_1',\ldots,k_l'}$} in the lexicographical order. {Note by (\ref{eq:Sigma-ab}) that $\Sigma_{\a,\b}$ with $\a,\b\in\Sigma_{\a,\b}$ are all closed in $\set{0,1}^\N$ under the product topology, however, the symbolic survivor sets $\K_\beta(t)$ with $\beta\in(1,2]$ and $t\in[0,1)$ defined in (\ref{eq:symbolic-K}) are not necessarily closed. This difference causes that each symbolic $m$-partition interval $\mathbf I_{m,k_1,\ldots,k_j}$ defined here is left closed and right open, while each $m$-partition interval $I_{m,k_1,\ldots,k_j}$ defined in (\ref{eq:partition-intervals}) is left open and right closed.}


In terms of Theorem \ref{main-3:critical-value-details}, we give a complete description of the function $\b\mapsto\theta_m(\b)$.
\begin{theorem}
  \label{main:application-Lorenz-map}
 For any $m\in\N_{\ge 2}$, the critical value function $\theta_m$ has $\psi(m)$ discontinuity points $\b_{m,k_1,\ldots,k_j}$ with $(m,k_1,\ldots,k_j)\in\A_m$, and 
 \[
 \b_{m,k_1,\ldots,k_j}\prec\b_{m,k_1',\ldots,k_l'}\quad\Longleftrightarrow\quad (m,k_1,\ldots,k_j)\prec (m,k_1',\ldots,k_l').
 \]
 Furthermore, the function $\theta_m$ is determined as follows. 
 \begin{enumerate}[{\rm(i)}]
   \item If {$\b\prec \b_{m,1}$}, then $\theta_m(\b)=0^\f$;
   \item If {$\b\succcurlyeq \b_{m,m-1}$}, then $\theta_m(\b)=(\w_{(m-1)/{m}})^\f$;
   
   \item If $\b\in\mathbf I_{m,k_1,\ldots,k_j}$ for some {other symbolic}  $m$-partition interval  $\mathbf I_{m,k_1,\ldots,k_j}$,   then 
   \[\theta_m(\b)=(\w_{k_1/m_1}\bullet\cdots\bullet\w_{k_j/m_j})^\f, \]
   where $m_1=m$ and $m_i=\gcd(m,k_1,\ldots, k_{i-1})$ for $2\le i\le j$.
 \end{enumerate}
\end{theorem}
\begin{remark}
 Instead of looking at the critical value function $\theta_m(\b)$, one can also consider the critical value function
\[
\tilde\theta_m(\a):=\inf\set{\b\in\set{0,1}^\N: \Sigma_{\a,\b} \textrm{ contains a periodic sequence of smallest period }m}.
\]
 Observe that $\overline{\Sigma_{\a,\b}}=\Sigma_{\overline{\b},\overline{\a}}$, where $\overline{\c}=(1-c_1)(1-c_2)\ldots$ for $\c=c_1c_2\ldots\in\set{0,1}^\N$, and $\overline{\Sigma_{\a,\b}}=\set{\overline{\c}: \c\in\Sigma_{\a,\b}}$. Then $\Sigma_{\a,\b}$ contains a periodic sequence of smallest period $m$ if, and only if, $\Sigma_{\overline{\b}, \overline{\a}}$ contains a periodic sequence of smallest period $m$. This implies that 
 \[
 \tilde\theta_m(\a)=\overline{\theta_m(\overline{\a})}\quad\forall\a\in\set{0,1}^\N.
 \]
So, {the results in Theorem \ref{main:application-Lorenz-map} on $\theta_m$ can be applied to $\tilde\theta_m$.} 
 \end{remark}
 
 {At the end of this section, we point out that our result   Theorem \ref{main-2:L-P-mk} describes for any $1\le k<m$ the largest Lyndon word of length $m$ with $k$ ones, which may be used to study twist periodic orbits for expansive Lorenz maps. Note that for any rational number $\frac{p}{q}\in\mathbb Q\cap(0,1)$ with $\gcd(p,q)=1$ the Farey word $\w_{p/q}$ symbolically encodes a \emph{twist periodic orbit} (TPO) of minimal period $q$ and rotation number $p/q$. Furthermore, such orbits play a crucial role in the study of the periodic structure of expansive Lorenz maps (see, e.g., \cite{Geller-Misiurewicz-2018, BNS-2025}). In particular, if the map possesses TPOs with rotation numbers $a/b<p/q$ that are Farey neighbors, then suitable concatenations of the words $\w_{a/b}$ and $\w_{p/q}$ yield the itineraries of other periodic orbits of the map. A notable special case of TPOs are the primary $n(k)$-cycles (cf.~\cite{Cholewa-Oprocha-2024, Oprocha-Potorski-Raith-2019, BNS-2025}). Whenever such a cycle exists, the rotation interval degenerates to singleton (see \cite[Theorem 5.1]{Cholewa-Oprocha-2024}). The existence of such a cycle is also closely related to renormalizations and topological mixing and transitivity properties.}
 
The rest of the paper is organized as follows. In the next section we give useful properties of $\beta$-expansions, Lyndon words, Perron words and Farey words. In particular, we describe {the properties of the substitution operator $\bullet$ defined as in (\ref{eq:bullet-operator})}. As a by-product we {completely} characterize the extremal words $\max\cL_{m,k}$ and $\min\P_{m,k}$ (Theorem \ref{main-2:L-P-mk}). This will be proved in Section \ref{sec:extremal words-coprime} for $\gcd(m,k)=1$, and in Section \ref{sec:extreme Lyndon-Perron-Notprime} for $\gcd(m,k)>1$. The main result (Theorem \ref{main-3:critical-value-details}) on the critical value $\tau_m$ will be proved in Section \ref{sec:critical-value-tau}, {where we split the non-extremal $m$-partition intervals into four types according to the admissible $m$-chains}. Finally, in Section \ref{sec:final-remarks} we give an algorithm to determine $\tau_m(\beta)$ for any given $m\in\N$ and $\beta\in(1,2]$. We also make some other remarks on $\tau_m$.

\section{Preliminaries}\label{sec:preliminaries}
In this section we will   state some properties of $\beta$-expansions, Lyndon words, Perron words and Farey words.   
  First we introduce some terminology from symbolic dynamics (cf.~\cite{Lind_Marcus_1995}). 
Let  $\set{0,1}^*$ be the set of all finite words with each digit from the alphabet $\set{0,1}$. For a word $\w\in\set{0,1}^*$ we denote its length by $|\w|$.  {In particular, for the empty word $\ep\in\set{0,1}^*$  we have    $|\ep|=0$.}  For a digit $d\in\set{0,1}$ we denote by $|\w|_d$   the number of digit $d$ occurring in $\w$. For example, if $\w=01001$, then $|\w|=5, |\w|_0=3$ and $|\w|_1=2$. For a word $\w=w_1\ldots w_n\in\set{0,1}^*$ we denote its \emph{conjugate} by $\overline{\w}=(1-w_1)(1-w_2)\ldots (1-w_n)$. If $w_n=0$, then we write $\w^+:=w_1\ldots w_{n-1}(w_n+1)$; and if $w_n=1$, then we write $\w^-:=w_1\ldots w_{n-1}(w_n-1)$. 
  For a word $\w\in\set{0,1}^*$ let $\w^\f=\w\w\ldots\in\set{0,1}^\N$ be the periodic sequence obtained by concatenating $\w$ with itself infinitely many times, where $\set{0,1}^\N$ is the set of all infinite sequences over the alphabet $\set{0,1}$. Let $\si$ be the left shift map on $\set{0,1}^\N$ defined by $\si(w_1w_2\ldots)=w_2w_3\ldots$. Then $\si^{|\w|}(\w^\f)=\w^\f$ for any $\w\in\set{0,1}^*$.  
 
Throughout the paper we use  \emph{lexicographical ordering} between sequences and words. For two infinite sequences $(c_i)=c_1c_2\ldots, (d_i)=d_1d_2\ldots\in\set{0,1}^\N$ we write $(c_i)\prec (d_i)$ or $(d_i)\succ (c_i)$ if   {there exists $n\in\N$ such that $c_1\ldots c_{n-1}=d_1\ldots d_{n-1}$ and $c_{n}<d_{n}$.} Similarly, we write $(c_i)\lle(d_i)$ or $(d_i)\lge (c_i)$ if $(c_i)\prec (d_i)$ or $(c_i)=(d_i)$. {Moreover, for two words $\u=u_1\ldots u_j, \v=v_1\ldots v_\ell\in\set{0,1}^*$ we write $\u\prec \v$ { if $\u^\f\prec\v^\f$. }
 In this paper we will also {compare}  two vectors  $(m_1,m_2,\ldots, m_j)\in\R^j$ and $(n_1,n_2,\ldots, n_l)\in\R^l$. Similar to the lexicographical ordering defined for words in $\set{0,1}^*$, we say $(m_1,m_2,\ldots, m_j)\prec (n_1,n_2,\ldots, n_l)$ if the words satisfy $m_1m_2\ldots m_j \prec n_1 n_2 \ldots  n_l$.
 
 In contrast  with the left shift map $\si$   on $\set{0,1}^\N$, we define the cyclic permutation $\si_c$ on $\set{0,1}^*$. For a word $\w=w_1w_2\ldots w_n\in\set{0,1}^*$  {let $\si_c(\w)=w_2\ldots w_nw_1$ be its \emph{cyclic permutation}.} Then the words $\si_c^k(\w), k=1,2,\ldots$ have the same length, and {$\si_c^{|\w|}(\w)=\w$}.  Among these words, let $\L(\w)$ and $\S(\w)$ be the lexicographically largest and lexicographically smallest cyclic permutations of $\w$, respectively. In other words, 
 \[\L(\w)=\max\set{\si_c^k(\w): k=0,1,\ldots,n-1},\quad \S(\w)=\min\set{\si_c^k(\w): k=0,1,\ldots,n-1}.\] 
 For instance, if $\w=010011$, then $\L(\w)=110100$ and $\S(\w)=001101$.

\subsection{$\beta$-expansions}\label{subsec:beta-expansions}  Given $\beta\in(1,2]$ and a sequence $(d_i)\in\set{0,1}^\N$, recall from (\ref{eq:beta-expansion}) that 
\[
 ((d_i))_\beta=\sum_{i=1}^{\f}\frac{d_i}{\beta^i}\in\left[0,\frac{1}{\beta-1}\right],
\]
and the infinite sequence $(d_i)$ is called a \emph{$\beta$-expansion} of $((d_i))_\beta$. If $\beta=2$, we know that each $x\in[0,1]$ has a unique $\beta$-expansion except for countably   many points having precisely two $\beta$-expansions. However, for $\beta\in(1,2)$ Sidorov \cite{Sidorov_2003} showed that   Lebesgue almost every   $x\in [0,\frac{1}{\beta-1}]$ has a continuum of $\beta$-expansions.   For $\beta\in(1,2]$ and $x\in[0,\frac{1}{\beta-1}]$, let $b(x,\beta)=b_1(x,\beta)b_2(x,\beta)\ldots\in\set{0,1}^\N$ be the \emph{greedy} $\beta$-expansion of $x$, which is the lexicographically largest $\beta$-expansion of $x$ (cf.~\cite{Renyi_1957, Parry_1960}).

 To describe the greedy $\beta$-expansions, the quasi-greedy $\beta$-expansion of $1$ plays an important role. Let \[\de(\beta)=\de_1(\beta)\de_2(\beta)\ldots\in\set{0,1}^\N\]  be the \emph{quasi-greedy} $\beta$-expansion of $1$, that is the lexicographically largest $\beta$-expansion of $1$ not ending with a string of zeros. 
The following characterization  on $\de(\beta)$ is well-known (cf.~\cite{Baiocchi_Komornik_2007}).
\begin{lemma}
  \label{lem:delta-beta}
  The map $\beta\mapsto\de(\beta)$ is a strictly  increasing bijection from $(1,2]$ onto the set of sequences $(a_i)\in\set{0,1}^\N$ not ending with $0^\f$ and satisfying
  \[
   a_{n+1}a_{n+2}\ldots\lle a_1a_2\ldots\quad\forall n\ge 0.
  \]
\end{lemma}
In terms of Lemma \ref{lem:delta-beta}, each base $\beta\in(1,2]$ is uniquely determined by its quasi-greedy expansion $\de(\beta)$. For example, $\de(\frac{1+\sqrt{5}}{2})=(10)^\f$ and $\de(2)=1^\f$. 
Based on $\de(\beta)$,  the following characterization on greedy $\beta$-expansions is due to Parry \cite{Parry_1960}.
\begin{lemma}
  \label{lem:greedy-expansion}
  Given $\beta\in(1,2]$, the map $x\mapsto b(x,\beta)$ is a strictly increasing bijection from $[0,1)$ to the set
  \[
  \Sigma_\beta:=\set{(b_i)\in\set{0,1}^\N: \si^n((b_i))\prec \de(\beta)~\forall n\ge 0}.
  \]
\end{lemma} 
Note that {for $x=\frac{1}{\beta-1}$ it has a unique $\beta$-expansion $1^\f$;} and for $x\in{[1,\frac{1}{\beta-1})}$ its greedy $\beta$-expansion $b(x,\beta)$ begins with {$1^k=\overbrace{1\ldots 1}^k$} for some $k\in\N$ and then ends with a sequence in $\Sigma_\beta$. So, for any $x\in{[0,\frac{1}{\beta-1})}$ its greedy $\beta$-expansion $b(x,\beta)$ eventually ends in $\Sigma_\beta$. 

Recall that the survivor set $K_\beta(t)$ consists of all $x\in[0,1)$ {whose} orbit $\{T_\beta^n(x)\}_{n=0}^\f$ never hits the hole $(0,t)$. Note that $T_\beta(0)=0$ is the fixed point of the expanding map $T_\beta: [0,1)\to[0,1); ~x\mapsto \beta x\pmod 1$. Therefore, to study the critical value $\tau_m(\beta)$ with $m\in\N_{\ge 2}$ it suffices to consider  the subset 
\[
\widetilde{K}_\beta(t):=\set{x\in[0,1): T_\beta^n(x)\ge t\quad\forall n\ge 0}.
\]
Note  that the  dynamical system $([0,1), T_\beta)$ is conjugate to the symbolic dynamical system $(\Sigma_\beta, \si)$. Then we can reduce our study of $\tau_m(\beta)$ to the symbolic analogue    of $\widetilde{K}_\beta(t)$ defined as in (\ref{eq:symbolic-K}):
\begin{equation*} 
\K_\beta(t)=\set{(d_i)\in\set{0,1}^\N: b(t,\beta)\lle\si^n((d_i))\prec \de(\beta)\quad\forall n\ge 0}.
\end{equation*}
In other words, for any $m\in\N_{\ge 2}$ the survivor set $K_\beta(t)$ contains a periodic orbit  of smallest period $m$ if and only if $\K_\beta(t)$ contains a periodic sequence of smallest period $m$. So,
\[
\tau_m(\beta)=\sup\set{t\in[0,1): \K_\beta(t)\textrm{ contains a periodic sequence of smallest period }m}.
\]
 {Here we emphasize that the supremum can be replaced by the maximum.}
\subsection{Lyndon words and Perron words}\label{subsec:Lyndon-Perron-words}
A word $\w=w_1\ldots w_n\in\set{0,1}^*$ is called \emph{periodic} if there exists a positive integer $k<n$ such that $k|n$ and $\w=(w_1\ldots w_k)^{n/k}$. 
A word $\w$ is called a \emph{Lyndon word} if it is not periodic and it is the lexicographically smallest among all of its cyclic permutations, i.e., $\w=\S(\w)$. {The following equivalent characterization on Lyndon words can be found in \cite{Allouche_Shallit_2003}.}
\begin{lemma}
  \label{lem:Lyndon-words}
  $\w=w_1\ldots w_m\in\set{0,1}^*$ is a Lyndon word if and only if 
  \[
  w_{i+1}\ldots w_m\succ w_1\ldots w_{m-i}\quad\forall 1\le i<m.
  \]
\end{lemma}
Clearly, digits $0$ and $1$ are trivial Lyndon words. 
Let $\mathcal L^*$ be the set of all Lyndon words of length at least two. Then by Lemma \ref{lem:Lyndon-words} it follows that any word from $\mathcal L^*$ begins with digit $0$ and ends with digit $1$. Recall from (\ref{eq:bullet-operator}) the substitution operator $\bullet$ on $\mathcal L^*$. Given a Lyndon word $\s\in\mathcal L^*$ and a word $\r=r_1\ldots r_n\in\set{0,1}^*$, we have 
\begin{equation*}
\s\bullet\r=\w_1\w_2\ldots \w_n,
\end{equation*}
where  $\w_1=\s^-$ if $r_1=0$, and $\w_1=\L(\s)^+$ if $r_1=1$; and for $i=1,2,\ldots,n-1$, we have
\[
\w_{i+1}=\left\{
\begin{array}{lll}
  \L(\s) & \textrm{if} & r_ir_{i+1}=00, \\
  \L(\s)^+ & \textrm{if} & r_ir_{i+1}=01,\\
  \s^-&\textrm{if}& r_ir_{i+1}=10,\\
  \s&\textrm{if}& r_ir_{i+1}=11.
\end{array}\right.
\]
 Note that this definition  can   be easily extended to $\s\bullet(d_i)$ for any infinite sequence $(d_i)\in\set{0,1}^\N$. For example,
$\s\bullet(011)^\f=(\s^-\L(\s)^+\s)^\f$.

Observe  that the operator $\bullet$ is defined via a one step Markov chain. Then we can 
  define    $\s\bullet\r$    via a directed graph $G_\s=(V_\s, E_\s)$  {as plotted in Figure \ref{fig:directed-graph}}. Here each vertex in {$V_\s$} has a label in $\set{\s,\s^-, \L(\s), \L(\s)^+}$ except two special vertices, one with the label `Start-0' and the other with the label `Start-1'. Furthermore, each edge in {$E_\s$} has a label in $\set{0,1}$. {Now we explain how to construct  the word $\s\bullet\r=\w_1\w_2\ldots \w_n$ from the directed graph $G_\s=(V_\s, E_\s)$ (cf.~\cite{Allaart-Kong-2023}).} If $r_1=0$, then we begin with the vertex named `Start-0' and the directed edge terminated at the vertex named $\s^-$. So we set $\w_1=\s^-$. Similarly, if $r_1=1$, then we begin with the vertex named `Start-1' and the directed edge ends with the vertex named $\L(\s)^+$. In this case we {write} $\w_1=\L(\s)^+$. Next we follow the unique path $\gamma$ in $G_\s$ labeled by  $\r=r_1r_2\ldots r_n$, and then we write down the word $\s\bullet\r=\w_1\w_2\ldots \w_n$ such that  each $\w_i$ is the label of the terminal vertex of the directed edge in $\gamma$ with the label $r_i$. For example, $\s\bullet01101=\s^-\L(\s)^+\s\s^-\L(\s)^+$ and $\s\bullet 10110=\L(\s)^+\s^-\L(\s)^+\s\s^-$.
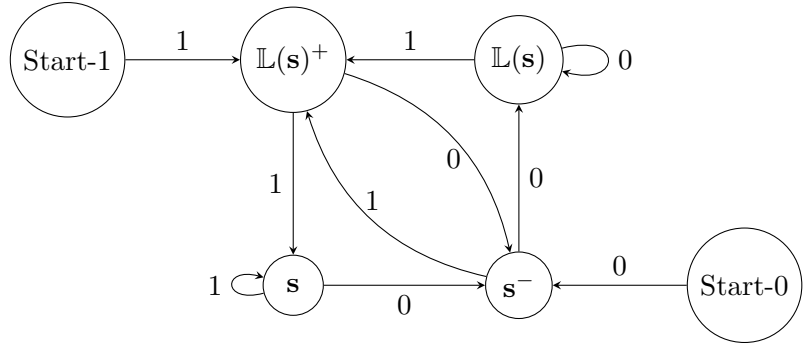
\begin{figure} [h!]
\begin{center}
  \begin{tikzpicture}[>=stealth, node distance=2cm]
    \node[circle, draw, minimum size=1cm] (Start1) at (0, 0) {Start-1};
    \node[circle, draw, minimum size=0.8cm] (aPlus) at (3, 0) {$\L(\s)^+$};
    \node[circle, draw, minimum size=0.8cm] (a) at (6, 0) {$\L(\s)$};
    \node[circle, draw, minimum size=0.8cm] (s) at (3, -3) {$\mathbf{s}$};
    \node[circle, draw, minimum size=0.8cm] (sMinus) at (6, -3) {$\mathbf{s}^-$};
    \node[circle, draw, minimum size=1cm] (Start0) at (9, -3) {Start-0};

    \draw[->] (Start1) -- (aPlus) node[midway, above] {1};
    \draw[->] (a) -- (aPlus) node[midway, above] {1};
    \draw[->] (a) edge[loop right, looseness=7] node {0} (a);
    \draw[->] (sMinus) -- (a) node[midway, right] {0};
    \draw[->, bend left=30] (sMinus) to node[midway, above] {1} (aPlus);
    \draw[->, bend left=30] (aPlus) to node[midway, below] {0} (sMinus);
    \draw[->] (s) -- (sMinus) node[midway, below] {0};
    \draw[->] (s) edge[loop left, looseness=7] node {1} (s);
    \draw[->] (aPlus) -- (s) node[midway, left] {1};
    \draw[->] (Start0) -- (sMinus) node[midway, above] {0};
  \end{tikzpicture}
\caption{The directed graph   $G_\s=(V_\s, E_\s)$.}\label{fig:directed-graph}
\end{center}
\end{figure}

By the definition of the operator $\bullet$ we see that for any $\s\in\mathcal L^*$ and any $\r\in\set{0,1}^*$ we have $|\s\bullet\r|=|\s|\cdot|\r|$. The following properties for the substitution operator  $\bullet$ was obtained in \cite{Allaart-Kong-2023}.
\begin{lemma}
  \label{lem:property-bullet}
  $(\mathcal L^*, \bullet)$ forms a non-Abelian semi-group.  
  \begin{enumerate}[{\rm(i)}]
    \item  For any $\s,\r\in\mathcal L^*$ we have $\L(\s\bullet\r)=\s\bullet\L(\r)$.
    \item  {For any $\s\in\mathcal L^*$ and $\c, \d\in\set{0,1}^*$ we have 
   $
    \s\bullet\c\prec\s\bullet\d~\Longleftrightarrow ~\c\prec \d.
   $}
   {\item For any $\s\in\mathcal L^*$ and $\c=c_1\ldots c_n\in\set{0,1}^*$ with $c_1\neq c_n$ we have $(\s\bullet \c)^\f=\s\bullet \c^\f$.}
  \end{enumerate}
\end{lemma}
{For some further properties of a general substitution or automaton we refer to \cite{Dekking-1977, Dekking-Keane-1978} and the books \cite{Allouche_Shallit_2003, Lothaire-2002}.}

Note by Lemma \ref{lem:Lyndon-words} that each Lyndon word $\w$ is the lexicographically smallest word among all of its cyclic permutations, i.e., $\w=\S(\w)$. On the other hand, we define
\[\P^*:=\set{\L(\w): \w\in\mathcal L^*}.\]
 Then each word in $\P^*$ has length at least two,  is not periodic and is the lexicographically largest among all of its cyclic permutations. In contrast with Lemma \ref{lem:Lyndon-words}, we have the following characterization {on $\P^*$}. 
\begin{lemma}
  \label{lem:perron-words}
  $\a=a_1\ldots a_n\in\P^*$ if and only if 
  \[
  a_{i+1}\ldots a_n\prec a_1\ldots a_{n-i}\quad \forall 1\le i<n.
  \]
\end{lemma}  
By Lemmas \ref{lem:perron-words} and \ref{lem:delta-beta} it follows that each word $\a=a_1\ldots a_n\in\P^*$ determines a unique base $\beta_\a\in(1,2]$ such that $\de(\beta_\a)=\a^\f$, i.e., $\beta_\a\in(1,2]$ satisfies the equation
\[
1=\frac{a_1}{\beta}+\frac{a_2}{\beta^2}+\cdots+\frac{a_{n-1}}{\beta^{n-1}}+\frac{a_n+1}{\beta^n}.
\] Recall that a \emph{Perron number} is a real algebraic integer greater than one and all of its Galois conjugates are smaller than itself in absolute value. It is known that for any $\a\in\P^*$ the base $\beta_\a$ is a Perron number (cf.~\cite{Blanchard-1989}). So, we call each word $\a\in\P^*$   a \emph{Perron word}. {Note that words in $\P^*$ are also called inverse Lyndon or anti-Lyndon (see, e.g., \cite{BDZZ-2018}).}
In Theorem \ref{main-2:L-P-mk} we will determine extremal Lyndon words and extremal Perron words. These extremal Lyndon words are useful to describe the critical value $\tau_m$, and the extremal Perron words are pivotal to describe the discontinuity points of $\tau_m$.

\subsection{Farey words}\label{subsec:Farey-words}
  Farey words have attracted much attention in the literature due to their intimate connection with rational rotations on the circle and continued fractions (\cite{Lothaire-2002}). It also has many applications in clock-making, numerical approximation, Ford circles, and even the Riemann hypothesis (cf.~\cite{Edwards-1974, Short-2011}). In the following we adopt the definition from \cite[Section 2]{Carminati-Isola-Tiozzo-2018}.

First we define a sequence of ordered sets $F_n, n=0,1,2,\ldots,$ recursively. Set $F_0=(0,1)$; and for $n\ge 0$ the ordered set $F_{n+1}=(v_1,\ldots, v_{2^{n+1}+1})$ is obtained from $F_n=(w_1, \ldots, w_{2^n+1})$ by
\[
\left\{\begin{array}{lll}
         v_{2i-1}=w_i & \textrm{for} & 1\le i\le 2^n+1, \\
         v_{2i}=w_iw_{i+1} & \textrm{for} & 1\le i\le 2^n.
       \end{array}\right.
\]
For example, 
\[F_1=(0,01,1), \quad F_2=(0,001,01, 011, 1),\quad  F_3=(0,0001,001,00101, 01, 01011, 011, 0111, 1),\]
 and so on.  A word $\w\in\set{0,1}^*$ is called a \emph{Farey word} if $\w\in F_n$ for some $n\ge 0$. Let $\mathcal F^*$ be the set of all Farey words of length at least two, i.e., 
 \[
 \mathcal F^*=\bigcup_{n=0}^\f F_n\setminus\set{0,1}.
 \]
It is known that each Farey word is a Lyndon word, i.e, $\mathcal F^*\subset \mathcal L^*$.
 
 The Farey words can also be obtained via the following substitutions: 
\begin{equation}\label{eq:U0-U1}
U_0:\left\{\begin{array}{ccc}
      0 & \mapsto & 0 \\
      1 & \mapsto & 01 
    \end{array}\right.
    \quad\textrm{and}\quad
 U_1:\left\{\begin{array}{ccc}
      0 & \mapsto & 01 \\
      1 & \mapsto & 1 
    \end{array}.  \right. 
\end{equation}
For $d_1\ldots d_n\in\set{0,1}^*$ let $U_{d_1\ldots d_n}:=U_{d_1}\circ U_{d_2}\circ\cdots\circ U_{d_n}$. In particular, for the empty word $\ep$ we set $U_\ep=I_d$ as the identity map. 
  The following result on $\mathcal F^*$ was {established in \cite[Propositions 2.5 and 2.9]{Carminati-Isola-Tiozzo-2018}.}
\begin{lemma}
  \label{lem:Farey-words-characterization}
  $\w\in\mathcal F^*$ if and only if 
 $\w=U_{d_1\ldots d_n}(01)$ for some $d_1\ldots d_n\in \set{0,1}^*$.
  
  Furthermore, for any $\w=w_1\ldots w_m\in\mathcal F^*$ we have
  \begin{enumerate}[{\rm(i)}]
    \item $\S(\w)=\w$ and $\L(\w)=w_mw_{m-1}\ldots w_1$;
    \item $\w^-$ is a palindrome; that is $w_1\ldots w_{m-1}(w_m-1)=(w_m-1)w_{m-1}w_{m-2}\ldots w_1$.
  \end{enumerate} 
\end{lemma}

The set $\mathcal F^*$ can be bijectively mapped to $\mathbb Q\cap(0,1)$ by the map
 \begin{equation}\label{eq:bijection-map-Farey-rational}
 \xi: \mathcal F^*\to\mathbb Q\cap(0,1);\quad \w\mapsto\frac{|\w|_1}{|\w|},
 \end{equation}
 where $|\w|_1$ is the number of digit one in $\w$ and $|\w|$ is the length of $\w$. For example, $\xi(00101)=\frac{2}{5}$. Conversely, for  a rational number $\frac{p}{q}\in(0,1)$ with $\gcd(p,q)=1$ there is a unique   $\w=w_1\ldots w_q\in\mathcal F^*$ such that $\xi(\w)=\frac{p}{q}$, and we denote this Farey word $\w$ by $\w_{p/q}$. Indeed, $\w_{p/q}$ can be constructed dynamically.   {Let $R_{p/q}: \R/\mathbb Z\to \R/\mathbb Z; ~ x\mapsto  x+\frac{p}{q}\pmod 1$. Then}
   $\w_{p/q}=w_1w_2\ldots w_q\in\set{0,1}^q$ satisfies 
 \begin{equation}\label{eq:farey-word}
 w_k=\left\{\begin{array}
   {lll}
   0&\textrm{if}& 0\notin (R_{p/q}^{k-1}(0), R_{p/q}^k(0) ],\\
   1&\textrm{if}& 0\in (R_{p/q}^{k-1}(0), R_{p/q}^k(0) ].
 \end{array}\right.
 \end{equation}
 For example, $\w_{3/8}=00100101$, since 
 \[
 0\overset{R_{3/8}}{\underset{0}{\longrightarrow}}\frac{3}{8}\overset{R_{3/8}}{\underset{0}{\longrightarrow}}\frac{6}{8}
 \overset{R_{3/8}}{\underset{1}{\longrightarrow}}\frac{1}{8}
 \overset{R_{3/8}}{\underset{0}{\longrightarrow}}\frac{4}{8}\overset{R_{3/8}}{\underset{0}{\longrightarrow}}\frac{7}{8}
 \overset{R_{3/8}}{\underset{1}{\longrightarrow}}\frac{2}{8}
 \overset{R_{3/8}}{\underset{0}{\longrightarrow}}\frac{5}{8}\overset{R_{3/8}}{\underset{1}{\longrightarrow}}0.
 \]
 
 The following result on {the map} $\xi$ in (\ref{eq:bijection-map-Farey-rational}) was established in \cite[Proposition 2.3]{Carminati-Isola-Tiozzo-2018}.
\begin{lemma}
  \label{lem:bijection-Fareyword-Rational}
  The map $\xi: \mathcal F^*\to\mathbb Q\cap(0,1); ~\w\mapsto\frac{|\w|_1}{|\w|}$ is bijective and strictly increasing with respect to the lexicographical ordering in $\mathcal F^*$.  
\end{lemma}

Recall that $\mathcal L^*$ consists of all Lyndon words of length at least two. The following lemma shows that the substitutions  $U_0$ and $U_1$ in (\ref{eq:U0-U1}) associate   with the operator $\bullet$ on $\mathcal L^*$.
\begin{lemma}
  \label{lem:commute-U-bullet}
 Let $\s, \r\in\mathcal L^*$. Then for any $d\in\set{0,1}$ we have $U_d(\s\bullet\r)=U_d(\s)\bullet\r$. 
\end{lemma}
\begin{proof}
  Since the proof for $U_1$ is similar, we only prove the result for $U_0$. We will prove 
  \begin{equation}\label{eq:oct29-1}
    U_0(\s\bullet\r)=U_0(\s)\bullet\r.
  \end{equation}
 This is based on the following three equalities: 
 \begin{enumerate}[{\rm(i)}]
   \item $U_0(\s^-)0=U_0(\s)^-$.
   \item $0\L(U_0(\s))=U_0(\L(\s))0$.
   \item $U_0(\L(\s)^+)=0\L(U_0(\s))^+$.
 \end{enumerate}
 
Let $\s=s_1s_2\ldots s_m\in\cL^*$ and write $\L(\s)=a_1a_2\ldots a_m$. Then $s_1=a_m=0$ and $s_m=a_1=1$. 
Note that $\s^-=s_1\ldots s_{m-1}(s_m-1)$. Then (i) follows by (\ref{eq:U0-U1})   that
 \begin{align*}
   U_0(\s)^-&=U_0(s_1)U_0(s_2)\ldots U_0(s_{m-1})00 \\
            &=U_0(s_1)U_0(s_2)\ldots U_0(s_{m-1})U_0(s_m^-)0=U_0(\s^-)0.
  \end{align*}

{Next we consider (ii). Note by (\ref{eq:U0-U1}) that  \begin{align*}
   U_0(\L(\s))0&=U_0(a_1)U_0(a_2)\ldots U_0(a_m)0 =01U_0(a_2)\ldots  U_0(a_m)0.
 \end{align*}
Let 
\begin{equation}\label{eq:gamma-sequence}
\ga_1\ga_2\ldots\ga_n:=1U_0(a_2)\ldots  U_0(a_m)0=\si_c(01 U_0(a_2)\ldots U_0(a_m))=\si_c(U_0(\L(\s))),
\end{equation} where $\si_c$ is the cyclic permutation. Then $\ga_1\ldots \ga_n=\si_c^j(U_0(\s))$ for some $j$. By Lemma \ref{lem:perron-words}, to verify (ii) it suffices to prove that   \begin{equation}\label{eq:nov4-1}
  \ga_{i+1}\ldots\ga_n\prec\ga_1\ldots\ga_{n-i} \quad \forall 1\le i<n.
\end{equation} 
Take $i\in\set{1,2,\ldots,n-1}$.
  If $\ga_{i+1}=0$, then (\ref{eq:nov4-1}) is obvious since $\ga_1=1$. Now suppose $\ga_{i+1}=1$.  Since the block $11$ is forbidden in $U_0(\s)$, so is in $\ga_1\ldots \ga_n$. Then $\ga_i=0$, and thus by (\ref{eq:gamma-sequence}) we obtain that
 \begin{equation}\label{eq:Nov3-1}
   \ga_i\ga_{i+1}\ldots\ga_n=U_0(a_{k+1}\ldots a_m)0=01\, U_0(a_{k+2}\ldots a_m)0
 \end{equation}
 for some $1\le k<m$. Since $a_1\ldots a_m=\L(\s)$, by Lemma \ref{lem:perron-words} we have $a_{k+1}\ldots a_m\prec a_1\ldots a_{m-k}$. By the monotonicity of $U_0$ and (\ref{eq:Nov3-1}) it follows that 
 \[
 {01 U_0(a_{k+2}\ldots a_m)0=U_0(a_{k+1}\ldots a_m)0\prec U_0(a_1\ldots a_{m-k})=01 U_0(a_2\ldots a_{m-k}),}
 \]
 where the last equality holds since $a_1=1$. Hence, by (\ref{eq:Nov3-1})  we conclude  that 
 \[
 {\ga_{i+1}\ldots \ga_n=1 U_0(a_{k+2}\ldots a_m)0\prec 1 U_0(a_2\ldots a_{m-k}).}
 \]
This {together with   (\ref{eq:gamma-sequence})} proves (\ref{eq:nov4-1}), and thus establishes (ii).
 }

Finally we prove  (iii), which follows from (ii) that 
 \begin{align*} 
    U_0(\L(\s)^+)&=U_0(a_1\ldots a_{m-1}a_m^+)=U_0(a_1\ldots a_{m-1})01  \\ 
                &=U_0(a_1\ldots a_{m-1}a_m)1 =(U_0(\L(\s))0)^+  
                =(0\L(U_0(\s)))^+ =0\L(U_0(\s))^+.
 \end{align*}
  This proves the equalities in (i)--(iii).

 Now we are ready to prove (\ref{eq:oct29-1}).   Since our substitution operator $\bullet$   in (\ref{eq:bullet-operator}) is defined by a one-step Markov chain, it suffices to consider $\r=001011$ because this block contains all length two blocks $00, 01, 10, 11$. Applying the equalities in  (iii), (ii) and (i) successively,  we obtain that 
 \begin{align*}
  U_0(\s\bullet\r)&=U_0(\s^-\L(\s)\L(\s)^+\s^-\L(\s)^+\s) \\ 
                  &=U_0(\s^-) U_0(\L(\s)) U_0(\L(\s)^+)U_0(\s^-)U_0(\L(\s)^+)U_0(\s)  \\
                 &=U_0(\s^-)U_0(\L(\s))0\L(U_0(\s))^+ U_0(\s^-)0 \L(U_0(\s))^+ U_0(\s)\\
                 &=U_0(\s^-)0\L(U_0(\s))\L(U_0(\s))^+ U_0(\s^-)0 \L(U_0(\s))^+ U_0(\s)   \\
                 &=U_0(\s)^-\L(U_0(\s))\L(U_0(\s))^+ U_0(\s)^-\L(U_0(\s))^+ U_0(\s)  =U_0(\s)\bullet\r.
 \end{align*}
This proves (\ref{eq:oct29-1}),  completing the proof.
\end{proof}

\section{Extremal   Lyndon   and Perron words: coprime case}\label{sec:extremal words-coprime}

 Recall that $\cL^*$ consists of all Lyndon words of length at least two, and $\P^*$ consists of all Perron words of length at least two. 
Given $m\in\N_{\ge 2}$ and $k\in\set{1,2,\ldots, m-1}$, let $\cL_{m,k}$ be the set of all length $m$ Lyndon words with precisely $k$ ones, i.e., $\cL_{m,k}=\set{\w\in\mathcal L^*: |\w|=m,~|\w|_1=k}$. Then 
\[
\cL^*=\bigcup_{m=2}^\f\bigcup_{k=1}^{m-1}\cL_{m,k},
\]
where the unions on the right are pairwise disjoint. Accordingly, let $\P_{m,k}$ be the set of all length $m$ Perron words with precisely $k$ ones. In other words, $\P_{m,k}=\set{\L(\w): \w\in\cL_{m,k}}$. Then 
\[
\P^*=\bigcup_{m=2}^\f\bigcup_{k=1}^{m-1}\P_{m,k}
\]
with the union pairwise disjoint. It is easy to verify that $0^{m-k}1^k$ is
 the lexicographically smallest word of $\cL_{m,k}$, and $1^k0^{m-k}$ is the lexicographically largest word of $\P_{m,k}$. 
 In other words,  
\[\min\cL_{m,k}=0^{m-k}1^k,\quad \max\P_{m,k}=1^k 0^{m-k}.\]
Then it is natural to determine  the extremal words $\max\cL_{m,k}$ and  $\min\P_{m,k}$. One might expect that both $\max\cL_{m,k}$ and $\min\P_{m,k}$ are balanced words. Recall from \cite{Lothaire-2002} (see also, \cite{Allouche_Shallit_2003}) that a word $\w\in\set{0,1}^*$ is called \emph{balanced} if for any {subwords} $\u, \v$ of $\w$ with equal length, we have $||\u|_1-|\v|_1|\le 1$. It turns out that if $\gcd(m,k)=1$ then both $\max\cL_{m,k}$ and $\min\P_{m,k}$ are balanced. However, if $\gcd(m,k)>1$ then they are not balanced. {We will describe these extremal words $\max\cL_{m,k}$ and $\min\P_{m,k}$ by using} Farey words. 
Recall from (\ref{eq:farey-word}) that for each rational number $\frac{p}{q}\in(0,1)$ with $\gcd(p,q)=1$ there is a unique Farey word $\w_{p/q}$ of length $q$ with $p$ ones. If $\gcd(p,q)=d>1$, we still {write} $\w_{p/q}=\w_{\frac{p/d}{q/d}}$.  

\begin{theorem}
  \label{main-2:L-P-mk}
  Let $m\in\N_{\ge 2}$ and $k\in\set{1,2,\ldots, m-1}$. 
  \begin{enumerate}[{\rm(i)}]
    \item If $\gcd(m,k)=1$, then 
    \[
    \max\cL_{m,k}=\w_{k/m},\quad \min\P_{m,k}=\L(\w_{k/m}).
    \]
    \item  If $\gcd(m,k)={d}\ge 2$, then 
    \begin{align*}
    \max\cL_{m,k}&=\w_{k/m}\bullet(01^{{d}-1}){=\w_{k/m}^-\L(\w_{k/m})^+(\w_{k/m})^{d-2}},\\
     \min\P_{m,k}&=\w_{k/m}\bullet(10^{{d}-1}){=\L(\w_{k/m})^+\w_{k/m}^-(\L(\w_{k/m}))^{d-2}}.
    \end{align*}
  \end{enumerate}
  
\end{theorem}
\begin{remark}
  \label{rem:Lyndon-Perron}
  Note that Farey words are balanced. So, by Theorem \ref{main-2:L-P-mk} it follows that if $\gcd(m,k)=1$ then both $\max\cL_{m,k}$ and $\min\P_{m,k}$ are   balanced. 
  However, if $\gcd(m,k)>1$ then neither $\max\cL_{m,k}$ nor $\min\P_{m,k}$ is balanced, since both $\max\cL_{m,k}$ and $\min\P_{m,k}$ contain the words $\w_{k/m}^-$ and $\L(\w_{k/m})^+$.
 \end{remark}

Let
\[
 \l_{m,k}:=\max\cL_{m,k}\quad\textrm{and}\quad \p_{m,k}:=\min\P_{m,k}.
\]
The following result shows that both $\l_{m,k}$ and $\p_{m,k}$ are strictly increasing in $k$.

\begin{lemma}
  \label{lem:monotonicity-lyndon-perron-m-k}
  Let $m\in\N_{\ge 2}$. Then for any $ 1\le k<m-1$,
  \[
 \l_{m,k}\prec \l_{m,k+1}\quad \textrm{and}\quad \p_{m,k}\prec \p_{m,k+1}.
  \]
\end{lemma} 
\begin{proof}
  Let  $\l_{m,k}=\ell_1\ldots \ell_m$ with $1\le k<m-1$. Then the number of digit zero in $\l_{m,k}$ is precisely $m-k(>1)$. Let $j$ be the largest index such that $\ell_j=0$. Then $\tilde \l=\ell_1\ldots \ell_{j-1}1 \ell_{j+1}\ldots \ell_m\in \cL_{m,k+1}$ by Lemma \ref{lem:Lyndon-words}. This implies 
  $\l_{m,k}\prec \tilde \l\lle \l_{m,k+1}$.
   
Similarly, let $\p_{m,k+1}=p_1p_2\ldots p_m$ with $1\le k<m-1$. Then the number of digit one in $\p_{m,k+1}$ is precisely $k+1(> 1)$. Let $t$ be the largest index such that $p_t=1$. Then $\tilde\p=p_1\ldots p_{t-1}0 p_{t+1}\ldots p_m\in \P_{m,k}$ by Lemma \ref{lem:perron-words}, and thus
$\p_{m,k+1}\succ\tilde\p\lge \p_{m,k}.$
\end{proof}

For a word $\w=w_1\ldots w_m\in\set{0,1}^*$, recall that  its conjugate is defined by $\overline{\w}=(1-w_1)\ldots (1-w_m)$. Similarly, for a subset $A\subset\set{0,1}^*$ we denote its \emph{conjugate} by $\overline{A}:=\set{\overline{\w}: \w\in A}$. 
\begin{lemma}
  \label{lem:relation-S-L}
  Let $m\in\N_{\ge 2}$. Then for any $1\le k<m$ we have 
  \[\overline{\l_{m,k}}=\p_{m,m-k}\quad\textrm{and}\quad\overline{\cL_{m,k}}=\P_{m, m-k}.\]  
\end{lemma}
\begin{proof}
  First we prove $\overline{\cL_{m,k}}=\P_{m,m-k}$. Take $s_1\ldots s_m\in \cL_{m,k}$. Then by Lemma \ref{lem:Lyndon-words} we have $s_{i+1}\ldots s_m\succ s_1\ldots s_{m-i}$ for all $1\le i<m$. This implies that $\overline{s_{i+1}\ldots s_m}\prec \overline{s_1\ldots s_{m-i}}$ for all $1\le i<m$. So, $\overline{s_1\ldots s_m}\in \P_{m,m-k}$ by Lemma \ref{lem:perron-words}, and thus $\overline{\cL_{m,k}}\subset \P_{m,m-k}$.
  
  On the other hand, take $a_1\ldots a_m\in \P_{m,m-k}$. Then by Lemma \ref{lem:perron-words} we have $a_{i+1}\ldots a_m\prec a_1\ldots a_{m-i}$ for all $1\le i<m$, which yields $\overline{a_{i+1}\ldots a_m}\succ \overline{a_1\ldots a_{m-i}}$. So, $\overline{a_1\ldots a_m}\in \cL_{m,k}$ by Lemma \ref{lem:Lyndon-words}, i.e., $a_1\ldots a_m\in\overline{\cL_{m,k}}$. This proves $\P_{m,m-k}= \overline{\cL_{m,k}}$.
  
  Next we prove $\overline{\l_{m,k}}=\p_{m,m-k}$.  Take $a_1\ldots a_m\in \P_{m,m-k}$. Then $\overline{a_1\ldots a_m}\in \overline{\P_{m,m-k}}=\cL_{m,k}$. Since $\l_{m,k}=\max\cL_{m,k}$,  it follows that $\overline{a_1\ldots a_m}\lle \l_{m,k}$, which implies $a_1\ldots a_m\lge \overline{\l_{m,k}}$. Note that $\overline{\l_{m,k}}\in\overline{\cL_{m,k}}=\P_{m,m-k}$. So, $\overline{\l_{m,k}}=\min \P_{m,m-k}=\p_{m,m-k}$, completing the proof.
\end{proof}
In this section we will determine the extremal words $\l_{m,k}=\max\cL_{m,k}$ and $\p_{m,k}=\min\P_{m,k}$ when $\gcd(m,k)=1$, and prove Theorem \ref{main-2:L-P-mk} (i). When $\gcd(m,k)>1$, we will prove Theorem \ref{main-2:L-P-mk} (ii) in the next section.
A word $\w\in\set{0,1}^*$ is called \emph{$1$-balanced} if the numbers  of consecutive zeros in $\w$ {are}  different up to $1$.  
   {Here we point out that $1$-balanced words are not necessarily balanced. For example, $0010010101$ is $1$-balanced, but not balanced since $|10101|_1-|00100|_1=2$. On the other hand, balanced words are not necessarily $1$-balanced.  For example, $0100001$ is balanced but not $1$-balanced.}
   We will first show that both $\l_{m,k}$ and $\p_{m,k}$ are $1$-balanced words. 
   
For $1< k<m$ with $k\nmid m$, let $\tilde \cL_{m,k}$ be the set  of all $1$-balanced words in $\cL_{m,k}$, and let $\tilde \P_{m,k}$ be the set of all $1$-balanced words in $\P_{m,k}$. Then  
\begin{equation}\label{eq:balanced-Lyndon-Perron}
\tilde\cL_{m,k}:=\cL_{m,k}\cap Y_{\lfloor\frac{m}{k}\rfloor-1}\quad\textrm{and}\quad \tilde \P_{m,k}:=\P_{m,k}\cap Z_{\lfloor\frac{m}{k}\rfloor-1},
\end{equation}
where $\lfloor r\rfloor$ denotes the integer part of a real number $r$, and for $q\in\N_{0}$,
\begin{equation}\label{eq:Yq-Zq}
\begin{split}
Y_q&:=\bigcup_{\ell=1}^\f\set{0^{i_1}10^{i_2}1\cdots 0^{i_\ell}1: i_j\in\set{q, q+1}~\forall 1\le j\le \ell},\\
Z_q&:=\bigcup_{\ell=1}^\f\set{10^{i_1}10^{i_2}\cdots 10^{i_\ell}: i_j\in\set{q, q+1}~\forall 1\le j\le \ell}.
\end{split}
\end{equation}
{Clearly, $\tilde\cL_{m,k}\ne\emptyset$ and $\tilde\P_{m,k}\ne\emptyset$, since for $m=k(q+1)+r$ with $q\in\N_0$ and $1\le r<k$ we have $(0^{q+1}1)^r(0^q1)^{k-r}\in\tilde\cL_{m,k}$ and $(10^q)^{k-r}(10^{q+1})^r\in\tilde\P_{m,k}$.}

 First we show that $\l_{m,k}$ is a $1$-balanced word.

\begin{lemma}
  \label{lem:lyndon-balanced}
  Let $m\in\N_{\ge 2}$ and $1< k< m$. If $k\nmid m$, then 
  $
  \l_{m,k}=\max \tilde\cL_{m,k}.
  $
\end{lemma}
 
\begin{proof}
 Let $m=k(q+1)+r$ with $q\in\N_0$ and $r\in\set{1,2,\ldots, k-1}$. Then $\tilde\cL_{m,k}=\cL_{m,k}\cap Y_q$.   Take $\s\in \cL_{m,k}\setminus\tilde \cL_{m,k}$. It suffices to find a word $\tilde \s\in \tilde \cL_{m,k}$ such that $\tilde \s\succ \s$.
 Write 
 \begin{equation}\label{eq:def-s}
 \s=\s_1 0^{t_1}1\;\s_2 0^{t_2}1\;\cdots \s_j 0^{t_j}1\;\s_{j+1},
 \end{equation}
 where $j\in\N$, each $\s_i\in {Y_{q}\cup\set{\ep}}$, and $t_i<q$ or $t_i>q+1$.
 Since $\s\in \cL_{m,k}$, by Lemma \ref{lem:Lyndon-words} it follows that $t_i<q$ for all $i$ if $\s_1\ne\ep$. Otherwise, $\s$ begins with $0^{t_1}1$ with $t_1>q+1$, and in this case we have $\tilde \s\succ\s$ for any $\tilde \s\in\tilde \cL_{m,k}$.
 
 In the following we assume $\s_1\ne\ep$ and then all $t_i<q$. We will inductively construct a word $\tilde \s\in\tilde \cL_{m,k}$ such that $\tilde \s\succ\s$. First we consider the following transformation on $\s$: if $\s_i\in Y_{q}$ contains the block $0^{q+1}1$, then we replace the last block $0^{q+1} 1$ in $\s_i$ by $0^{q}1$, and accordingly we replace {its right neighboring word} $0^{t_i}1$ by $0^{t_i+1}1$; if $\s_i$ does not contain the block $0^{q+1}1$, then we do nothing for the block $\s_i 0^{t_i}1$. {Here we point out that if $\s_i=\ep$, then we do nothing for the block $\s_i0^{t_i}1$.} Performing this transformation on $\s$ we obtain a new word of length $m$ with $k$ ones:
 \[
 \s'=\s_1'0^{t_1'}1\;\s_2' 0^{t_2'}1\;\cdots \s_j' 0^{t_j'}1\;\s_{j+1}'.
 \]
 Note that $|\s_i' 0^{t_i'}1|=|\s_i 0^{t_i}1|$ and  $|\s_i' 0^{t_i'}1|_1=|\s_i 0^{t_i}1|_1$ for all $1\le i\le j$, and $\s_{j+1}'=\s_{j+1}$. Since $\s'$ is not periodic, we have 
 \[{T_1(\s):=}\S(\s')\in\cL_{m,k}.\] We claim that $T_1(\s)\succ \s$.
 
 Note that $t_i'\le t_i+1\le q$ for all $1\le i\le j$. If $T_1(\s)$ begins with $0^{t_{i}'}1$ for some $i$, then by Lemma \ref{lem:Lyndon-words} it follows that $t_{i}'=q$, and hence $T_1(\s)=(0^{q}1)^{k}\in \cL_{m,k}$,    leading to a contradiction with our assumption $k\nmid m$. So, $T_1(\s)$ can only begin with a suffix of some $\s_i'$, and  this $\s_i'$ must contain the block $0^{q+1} 1$. Note that the original  $\s_i$ also contains the block $0^{q+1}1$, and it can be written  as  
 \[
 \s_i=s_1^{(i)}\cdots s_{p_i}^{(i)}s_{p_i+1}^{(i)}\cdots s_{q_i}^{(i)},
 \]
 where $s_1^{(i)}\ldots s_{p_i}^{(i)}$ ends with $0^{q+1}1$, and $s_{p_i+1}^{(i)}\ldots s_{q_i}^{(i)}$ is either the empty word $\ep$ or it equals $(0^{q}1)^n$ for some $n\in\N$.
 So, by the transformation on $\s_i$ we have
 \begin{equation}\label{eq:ui}
 \s_i'=s_1^{(i)}\cdots s_{p_i-1}^{(i)^+}s_{p_{i}+1}^{(i)}\ldots s_{q_i}^{(i)}.
 \end{equation}
 Now we prove $T_1(\s)\succ\s$ in the following two cases.
 
 Case 1. $T_1(\s)$  begins with a suffix of $\s_i'$ for some $i\in\set{1, 2,\ldots, j}$. Note that $T_1(\s)$  begins with the block $0^{q+1}1$. Then by (\ref{eq:ui}) it follows that $T_1(\s)$ begins with $s_{\ell+1}^{(i)}\ldots s_{p_i-1}^{(i)^+}s_{p_i+1}^{(i)}\ldots s_{q_i}^{(i)}$ for some $\ell<p_i-1$.   Thus, by Lemma \ref{lem:Lyndon-words} and using $\s=s_1 \ldots s_m\in \cL_{m,k}$ it follows that
 \[
s_{\ell+1}^{(i)}\ldots s_{p_i-1}^{(i)^+}\succ s_{\ell+1}^{(i)}\ldots s_{p_i-1}^{(i)}\lge s_1\ldots s_{p_i-\ell-1},
 \]
 which implies $T_1(\s)\succ \s$.
 
 Case 2. $T_1(\s)$ begins with a suffix of $\s_{j+1}'$. Since $\s_{j+1}'=\s_{j+1}$,  by (\ref{eq:def-s}) it gives that  $T_1(\s)$ also begins with a suffix of $\s=s_1\ldots s_m\in\cL_{m,k}$. Then by Lemma \ref{lem:Lyndon-words} we obtain 
that  $T_1(\s)\succ \s$.

 By Case 1 and Case 2 it follows that $T_1(\s)\succ \s$. Note that $T_1(\s)\in \cL_{m,k}$, and it is closer to be a $1$-balanced word. In other words, ${\rm dist}(T_1(\s), Y_q)<{\rm dist}(\s,Y_q)$, where ${\rm dist}$ denotes the Hamming distance. {If $T_1(\s)\notin Y_q$, then by} doing the same transformation on $T_1(\s)$ we get a new word  $T_1^2(\s)\in\cL_{m,k}$ with even smaller distance to $Y_q$. Furthermore, this new word is larger than $T_1(\s)$.  After performing this transformation  finitely many times we   get a $1$-balanced word   $\tilde\s\in Y_{q}\cap \cL_{m,k}=\tilde \cL_{m,k}$ and $\tilde \s\succ \s$, completing the proof.
 \end{proof}

{Similar strategy  for the proof of Lemma \ref{lem:lyndon-balanced} can be applied} to show that $\p_{m,k}$ is also $1$-balanced. 
\begin{lemma}
  \label{lem:perron-balanced}
  Let $m\in\N_{\ge 2}$ and $1<k<m$. If $k\nmid m$, then 
  $
 \p_{m,k}= \min \tilde\P_{m,k}.
  $
\end{lemma}
\begin{proof}
  Let $m=k(q+1)+r$ with $q\in\N_0$ and $r\in\set{1,2,\ldots, k-1}$. Then $\tilde \P_{m,k}:=\P_{m,k}\cap {Z_q}$. Take $\a\in \P_{m,k}\setminus\tilde \P_{m,k}$. It suffices to find a word $\tilde \a\in\tilde \P_{m,k}$ such that $\tilde \a\prec \a$.
  Write 
  \begin{equation}\label{eq:def-a}
  \a=\a_1 10^{t_1}\;\a_2 10^{t_2}\;\cdots \a_j 10^{t_j}\;\a_{j+1},
  \end{equation}
  where $j\in\N$, each $\a_i\in {Z_{q}\cup\{\ep\}}$, and $t_i<q$ or $t_i>q+1$.  
  Since $\a\in \P_{m,k}$, by Lemma \ref{lem:perron-words} it follows that $t_i>q+1$ for all $i$ assuming $\a_1\ne\ep$. If $\a_1=\ep$, then $\a$ begins with $10^{t_1}$ with $t_1<q$, and thus $\tilde \a\prec \a$ for all $\tilde \a\in\tilde \P_{m,k}$. 
  
  In the following we assume $\a_1\ne\ep$ and then all $t_i>q+1$. We will inductively construct a word $\tilde{\a}\in \tilde \P_{m,k}$ such that $\tilde\a\prec \a$. First we consider the following transformation on $\a$: if $\a_i\in Z_{q}$ contains the block $10^{q}$, then we replace the last block $10^{q}$ in $\a_i$ by $10^{q+1}$, and accordingly we replace { its right neighboring} word $10^{t_i}$ by $10^{t_i-1}$; otherwise, we do nothing on the block $\a_i 10^{t_i}$. { In particular, if $\a_i=\ep$  then we do nothing on the block $\a_i10^{t_i}$.}  Performing this transformation on $\a$ we get a new word of length $m$ with $k$ ones:
  \[
  \a'=\a_1' 10^{t_1'}\;\a_2' 10^{t_2'}\;\cdots \a_j' 10^{t_j'}\;\a_{j+1}'.
  \]
  Note that $|\a_i' 10^{t_i'}|=|\a_i 10^{t_i}|$ and $|\a_i' 10^{t_i'}|_1=|\a_i 10^{t_i}|_1$ for all $1\le i\le j$, and $\a_{j+1}'=\a_{j+1}$. Since $\a'$ is not periodic, we have 
  \[{T_2(\a):=}\L(\a')\in\P_{m,k}.\]
   We claim that $T_2(\a)\prec \a$.

  Note that $t_i'\ge t_i-1\ge  q+1$ for all $i$. If $T_2(\a)$ begins with $10^{t_{i}'}$ for some $i$, then by using $T_2(\a)\in\P_{m,k}$ we must have $t_{i}'=q+1$, and from this we can deduce $T_2(\a)=(10^{q+1})^k$, leading to a contradiction with $k\nmid m$. So, $T_2(\a)$ can only begin with a suffix of some $\a_i'$, and this $\a_i'$ must contain the block $10^{q}$. Note that in this case the original word $\a_i$ also contains the block $10^{q}$. Then we can write it as
  \[
  \a_i=a_1^{(i)}\ldots a_{p_i}^{(i)}a_{p_i+1}^{(i)}\ldots a_{q_i}^{(i)},
  \] 
  where $a_1^{(i)}\ldots a_{p_i}^{(i)}$ ends with $10^{q}$, and $a_{p_i+1}^{(i)}\ldots a_{q_i}^{(i)}$ is either the empty word $\ep$ or it equals $(10^{q+1})^n$ for some $n\in\N$.
  So, by our transformation on $\a_i$ we have
  \begin{equation}\label{eq:vi}
  \a_i'=a_1^{(i)}\ldots a_{p_i}^{(i)}\,0\, a_{p_i+1}^{(i)}\ldots a_{q_i}^{(i)}.
  \end{equation}
 Now we prove $T_2(\a)\prec \a$ in the following two cases.
  
Case 1. $T_2(\a)$ begins with a suffix of $\a_i'$ for some $i\in\set{1,2,\ldots, j}$. Note that $T_2(\a)$ begins with $10^{q}$. Then by (\ref{eq:vi}) it follows that $\L(\a')$ begins with 
  $
  a_{\ell+1}^{(i)}\ldots a_{p_i}^{(i)}\, 0\, a_{p_i+1}^{(i)}\ldots a_{q_i}^{(i)}
  $
  for some $\ell<p_i$. By using $\a=a_1 \ldots a_m\in \P_{m,k}$ and Lemma \ref{lem:perron-words} it follows that 
  \[
  a_{\ell+1}^{(i)}\ldots a_{p_i}^{(i)}\,0\prec a_{\ell+1}^{(i)}\ldots a_{p_i}^{(i)}\, 1\lle a_1\ldots a_{p_i-\ell+1},
  \]
  which implies $T_2(\a)\prec\a$. 
  
  Case 2. $T_2(\a)$ begins with a suffix of $\a_{j+1}'$. Note that $\a_{j+1}'=\a_{j+1}$. Then by (\ref{eq:def-a}) it follows that $T_2(\a)$ begins with a suffix of $\a\in \P_{m,k}$. So, $T_2(\a)\prec \a$ follows by Lemma \ref{lem:perron-words}.

  By Case 1 and Case 2 it follows that $T_2(\a)\prec\a$. Note that $T_2(\a)\in \P_{m,k}$ is closer to $Z_q$ in Hamming distance. {If $T_2(\a)\notin Z_q$, then by} doing the same transformation for $T_2(\a)$ we get a new word  $T_2^2(\a)\in\P_{m,k}$ with even smaller distance to $Z_q$. Furthermore, this new word is smaller than $T_2(\a)$.  After performing this transformation finitely many times we obtain a $1$-balanced word $\tilde \a\in Z_{q}\cap \P_{m,k}=\tilde \P_{m,k}$ and $\tilde{\a}\prec \a$. This completes the proof.  
\end{proof}
 
{We give an example to illustrate how our transformations $T_1$ and $T_2$ defined respectively in Lemmas \ref{lem:lyndon-balanced} and \ref{lem:perron-balanced} behave. 
\begin{example}First we consider the transformation $T_1$ as in Lemma \ref{lem:lyndon-balanced}. Let $m=16, k=5$. Then $q=2$.  Let  $\s=0^310^310^310101\in\cL_{16,5}\setminus\tilde\cL_{16,5}$.  Performing the transformation $T_1$  on $\s$ twice we obtain that 
\[T_1(\s)=0^310^310^210^2101\in\cL_{16,5}\setminus\tilde\cL_{16,5},\quad T_1^2(\s)=0^310^210^210^210^21\in\tilde\cL_{16,5}.\]

Next we consider the transformation $T_2$ as in Lemma \ref{lem:perron-balanced}. Let  $m=16, k=9$. Then $q=0$.  Let $\a=111101011010^210^2\in\P_{16,9}\setminus\tilde\P_{16,9}.$   Performing the transformation $T_2$ on $\a$  twice we obtain that 
\[T_2(\a)=111101010101010^2 \in \P_{16,9}\setminus\tilde\P_{16,9},\quad
T^2_2(\a)=1110101010101010\in\tilde\P_{16,9}.
\]
\end{example}}

In terms of Lemmas \ref{lem:lyndon-balanced} and \ref{lem:perron-balanced}, to {determine} $\l_{m,k}$ and $\p_{m,k}$ it suffices to {consider} 
$\tilde {\cL}_{m,k}:={\cL}_{m,k}\cap Y_{\lfloor\frac{m}{k}\rfloor-1}$ and $\tilde \P_{m,k}:=\P_{m,k}\cap Z_{\lfloor\frac{m}{k}\rfloor-1},$
 respectively.   
Given $q\in\N_0$, let $\Phi_q: Y_q\to\set{0,1}^*$ be the {substitution}   defined by 
\begin{equation}\label{eq:Phi}
  \Phi_q(0^{q+1}1)=0,\quad \Phi_q(0^q1)=1.
\end{equation}
Then for a word $\w=\w_1\w_2\ldots \w_n\in Y_q$ with each $\w_i\in\set{0^{q+1}1, 0^q 1}$ we have $\Phi_q(\w)=\Phi_q(\w_1)\Phi_q(\w_2)\ldots \Phi_q(\w_n)\in\set{0,1}^n$.
Similarly, {we define the substitution $\Psi_q: Z_q\to \set{0,1}^*$   by} 
\begin{equation}\label{eq:Psi}
  \Psi_q(10^{q+1})=0,\quad \Psi_q(10^q)=1.
\end{equation}
{We will} show that $\Phi_q$ and $\Psi_q$ are useful renormalization operators on $\tilde\cL_{m,k}$ and $\tilde \P_{m,k}$, respectively.  
\begin{lemma}
  \label{lem:property-Phi-Psi}
  Let $m=(q+1)k+r$ with $q\in\N_0$ and $1\le r<k$.  
  \begin{enumerate}[{\rm(i)}]
    \item $\Phi_q$ is strictly increasing on $Y_q$. Furthermore,  $\Phi_q(\tilde {\cL}_{m,k})= {\cL}_{k,k-r}$.
    \item $\Psi_q$ is strictly increasing on $Z_q$. Furthermore,  $\Psi_q(\tilde \P_{m,k})= \P_{k,k-r}$. 
  \end{enumerate}
\end{lemma}
\begin{proof}
  Since the proof of (ii) is similar to (i),   we only prove (i).  Let $\w=\w_1\ldots \w_n, \w'=\w_1'\ldots \w_\ell'\in Y_q$ with each $\w_i, \w_i'\in\set{0^{q+1}1, 0^q 1}$. Suppose $\w\prec \w'$. {Then $(\w_i):=\w^\f\prec(\w')^\f=:(\w_i')$, and  there exists $j\in\N$} such that
  \[
 \w_1\ldots \w_{j-1}=\w_1'\ldots \w_{j-1}'\quad\textrm{and}\quad \w_j\prec \w_j'.
 \]
  So, $\Phi_q(\w_1)\ldots \Phi_q(\w_{j-1})=\Phi_q(\w_1')\ldots \Phi_q(\w_{j-1}')$ and $\Phi_q(\w_j)=0<1=\Phi_q(\w_j')$. This implies $\Phi_q(\w)\prec \Phi_q(\w')$, proving the monotonicity of $\Phi_q$. 
 
  Note that $\tilde {\cL}_{m,k}={\cL}_{m,k}\cap Y_q$ with $m=(q+1)k+r$. Take $\w=\w_1\w_2\ldots \w_k\in\tilde {\cL}_{m,k}$ with each $\w_i\in\set{0^{q+1}1, 0^q 1}$. {Since $\w$ is a Lyndon word, by its definition  it follows that  $\si_c^i(\w)\succ \w$ for all $1\le i<m$, i.e.,
  \[
  \w_{i+1}\ldots \w_k\w_1\ldots \w_i\succ \w_1\ldots \w_{k}\quad \forall ~1\le i<k.
  \]
  By the monotonicity of $\Phi_q$ we obtain that 
  \[
  \Phi_q(\w_{i+1})\ldots \Phi_q(\w_k){\Phi_q(\w_1)\ldots\Phi_q(\w_i)}\succ \Phi_q(\w_1)\ldots \Phi_q(\w_{k})\quad\forall~1\le i<k.
  \]}
  This implies that $\Phi_q(\w)=\Phi_q(\w_1)\ldots \Phi_q(\w_k)$ is a Lyndon word. Note that $\Phi_q(\w)$ is a word of length $k$ and it contains $r$ zeros. Then $\Phi_q(\w)\in {\cL}_{k,k-r}$. Since   $\w\in \tilde {\cL}_{m,k}$ was chosen  arbitrarily, we conclude that $\Phi_q(\tilde {\cL}_{m,k})\subset {\cL}_{k, k-r}$.  
  
  On the other hand, take $\mathbf s=s_1\ldots s_k\in {\cL}_{k,k-r}$ with each $s_i\in\set{0,1}$. 
  Note that $\Phi_q^{-1}: 0\mapsto 0^{q+1}1;~ 1\mapsto 0^q 1$. Then $\Phi_q^{-1}(\mathbf s)=\Phi_q^{-1}(s_1)\ldots\Phi_q^{-1}(s_k)$ is a word of length $(q+1)k+r=m$ and it has precisely $k$ ones. Since $\mathbf s=s_1\ldots s_k$ is a Lyndon word, we have {$s_{i+1}\ldots s_ks_1\ldots s_i\succ s_1\ldots s_{k}$} for all $1\le i<k$. Note that the block map $\Phi_q^{-1}$ is also strictly increasing. Then 
  \begin{equation}\label{eq:inequality-s-Phiq}
    {\Phi_q^{-1}(s_{i+1})\ldots \Phi_q^{-1}(s_k)\Phi_q^{-1}(s_1)\ldots\Phi_q^{-1}(s_i)\succ \Phi_q^{-1}(s_1)\ldots \Phi_q^{-1}(s_{k})\quad \forall 1\le i<k.}
  \end{equation}
  Furthermore, $\Phi_q^{-1}(s_1)=\Phi_q^{-1}(0)=0^{q+1}1$. This together with (\ref{eq:inequality-s-Phiq}) implies that 
  \[
  {\si_c^n}(\Phi_q^{-1}(\mathbf s))\succ \Phi_q^{-1}(\mathbf s)\quad \forall 1\le n<m.
  \]
  So,   $\Phi_q^{-1}(\mathbf s)\in {\cL}_{m,k}$. Clearly, $\Phi_q^{-1}(\mathbf s)\in Y_q$. Thus, $\Phi_q^{-1}(\mathbf s)\in\tilde {\cL}_{m,k}$, i.e., $\mathbf s\in \Phi_q(\tilde {\cL}_{m,k})$. This proves ${\cL}_{k,k-r}\subset\Phi_q(\tilde {\cL}_{m,k})$. 
\end{proof}

For a word $\w\in\set{0,1}^*$, recall that    ${\L}(\w)$ and ${\S}(\w)$ are the lexicographically largest and lexicographically smallest cyclic permutations of $\w$, respectively. 
In the following we show that the following two diagrams are both commutative. 
\[
\begin{tikzcd}
\tilde\cL_{m,k} \arrow[r, "\Phi_q"] \arrow[d,"\L"]& \cL_{k,k-r} \arrow[d,"\L"]\\
\tilde\P_{m,k}\arrow[r,"\Psi_q"]&\P_{k,k-r}
\end{tikzcd}
\qquad\textrm{and}\qquad 
\begin{tikzcd}
\tilde \P_{m,k}\arrow[r, "\Psi_q"] \arrow[d,"\S"]& \P_{k,k-r}\arrow[d,"\S"]\\
\tilde \cL_{m,k}\arrow[r,"\Phi_q"]&\cL_{k,k-r}
\end{tikzcd}
\] 
\begin{lemma}
  \label{lem:commute-S-L-Phi-Psi}
 {Let $m=k(q+1)+r$ with $q\in\N_0$ and $r\in\set{1,\ldots,k-1}$.  
 \begin{enumerate}[{\rm(i)}]
   \item For any $\w\in \tilde \cL_{m,k}$ we have
 $
      {\L}(\Phi_q(\w))=\Psi_q({\L}(\w)).
 $
   
   \item  For any    $\w\in \tilde \P_{m,k}$ we have 
 $
  {\S}(\Psi_q(\w))=\Phi_q({\S}(\w)).
 $
 \end{enumerate}
}
\end{lemma} 
\begin{proof}{First we prove (i).} Let $\w=\w_1\w_2\ldots\w_k\in \tilde \cL_{m,k}=\cL_{m,k}\cap Y_q$ with each $\w_i\in\{0^q1, 0^{q+1}1\}$. {Then  
\[{\L}(\w)=\si_c^{-1}( \w_{j+1}\ldots \w_k\w_1\ldots  \w_j)\in\tilde \P_{m,k}\] for some $0\le j<k$,   where {$\si_c$  is the cyclic permutation}.}  
By Lemma \ref{lem:property-Phi-Psi} it follows that 
  \begin{equation}\label{eq:pi-perron}
    \Psi_q({\L}(\w))=\Psi_q({\si_c^{-1}( \w_{j+1}\ldots \w_k\w_1\ldots   {\w}_j)})=\pi_{j+1}\ldots\pi_k\pi_1\ldots\pi_j\in \P_{k,k-r},
  \end{equation}
  where $\pi_1\ldots\pi_k=\Phi_q(\w_1\ldots\w_k)\in\cL_{k,k-r}$.
 Comparing the definitions of $\Phi_q$ and $\Psi_q$ in (\ref{eq:Phi}) and (\ref{eq:Psi}) we obtain that 
  \[
    \Psi_q(10^{q+1})=\Phi_q(0^{q+1}1)=0,\quad \Psi_q(10^q)=\Phi_q(0^q1)=1.
  \]
Therefore, 
  \[
  {\L}(\Phi_q(\w))={\L}(\Phi_q(\w_1)\cdots\Phi_q(\w_k))={\L}(\pi_1\ldots\pi_k)=\pi_{j+1}\ldots\pi_k\pi_1\ldots\pi_j,
  \]
  where the last equality follows by (\ref{eq:pi-perron}). This proves
  ${\L}(\Phi_q(\w))=\Psi_q({\L}(\w))$.
  
  {Next we prove (ii).} Let $\w=\w_1\w_2\ldots\w_k\in \tilde\P_{m,k}=\P_{m,k}\cap Z_q$ with each $\w_i\in\{10^q, 10^{q+1}\}$. {Then there exists $0\le j<k$ such that 
   ${\S}(\w)={\si_c(\w_{j+1}\ldots \w_k\w_1\ldots \w_j)} \in \tilde\cL_{m,k}.$
 }
By Lemma \ref{lem:property-Phi-Psi} we have
  \begin{equation}\label{eq:pi-lyndon}
  \Phi_q({\S}(\w))=\Phi_q({\si_c(\w_{j+1}\ldots\w_k\w_1\ldots\w_j)})=\pi_{j+1}\ldots\pi_k\pi_1\ldots\pi_j \in {\cL}_{k,k-r},
  \end{equation}
  where $\pi_1\ldots\pi_k=\Psi_q(\w_1\ldots\w_k)\in\P_{k,k-r}$.
   Comparing the definitions of $\Psi_q$ and $\Phi_q$,  we obtain that 
  \[
  {\S}(\Psi_q(\w))={\S}(\Psi_q(\w_1)\cdots\Psi_q(\w_k)) ={\S}(\pi_1\ldots\pi_k)=\pi_{j+1}\ldots\pi_k\pi_1\ldots\pi_j,
  \]
  where the last equality holds by (\ref{eq:pi-lyndon}).
This establishes ${\S}(\Psi_q(\w))=\Phi_q({\S}(\w))$.
  \end{proof}
  
   Recall  from (\ref{eq:farey-word}) that for a rational $\frac{a}{b}\in(0,1)$ with $\gcd(a,b)=1$ there exists a unique Farey word $\w_{a/b}$ of length $b$ with precisely $a$ ones.  Recall from (\ref{eq:U0-U1}) that $U_0(0)=0, U_0(1)=01=U_1(0)$ and $U_1(1)=1$. Furthermore, for a word $d_1\ldots d_n\in\set{0,1}^*$ we have $U_{d_1\ldots d_n}=U_{d_1}\circ U_{d_2}\circ\cdots\circ U_{d_n}$.
  \begin{lemma}
    \label{lem:U-Phi-Farey-word}
   Let $\w_{a/b}$ be a Farey word with   $\gcd(a,b)=1$. Then for any $n\in\N$ we have 
   \[
   U_{0^n}(\w_{a/b})=\w_{\frac{a}{a n+b}},\quad U_{1^n}(\w_{a/b})=\w_{\frac{(b-a)n+a}{(b-a)n+b}}.
   \]
  Therefore, for any $q\in\N_0$ we have 
   \[
   \Phi_q^{-1}(\w_{a/b})=U_{0^q1}(\w_{a/b})=\w_{\frac{b}{b(q+2)-a}}.
   \]
  \end{lemma}
  \begin{proof}
    Note by (\ref{eq:U0-U1}) that $U_{0^n}(0)=0$ and $U_{0^n}(1)=0^n1$. Then the substitution $U_{0^n}$ does not increase the number of digit $1$, and it replaces each digit $1$ by a block $0^n 1$. Since $|\w_{a/b}|=b$ and $|\w_{a/b}|_1=a$, it follows that
    \[|U_{0^n}(\w_{a/b})|=a n+b,\quad |U_{0^n}(\w_{a/b})|_1=a.\]
    By Lemma \ref{lem:Farey-words-characterization} we have $U_{0^n}(\w_{a/b})=\w_{\frac{a}{a n+b}}$.
    
    Similarly, note by (\ref{eq:U0-U1}) that  $U_{1^n}(0)=01^n$ and $U_{1^n}(1)=1$. Then the substitution $U_{1^n}$ does not increase the number of digit $0$, and it replaces each digit $0$ by a block $01^n$. Since $|\w_{a/b}|=b, |\w_{a/b}|_1=a$ and $|\w_{a/b}|_0=b-a$, we obtain
    \[
    |U_{1^n}(\w_{a/b})|=(b-a)n+b,\quad |U_{1^n}(\w_{a/b})|_1=(b-a)n+a.
    \]
    By Lemma \ref{lem:Farey-words-characterization} we conclude that $U_{1^n}(\w_{a/b})=\w_{\frac{(b-a)n+a}{(b-a)n+b}}$.
    
    Take $q\in\N_0$. Note by (\ref{eq:U0-U1}) and (\ref{eq:Phi}) that 
\begin{align*}
\Phi_q^{-1}(0)&=0^{q+1}1=U_{0^q1}(0),\quad \Phi_q^{-1}(1)=0^q 1=U_{0^q1}(1).
\end{align*}
Then $\Phi_q^{-1}=U_{0^q 1},$ and hence,
\begin{align*}
  \Phi_q^{-1}(\w_{a/b}) & = U_{0^q 1}(\w_{a/b})=U_{0^q}(\w_{\frac{b}{2b-a}})=\w_{\frac{b}{b(q+2)-a}}.
\end{align*} 
  \end{proof}

\begin{proof}[Proof of Theorem \ref{main-2:L-P-mk} (i)]
If $k=1$, then it is clear that $\l_{m,1}=0^{m-1}1=\w_{1/m}$ and $\p_{m,1}=10^{m-1}=\L(\w_{1/m})$. In the following we assume $1<k<m$
with $\gcd(m,k)=1$.

 First we show that $\l_{m,k}=\S(\p_{m,k})=\w_{k/m}$ implies $\p_{m,k}=\L(\l_{m,k})=\L(\w_{k/m})$.  
  Suppose $\l_{m,k}={\S}(\p_{m,k})=\w_{k/m}$. Write $\l_{m,k}=u_1\ldots u_m$ and $\p_{m,k}=v_1\ldots v_m$. Then there exists $j\in\set{1,2,\ldots, m-1}$ such that 
  \begin{equation}\label{eq:august-7}
  u_1\ldots u_m=v_{j+1}\ldots v_m v_1\ldots v_j.
  \end{equation} Since $\p_{m,k}\in \P_{m,k}$, by (\ref{eq:august-7}) it follows that 
  \[
  \p_{m,k}=v_1\ldots v_m={\L}(v_1\ldots v_m)={\L}(u_1\ldots u_m)={\L}(\l_{m,k})=\L(\w_{k/m}). 
  \]
  
  In the following it suffices to prove $\l_{m,k}={\S}(\p_{m,k})=\w_{k/m}$, which will be done recursively by using Euclid's algorithm on long division. Set $k_{-1}:=m$ and $k_0:=k$. Since $\gcd(k_{-1}, k_0)=\gcd(m,k)=1$ and $k_0<k_{-1}$, we can find $q_0\in\N_0$ and $r_0\in\set{1,\ldots, k_0-1}$ such that \begin{equation}\label{eq:k-minus-1}k_{-1}=(q_0+1)k_0+r_0.\end{equation}
   By Lemmas \ref{lem:lyndon-balanced} and \ref{lem:property-Phi-Psi} it follows that 
  \begin{equation}\label{eq:august13-1}
    \Phi_{q_0}(\l_{k_{-1},k_0})=\Phi_{q_0}(\max\tilde {\cL}_{k_{-1}, k_0})=\max\Phi_{q_0}(\tilde {\cL}_{k_{-1}, k_0})=\max {\cL}_{k_0, k_1}=\max\tilde {\cL}_{k_0, k_1},
  \end{equation}
  where $k_1:=k_0-r_0$. Similarly, by Lemmas \ref{lem:perron-balanced} and \ref{lem:property-Phi-Psi} we obtain that 
  \begin{equation}\label{eq:august13-1'}
    \Psi_{q_0}(\p_{k_{-1}, k_0})=\Psi_{q_0}(\min\tilde \P_{k_{-1}, k_0})=\min\Psi_{q_0}(\tilde \P_{k_{-1}, k_0})=\min \P_{k_0, k_1}=\min\tilde \P_{k_0, k_1}.
  \end{equation}
  
  If $r_0=k_0-1$, then $k_1=1$, which implies  $\tilde {\cL}_{k_0, k_1}=\set{0^{k_0-1}1}$ and $\tilde \P_{k_0, k_1}=\set{10^{k_0-1}}$. By (\ref{eq:august13-1}) and (\ref{eq:august13-1'}) it follows that  
  \begin{equation}\label{eq:dec18-1}
  \Phi_{q_0}(\l_{k_{-1}, k_0})=0^{k_0-1}1={\S}(10^{k_0-1})={\S}(\Psi_{q_0}(\p_{k_{-1}, k_0}))=\Phi_{q_0}({\S}(\p_{k_{-1}, k_0})),
  \end{equation}
  where the last equality follows by Lemma \ref{lem:commute-S-L-Phi-Psi}. {Note by (\ref{eq:farey-word})} that $0^{k_0-1}1=\w_{\frac{1}{k_0}}$. Then by (\ref{eq:dec18-1}), (\ref{eq:k-minus-1}) and Lemma \ref{lem:U-Phi-Farey-word} we obtain that 
  \[
  \l_{k_{-1},k_0}=\S(\p_{k_{-1}, k_0})=\Phi_{q_0}^{-1}(\w_{\frac{1}{k_0}})=\w_{\frac{k_0}{k_0(q_0+2)-1}}=\w_{\frac{k_0}{k_{-1}}}
  \]
 as required.

  If $r_0<k_0-1$, then $k_1=k_0-r_0\in(1,k_0)$. Note that $\gcd(k_0,k_1)=\gcd(k_0,r_0)=\gcd(k_{-1}, k_0)=1$. Then we can find $q_1\in\N_0$ and $r_1\in\set{1,\ldots, k_1-1}$ such that 
  \[k_0=(q_1+1)k_1+r_1.\]
  By Lemmas \ref{lem:lyndon-balanced}, \ref{lem:property-Phi-Psi} and (\ref{eq:august13-1}) it follows that 
  \begin{equation*} 
    \Phi_{q_1}\circ \Phi_{q_0}(\l_{k_{-1},k_0})=\Phi_{q_1}(\max\tilde {\cL}_{k_0, k_1})=\max\Phi_{q_1}(\tilde {\cL}_{k_0, k_1})=\max {\cL}_{k_1, k_2}=\max\tilde {\cL}_{k_1, k_2},
  \end{equation*}
  where $k_2:=k_1-r_1$. Similarly, by Lemmas \ref{lem:perron-balanced},  \ref{lem:property-Phi-Psi} and (\ref{eq:august13-1'}) we obtain that 
  \begin{equation*} 
    \Psi_{q_1}\circ\Psi_{q_0}(\p_{k_{-1}, k_0})=\Psi_{q_1}(\min\tilde \P_{k_0, k_1})=\min\Psi_{q_1}(\tilde \P_{k_0, k_1})=\min \P_{k_1, k_2}={\min\tilde \P_{k_1, k_2}.}
  \end{equation*}
  If $r_1=k_1-1$, then $k_2=1$, which implies $\tilde {\cL}_{k_1, k_2}=\set{0^{k_1-1}1}$ and $\tilde \P_{k_1, k_2}=\set{10^{k_1-1}}$. Note that $0^{k_1-1}1=\w_{\frac{1}{k_1}}$. By Lemma \ref{lem:U-Phi-Farey-word} and  the same argument as above we can deduce that 
  \[\l_{k_{-1}, k_0}={\S}(\p_{k_{-1},k_0})=\Phi_{q_0}^{-1}\circ\Phi_{q_1}^{-1}(\w_{\frac{1}{k_1}})=\Phi_{q_0}^{-1}(\w_{\frac{k_1}{k_0}})=\w_{\frac{k_0}{k_{-1}}}.\]

  If $r_1<k_1-1$, then we can continue the above argument. After finitely many steps, we can   find $(q_i, k_i, r_i)\in\N_0^3$ with $i=0,1,\ldots, \ell$ such that 
  \begin{equation}\label{eq:ki}
  k_{i-1}=(q_i+1)k_i+r_i,\quad i=0,1,\ldots, \ell,
  \end{equation}
  where $k_{i+1}:=k_i-r_i$ and {$1\leq r_i<k_i-1$} for all $0\le i<\ell$, and $r_\ell=k_\ell-1$. 
  Therefore, 
   \[
   \Phi_{q_\ell}\circ\Phi_{q_{\ell-1}}\circ\cdots\circ\Phi_{q_0}(\l_{k_{-1}, k_0})=\max\tilde {\cL}_{k_\ell, k_\ell-r_\ell}=0^{k_\ell-1}1,
   \]
   and 
   \[
   \Psi_{q_\ell}\circ\Psi_{q_{\ell-1}}\circ\cdots\circ \Psi_{q_0}(\p_{k_{-1}, k_0})=\min\tilde \P_{k_{\ell}, k_{\ell}-r_{\ell}}=10^{k_\ell-1}.
   \]
  Note that {$0^{k_\ell-1}1=\w_{\frac{1}{k_\ell}}$}. Then by Lemma \ref{lem:commute-S-L-Phi-Psi}   it follows that
   \begin{align*}
\Phi_{q_\ell}\circ\Phi_{q_{\ell-1}}\circ\cdots\circ\Phi_{q_0}(\l_{k_{-1}, k_0})=\w_{\frac{1}{k_\ell}}&={\S}( \Psi_{q_\ell}\circ\Psi_{q_{\ell-1}}\circ\cdots\circ \Psi_{q_0}(\p_{k_{-1}, k_0}))\\
   &= \Phi_{q_\ell}\circ\Phi_{q_{\ell-1}}\circ\cdots\circ \Phi_{q_0}({\S}(\p_{k_{-1}, k_0})).
   \end{align*}
  So, by (\ref{eq:ki}), Lemmas \ref{lem:property-Phi-Psi} and \ref{lem:U-Phi-Farey-word} we conclude that 
   \begin{align*}
     \l_{k_{-1},k_0} =\S(\p_{k_{-1},k_0})&=\Phi_{q_0}^{-1}\circ\cdots\circ\Phi_{q_{\ell-1}}^{-1}\circ\Phi_{q_\ell}^{-1}(\w_{\frac{1}{k_\ell}}) \\
       &=\Phi_{q_0}^{-1}\circ\cdots\circ\Phi_{q_{\ell-1}}^{-1}(\w_{\frac{k_\ell}{k_{\ell-1}}}) =\cdots =\w_{\frac{k_0}{k_{-1}}}.
   \end{align*}Since $k_{-1}=m$ and $k_0=k$, this completes the proof.
\end{proof}

\section{Extremal   Lyndon   and Perron words: non-coprime case}\label{sec:extreme Lyndon-Perron-Notprime}
In this section we will determine the extremal words   $\l_{m,k}=\max\cL_{m,k}$ and $\p_{m,k}=\min\P_{m,k}$ when $\gcd(m,k)>1$, and prove Theorem \ref{main-2:L-P-mk} (ii). First we assume ${\gcd(m,k)}=k$, i.e., $k|m$. 
Let    $2\le k<m$, and write $m=k(q+1)$ for some $q\in\N$. Then the Farey word $\w_{k/m}=\w_{\frac{1}{q+1}}=0^q1$.   
In this case, Theorem \ref{main-2:L-P-mk} (ii) can be simplified as follows.  
\begin{proposition}\label{prop:extrem-Lyndon-Perron-dividable}
Let   $m=k(q+1)$ with $q\in\N$. Then
\[
  \l_{m,k}=0^q 1\bullet 01^{k-1}\quad\textrm{and}\quad \p_{m,k}=0^q 1\bullet 10^{k-1}.
 \]
\end{proposition}
Note that for $m=k(q+1)$ with $q\in\N$, neither    $\cL_{m,k}$  nor $\P_{m,k}$ contains a balanced word. Similar to the definition of $1$-balanced words, we call a  word $\w\in\set{0,1}^*$ \emph{$2$-balanced} if the numbers of consecutive zeros in $\w$ {are} different up to $2$. {Clearly, a $1$-blanced word is also  $2$-balanced.}  Let $\tilde\cL_{m,k}$ be the set of all $2$-balanced words in $\cL_{m,k}$, and let $\tilde\P_{m,k}$ be the set of all $2$-balanced words in $\P_{m,k}$. Then for $m=k(q+1)$ we have
\begin{equation}\label{eq:quasi-balanced-Lyndon-Perron}
\tilde\cL_{m,k}=\cL_{m,k}\cap Y_{q}'\quad\textrm{and}\quad \tilde{\P}_{m,k}=\P_{m,k}\cap Z'_{q},
\end{equation}
where  
\begin{equation}\label{eq:Yq'-Zq'}
\begin{split}
  Y_q' & :=\bigcup_{\ell=1}^\f \set{0^{i_1}10^{i_2}1\cdots 0^{i_\ell}1: i_j\in\set{q-1, q, q+1}~\forall 1\le j\le \ell},\\
  Z_q'& :=\bigcup_{\ell=1}^\f\set{10^{i_1}10^{i_2}\cdots 10^{i_\ell}: i_j\in\set{q-1, q, q+1}~\forall 1\le j\le \ell}.
\end{split}
\end{equation}
{Clearly, $0^{q+1}1(0^q1)^{k-2}0^{q-1}1\in\tilde\cL_{m,k}$ and $10^{q-1}(10^q)^{k-2}10^{q+1}\in\tilde\P_{m,k}$. So, both sets $\tilde\cL_{m,k}$ and $\tilde\P_{m,k}$ are nonempty.}

Inspired by Lemma \ref{lem:lyndon-balanced}, we show that $\l_{m,k}=\max\cL_{m,k}$ is a $2$-balanced word when $k\mid m$.

\begin{lemma}
 \label{lem:lyndon-quasi-balanced}
Let $m=k(q+1)$ with $q\in\N$. Then  
$
  \l_{m,k} = \max\tilde{\cL}_{m,k}.
$
\end{lemma}


\begin{proof}
Take $\s \in {\cL}_{m,k} \setminus {\tilde{\cL}_{m,k}}$. It suffices to find a word $\tilde{\s}\in {\tilde{\cL}_{m,k}}$ such that $\tilde{\s} \succ \s$. Note by (\ref{eq:quasi-balanced-Lyndon-Perron}) that
 ${\tilde{\cL}_{m,k}}={\cL}_{m,k} \cap Y'_q$. Then we can write $\s$ as
  \begin{equation}\label{eq:dec18-2}
    \s = \s_10^{t_1}1\s_20^{t_2}1 \cdots \s_j0^{t_j}1\s_{j+1},
  \end{equation}
where $j \in\mathbb{N}$, each $\s_i\in {Y_q'\cup\{\ep\}}$, and $t_i < q-1~\text{or}~t_i > q+1$. 
Since $\s\in {\cL}_{m,k}$, by Lemma \ref{lem:Lyndon-words} it follows that $t_i < q-1$ for all $i$ if $\s_1 \neq \varepsilon$. Otherwise. $\s$ begins with $0^{t_1}1$ with $t_1 > q+1$, and in this case we have $\tilde{\s} \succ \s$ for any $\tilde{\s}\in {\tilde{\cL}_{m,k}}$.

In the following we assume $\s_1\ne\ep$ and then all $t_i<q-1$. We will inductively construct a word $\tilde \s\in{\tilde{\cL}_{m,k}}$ such that $\tilde \s\succ\s$. First we consider the following transformation on $\s$: if $\s_i\in Y_q'$ contains the block $0^{q+1}1$, then we replace the last block $0^{q+1}1$ in $\s_i$ by $0^q1$, and accordingly we replace its right neighboring  word $0^{t_i}1$ by $0^{t_i+1}1$; if $\s_i$ does not contain the block $0^{q+1}1$, then we do nothing for the block $\s_i 0^{t_i}1$. In particular, in $\s_i=\ep$ then we do nothing for the blcok $\s_i 0^{t_i}1$. Performing this transformation on $\s$ we obtain a new word of length $m$ with $k$ ones:
 \begin{equation}\label{eq:T_3(s)}
 \s'=\s_1'0^{t_1'}1\;\s_2' 0^{t_2'}1\;\cdots \s_j' 0^{t_j'}1\;\s_{j+1}'.
 \end{equation}
 Note that $|\s_i' 0^{t_i'}1|=|\s_i 0^{t_i}1|$ and $|\s_i' 0^{t_i'}1|_1=|\s_i 0^{t_i}1|_1$ for all $1\le i\le j$, and $\s_{j+1}'=\s_{j+1}$. 
 
 {If $\S(\s')$ is not periodic, then we set $T_3(\s):=\S(\s')$. Otherwise, there exists a Lyndon word $\r_1\ldots \r_p$ with each $\r_i=0^{\ell_i}1\in\set{0^s1: 0\le s\le q+1}$ such that 
 \[\S(\s')=(\r_1\ldots \r_p)^n=(0^{\ell_1}10^{\ell_2}1\ldots 0^{\ell_p}1)^n\] for some $n\ge 2$. If $\S(\s')$ begins with $\r_1=0^{\ell_1}1$ for some $\ell_1<q$, then $\ell_i<q$ for all $1\le i\le p$, which leads to a contradiction with 
 \begin{equation}\label{eq:T_3-1}
 |\S(\s')|=|\s|=m=(q+1)k=(q+1)|\s|_1=(q+1)|\S(\s')|_1.
 \end{equation} If $\r_1=0^q1$, then by  (\ref{eq:T_3-1}) it follows that $\S(\s')=(0^q1)^k$, which implies that the original word $\s$ belongs to $Y_q'$, leading to a contradiction with our assumption.  So, we must have
   $\r_1=0^{q+1}1$, and by (\ref{eq:T_3-1}) it follows that  $p\ge 2$, and  there exists $i\in\set{2,3,\ldots,p}$ such that $\ell_i\le q-1$. Let $v:=\min\set{2\le i\le p: \ell_i\le q-1}$. Then $\ell_i\in\set{q,q+1}$ for all $i<v$. Note that $\ell_1=q+1$. Let $u:=\max\set{1\le j<v: \ell_j=q+1}$. Then $\ell_i=q$ for all $u<i<v$. In this case we set
   \begin{equation}\label{eq:T3-0}
   \begin{split}
   T_3(\s)&:=(\r_1\ldots \r_p)^{n-1}\r_1\ldots \r_{u-1}(0^{\ell_u-1}1)\r_{u+1}\ldots \r_{v-1}(0^{\ell_v+1}1)\r_{v+1}\ldots \r_p\\
   &=(\r_1\ldots \r_p)^{n-1}\r_1\ldots \r_{u-1}(0^{q}1)^{v-u} (0^{\ell_v+1}1)\r_{v+1}\ldots \r_p.
   \end{split}
   \end{equation}
 We claim that $T_3(\s)\in\cL_{m,k}$ and $T_3(\s)\succ\s$.   
 
 First we prove $T_3(\s)\in\cL_{m,k}$. Note that $T_3(\s)$ has length $m$ with $k$ ones. So, it suffices to prove that $T_3(\s)$ is Lyndon. If $\S(\s')$ is not periodic, then it is clear that $T_3(\s)=\S(\s')$ is a Lyndon word. Suppose $\S(\s')$ is periodic, and then $T_3(\s)$ is defined as in (\ref{eq:T3-0}). Note that $\r_1\ldots \r_p$ is Lyndon. Then by Lemma \ref{lem:Lyndon-words},
 \begin{equation}\label{eq:T3-1}
 \r_{i+1}\ldots \r_p\succ\r_1\ldots\r_{p-i}\quad\forall 1\le i<p.
 \end{equation}
Since $\ell_v+1\le q$ and $\r_1=0^{q+1}1$,  by (\ref{eq:T3-0}),   (\ref{eq:T3-1}) and Lemma \ref{lem:Lyndon-words} it follows that $T_3(\s)$ is Lyndon. 
    
 
Next we prove $T_3(\s)\succ\s$. Note by (\ref{eq:T3-0}) that $T_3(\s)\lge \S(\s')$. Then it suffices to prove $\S(\s')\succ\s$. By the above argument it follows that $\S(\s')$ must begin with $0^{q+1}1$. Note by (\ref{eq:T_3(s)}) that $t_i'\le t_i+1\le q-1$ for all $1\le i\le j$. Then   $\S(\s')$  begins with a suffix of some $\s_i'$, and by our transformation  the original word $\s_i$ also contains the block $0^{q+1}1$.} Thus, we can write it as  
 \[
 \s_i=s_1^{(i)}\cdots s_{p_i}^{(i)}s_{p_i+1}^{(i)}\cdots s_{q_i}^{(i)},
 \]
 where $s_1^{(i)}\ldots s_{p_i}^{(i)}$ ends with $0^{q+1}1$, and $s_{p_i+1}^{(i)}\ldots s_{q_i}^{(i)}$ never contains the word $0^{q+1}1$.
 So, by the transformation on $\s_i$ we have
 \begin{equation}\label{eq:u'i}
 \s_i'=s_1^{(i)}\cdots s_{p_i-1}^{(i)^+}s_{p_{i}+1}^{(i)}\ldots s_{q_i}^{(i)}.
 \end{equation}
Now we prove $\S(\s')\succ\s$ in the following two cases. 
 
 Case 1. $\S(\s')$  begins with a suffix of $\s_i'$ for some $i\in\set{1, 2,\ldots, j}$. Note that $\S(\s')$  begins with the block $0^{q+1}1$. Then by (\ref{eq:u'i}) it follows that $\S(\s')$ begins with $s_{\ell+1}^{(i)}\ldots s_{p_i-1}^{(i)^+}s_{p_i+1}^{(i)}\ldots s_{q_i}^{(i)}$ for some $\ell<p_i-1$.   Thus, by Lemma \ref{lem:Lyndon-words} and using $\s=s_1\ldots s_m\in {\cL}_{m,k}$ it follows that
 \[{
s_{\ell+1}^{(i)}\ldots s_{p_i-1}^{(i)^+}\succ s_{\ell+1}^{(i)}\ldots s_{p_i-1}^{(i)}\lge s_1\ldots s_{p_i-\ell-1},}
 \]
 which implies $\S(\s')\succ \s$.
 
 Case 2. $\S(\s')$ begins with a suffix of $\s_{j+1}'$. Since $\s_{j+1}'=\s_{j+1}$,  by (\ref{eq:dec18-2}) it gives that $\S(\s')$ also begins with a suffix of $\s\in\cL_{m,k}$. Then by Lemma \ref{lem:Lyndon-words} we obtain that $\S(\s')\succ \s$.

 By Case 1 and Case 2 it follows that $T_3(\s)\lge\S(\s')\succ \s$. Note that $T_3(\s)\in {\cL}_{m,k}$, and it is closer to $Y_q'$ in Hamming distance. {If $T_3(\s)\notin Y_q'$, then by doing} the same transformation on $T_3(\s)$ we get a new word  $T_3^2(\s)\in{\cL}_{m,k}$ with even smaller distance to $Y_q'$. Furthermore, this new word is larger than $T_3(\s)$.  After performing this transformation  finitely many times we can get a word   $\tilde\s\in Y'_q\cap {\cL}_{m,k}={\tilde{\cL}_{m,k}}$ and $\tilde \s\succ \s$, completing the proof.
 \end{proof}

 Similarly, inspired by the proof of Lemma \ref{lem:perron-balanced} we show that $\p_{m,k}$ is $2$-balanced when $k|m$.
\begin{lemma}
 \label{lem:perron-quasi-balanced}
  Let $m=k(q+1)$ with $q\in\N$. Then  
$
 \p_{m,k}= \min \tilde\P_{m,k}.
 $
 \end{lemma}
\begin{proof}
Take $\a\in \P_{m,k}\setminus{\tilde \P_{m,k}}$. It suffices to find a word $\tilde \a\in{\tilde \P_{m,k}}$ such that $\tilde \a\prec \a$. Note by (\ref{eq:quasi-balanced-Lyndon-Perron}) that
    ${\tilde \P_{m,k}}:=\P_{m,k}\cap Z'_q$. 
  Then we can write $\a$ as 
  \[
  \a=\a_1 10^{t_1}\;\a_2 10^{t_2}\;\cdots \a_j 10^{t_j}\;\a_{j+1},
  \]
  where $j\in\N$, each $\a_i\in {Z'_q\cup\{\ep\}}$, and $t_i<q-1$ or $t_i>q+1$.  Since $\a\in \P_{m,k}$, by Lemma \ref{lem:perron-words} it follows that $t_i>q+1$ for all $i$ assuming $\a_1\ne\ep$. If $\a_1=\ep$, then $\a$ begins with $10^{t_1}$ with $t_1<q-1$, and thus $\tilde \a\prec \a$ for all $\tilde \a\in{\tilde \P_{m,k}}$. 
  
  In the following we assume $\a_1\ne\ep$ and then all $t_i>q+1$. We will inductively construct a word $\tilde{\a}\in {\tilde \P_{m,k}}$ such that $\tilde\a\prec \a$. First we consider the following transformation on $\a$: if $\a_i\in Z_q'$ contains the block $10^{q-1}$, then we replace the last block $10^{q-1}$ in $\a_i$ by $10^q$, and accordingly we replace its right neighboring  word $10^{t_i}$ by $10^{t_i-1}$; otherwise, we do nothing on the block $\a_i 10^{t_i}$. {In particular, if $\a_i=\ep$ then we do nothing on the block $\a_i 10^{t_i}$.} Performing this transformation on $\a$ we get a new word of length $m$ with $k$ ones:
  \begin{equation}\label{eq:T_4}
  \a'=\a_1' 10^{t_1'}\;\a_2' 10^{t_2'}\;\cdots \a_j' 10^{t_j'}\;\a_{j+1}'.
  \end{equation}
  Note that $|\a_i' 10^{t_i'}|=|\a_i 10^{t_i}|$ and  $|\a_i' 10^{t_i'}|_1=|\a_i 10^{t_i}|_1$ for all $1\le i\le j$, and $\a_{j+1}'=\a_{j+1}$. 
  
  {If $\a'$ is not periodic, then we set $T_4(\a):=\L(\a')$. Otherwise, there is a Perron word $\r_1\ldots \r_p$ with each $\r_i=10^{\ell_i}\in\set{10^s: s\ge q-1}$ such that 
  \[\L(\a')=(\r_1\ldots \r_p)^n=(10^{\ell_1}10^{\ell_2}\ldots 10^{\ell_p})^n\] for some $n\ge 2$. If $\L(\a')$ begins with $\r_1=10^{\ell_1}$ for some $\ell_1\ge q+1$, then $\ell_i\ge q+1$ for all $1\le i\le p$, which leads to a contradiction with 
  \begin{equation}\label{eq:T_4-1}
    |\L(\a')|=|\a|=m=(q+1)k=(q+1)|\a|_1=(q+1)|\L(\a')|_1.
  \end{equation}
  If $\r_1=10^q$, then by (\ref{eq:T_4-1}) it follows that $\L(\a')=(10^q)^k$, which implies that the original word $\a$ belongs to $Z_q'$, leading to a contradiction with our hypothesis.  So, we must have $\r_1=10^{q-1}$. Then by (\ref{eq:T_4-1}) it follows that $p\ge 2$, and  there exists $i\in\set{2,3,\ldots,p}$ such that $\ell_i\ge q+1$. Let $v:=\min\set{2\le i\le p: \ell_i\ge q+1}$. Then $\ell_i\in\set{q-1,q}$ for all $i<v$. Note that $\ell_1=q-1$. Let $u:=\max\set{1\le i<v: \ell_i=q-1}$. Then $\ell_i=q$ for all $u<i<v$. Now we set
  \begin{equation}\label{eq:T4-0}
   \begin{split}
     T_4(\a)&:=(\r_1\ldots \r_p)^{n-1}\r_1\ldots \r_{u-1}(10^{\ell_u+1})\r_{u+1}\ldots\r_{v-1}(10^{\ell_v-1})\r_{v+1}\ldots \r_p\\
     &=(\r_1\ldots \r_p)^{n-1}\r_1\ldots \r_{u-1}(10^q)^{v-u}(10^{\ell_v-1})\r_{v+1}\ldots \r_p.
   \end{split}
  \end{equation}
  We claim that $T_4(\a)\in\P_{m,k}$ and $T_4(\a)\prec\a$. 
  
  First we prove $T_4(\a)\in\P_{m,k}$. Note that $T_4(\a)$ has length $m$ with $k$ ones. So it suffices to prove that $T_4(\a)$ is Perron. If $\L(\a')$ is not periodic, then it is clear that $T_4(\a)=\L(\a')$ is a Perron word. Suppose $\L(\a')$ is periodic. Then $T_4(\a)$ is defined as in (\ref{eq:T4-0}). Note that $\r_1\ldots\r_p$ is Perron. Then by Lemma \ref{lem:perron-words},
  \begin{equation}\label{eq:T4-1}
    \r_{i+1}\ldots \r_p\prec \r_1\ldots\r_{p-i}\quad\forall 1\le i<p.
  \end{equation}
  Since $\r_1=10^{q-1}$ and $\ell_v-1\ge q$, by (\ref{eq:T4-0}), (\ref{eq:T4-1}) and Lemma \ref{lem:perron-words} it follows that $T_4(\a)$ is a Perron word. 
  %

Next we prove $T_4(\a)\prec\a$. Note that $T_4(\a)\lle\L(\a')$. Then it suffices to prove $\L(\a')\prec\a$.  By the above argument it follows that $\L(\a')$ must begin with $10^{q-1}$.  Note by (\ref{eq:T_4}) that $t_i'\ge t_i-1\ge  q+1$ for all $i$. Then $\L(\a')$ begins with a suffix of some $\a_i'$, and by our transformation the original word   $\a_i$ also contains the block $10^{q-1}$.} Then we can write it as
  \[
  \a_i=a_1^{(i)}\ldots a_{p_i}^{(i)}a_{p_i+1}^{(i)}\ldots a_{q_i}^{(i)},
  \] 
  where $a_1^{(i)}\ldots a_{p_i}^{(i)}$ ends with $10^{q-1}$, and $a_{p_i+1}^{(i)}\ldots a_{q_i}^{(i)}$   doesn't contain $10^{q-1}$.
  So, by our transformation on $\a_i$ we have
  \begin{equation}\label{eq:v'i}
  \a_i'=a_1^{(i)}\ldots a_{p_i}^{(i)}\,0\, a_{p_i+1}^{(i)}\ldots a_{q_i}^{(i)}.
  \end{equation}
  Now we prove $\L(\a')\prec \a$ in the following two cases. 
  
Case 1. $\L(\a')$ begins with a suffix of $\a_i'$ for some $i\in\set{1,2,\ldots, j}$. Note that $\L(\a')$ begins with $10^{q-1}$. Then by (\ref{eq:v'i}) it follows that $\L(\a')$ begins with 
  $
  a_{\ell+1}^{(i)}\ldots a_{p_i}^{(i)}\, 0\, a_{p_i+1}^{(i)}\ldots a_{q_i}^{(i)}
  $
  for some $\ell<p_i$. By using $\a=a_1\ldots a_m\in \P_{m,k}$ and Lemma \ref{lem:perron-words} it follows that 
  \[{
  a_{\ell+1}^{(i)}\ldots a_{p_i}^{(i)}\,0\prec a_{\ell+1}^{(i)}\ldots a_{p_i}^{(i)}\, 1\lle a_1\ldots a_{p_i-\ell+1},}
  \]
  which implies $\L(\a')\prec\a$. 
  
  Case 2. $\L(\a')$ begins with a suffix of $\a_{j+1}'$. Note that $\a_{j+1}'=\a_{j+1}$. Then $\L(\a')$ also begins with a suffix of $\a\in \P_{m,k}$. So, by Lemma \ref{lem:perron-words} we conclude that  $\L(\a')\prec \a$.

  By Case 1 and Case 2 it follows that $T_4(\a)\lle\L(\a')\prec\a$. Note that $T_4(\a)\in \P_{m,k}$, and   $T_4(\a)$ is closer to $Z_q'$ in Hamming distance.  If $T_4(\a)\notin Z_q'$, then by doing the same transformation for $T_4(\a)$ we get a new word  $T_4^2(\a)\in\P_{m,k}$ with even smaller distance to $Z_q'$. Furthermore, this new word is smaller than $T_4(\a)$.  After performing this transformation finitely many times we obtain a word $\tilde \a\in Z'_q\cap \P_{m,k}={\tilde \P_{m,k}}$ and $\tilde{\a}\prec \a$. This completes the proof.  
\end{proof}

 {We give an example to illustrate how our transformations $T_3$ and $T_4$ defined respectively in Lemmas \ref{lem:lyndon-quasi-balanced} and \ref{lem:perron-quasi-balanced} behave. 
\begin{example}First we consider the transformation $T_3$ as in Lemma \ref{lem:lyndon-quasi-balanced}. Let $m=24, k=6$. Then $q=3$.  Let  $\s=0^410^41010^410^310^21\in\cL_{24,6}\setminus\tilde\cL_{24,6}$.   Performing the transformation   on $\s$ we first get $\s'=(0^410^310^21)^2=\S(\s')$, and then we obtain  
\[T_3(\s)=0^410^310^210^310^310^31\in\tilde\cL_{24,6}.\]

Next we consider the transformation $T_4$ as in Lemma \ref{lem:perron-quasi-balanced}. Let $m=24, k=8$. Then $q=2$. Let   $\a=10101010^3101010^410^4\in\P_{24,8}\setminus\tilde\P_{24,8}.$   Performing the transformation $T_4$ on $\a$ twice  we obtain that 
\begin{align*}
T_4(\a)&=10101010^31010^210^310^4\in \P_{24,8}\setminus\tilde\P_{24,8}, \quad
T_4^2(\a)=10101010^310^210^210^310^3\in\tilde\P_{24,8}.
\end{align*}
\end{example}}

In terms of Lemmas \ref{lem:lyndon-quasi-balanced} and \ref{lem:perron-quasi-balanced}, to determine $\l_{m,k}$ and $\p_{m,k}$ with $m=k(q+1)$ it suffices to consider $\tilde\cL_{m,k}=\cL_{m,k}\cap Y_{q}'$ and $\tilde\P_{m,k}=\P_{m,k}\cap Z_{q}'$. Note that words in $\tilde\cL_{m,k}$ and $\tilde \P_{m,k}$ are {all} $2$-balanced. {Comparing} with the substitutions $\Phi_q$ and $\Psi_q$ defined in (\ref{eq:Phi}) and (\ref{eq:Psi}), we define the following substitutions $\Phi_q'$ and $\Psi_q'$ on $Y_q'$ and $Z_q'$, respectively. 

 Given $q\in\N$, let $\Phi'_q: Y'_q \rightarrow \{0,1,2\}^\ast$ be  the substitution defined by
  \begin{equation}\label{eq:Phi'}
    \Phi'_q(0^{q+1}1)=0,\quad\Phi'_q(0^q1)=1, \quad\Phi'_q(0^{q-1}1)=2.
  \end{equation}
Then for a word $\w=\w_1\w_2\ldots\w_n\in Y'_q$ with each $\w_i\in\{0^{q+1}1, 0^q1, 0^{q-1}1\}$ we have $\Phi'_q(\w) = \Phi'_q(\w_1)\Phi'_q(\w_2)\ldots\Phi'_q(\w_n)\in\{0,1,2\}^n$. 
{Similarly, let} $\Psi'_q: Z'_q\rightarrow \{0,1,2\}^\ast$ be the substitution defined by 
  \begin{equation}\label{eq:Psi'}
  \Psi'_q(10^{q+1})=0,\quad\Psi'_q(10^q)=1,\quad \Psi'_q(10^{q-1})=2.
  \end{equation}
 {In the following we show that $\Phi_q'$ and $\Psi_q'$ share the same properties as $\Phi_q$ and $\Psi_q$, respectively.}

\begin{lemma}
 \label{lem:monotonicity-Phi'-Psi'}
Let $m=k(q+1)$ with $q\in \N$.
\begin{enumerate}[{\rm(i)}]
    \item $\Phi'_q$ is strictly increasing on $Y'_q$. 
   \item $\Psi'_q$ is strictly increasing on $Z'_q$. 
\end{enumerate}
\end{lemma}
\begin{proof}
Since the proof of (ii) is similar to (i),  we only prove (i). Let $\w=\w_1\ldots \w_n, \w'=\w_1'\ldots \w_\ell'\in Y'_q$ with each $\w_i, \w_i'\in\set{0^{q+1}1, 0^q 1, 0^{q-1}1}$. Suppose $\w\prec \w'$. {Then $(\w_i):=\w^\f\prec (\w')^\f=:(\w_i')$,} and  there exists $j\in\N$ such that 
 \[
  \w_1\ldots \w_{j-1}=\w_1'\ldots \w_{j-1}'\quad\textrm{and}\quad \w_j\prec \w_j'.
  \]
   So, by (\ref{eq:Phi'}) it follows that $\Phi'_q(\w_1)\ldots \Phi'_q(\w_{j-1})=\Phi'_q(\w_1')\ldots \Phi'_q(\w_{j-1}')$ and $\Phi'_q(\w_j)<\Phi'_q(\w_j')$. This implies $\Phi'_q(\w)\prec \Phi'_q(\w')$, proving the monotonicity of $\Phi'_q$. 
\end{proof}

Recall the substitution  operator $\bullet$ defined in (\ref{eq:bullet-operator}). For a Lyndon word $\s\in\cL^*$ and
  a set $A\subset\set{0,1}^*$, {let} $\s\bullet A:=\set{\s\bullet\a: \a\in A}$. Write 
\[\cL_k:=\bigcup_{j=1}^{k-1}\cL_{k,j}\quad\textrm{and}\quad \P_k:=\bigcup_{j=1}^{k-1}\P_{k,j}.\]
\begin{proof}
  [Proof of Proposition \ref{prop:extrem-Lyndon-Perron-dividable}]
  {Write $\w:=\w_{\frac{1}{q+1}}=0^q1$.} Note that   $\max\cL_k=01^{k-1}$ and $\min\P_k=10^{k-1}$. By Lemma \ref{lem:property-bullet} it suffices to prove that
  \begin{equation}
    \label{eq:dec20-0}
    \l_{m,k}\in\w\bullet\cL_k\quad\textrm{and}\quad \p_{m,k}\in\w\bullet\P_k.
  \end{equation}
  Note by (\ref{eq:bullet-operator}) and  (\ref{eq:Yq'-Zq'}) that 
  \begin{equation}\label{eq:dec20-1}
  \begin{split}
    \w\bullet\cL_k&={\set{\w^-\L(\w)^{i_1}\L(\w)^+\w^{j_1}\cdots \w^-\L(\w)^{i_\ell}\L(\w)^+ \w^{j_\ell}: \sum_{s=1}^\ell(i_s+j_s+2)=k\textrm{ with }i_s, j_s\geq 0}}.
    \end{split}
  \end{equation}
   Then by Lemma \ref{lem:property-bullet} it follows that each word in $\w\bullet\cL_k$ is a Lyndon word of length $k(q+1)=m$ with precisely $k$ ones. Furthermore, $\w\bullet\cL_k\subset Y_q'$. Thus, \begin{equation}\label{eq:dec20-2}
  \w\bullet\cL_k\subset\cL_{m,k}\cap Y_q'=\tilde\cL_{m,k}.
  \end{equation} Similarly, by (\ref{eq:bullet-operator}) and (\ref{eq:Yq'-Zq'}) we have
    \begin{equation}\label{eq:dec20-3}
    \begin{split}
    \w\bullet\P_k&={\set{\L(\w)^+\w^{i_1}\w^-\L(\w)^{j_1}\cdots \L(\w)^+\w^{i_\ell}\w^-\L(\w)^{j_\ell}:  \sum_{s=1}^\ell(i_s+j_s+2)=k\textrm{ with }i_s, j_s\geq 0}}.
    \end{split}
  \end{equation}
 Note by Lemma \ref{lem:property-bullet} that $\w\bullet\P_k=\w\bullet\L(\cL_k)=\L(\w\bullet\cL_k)$. Furthermore, $\w\bullet\P_k\subset Z_q'$. Then by (\ref{eq:dec20-2}) it follows that 
 \begin{equation}\label{eq:dec20-4}
   \w\bullet\P_k\subset\L(\w\bullet\cL_k)\cap Z_q'\subset\L(\cL_{m,k})\cap Z_q'=\P_{m,k}\cap Z_q'=\tilde{\P}_{m,k}.
 \end{equation} 
 So, in terms of Lemmas \ref{lem:lyndon-quasi-balanced} and \ref{lem:perron-quasi-balanced}, to prove (\ref{eq:dec20-0}) we only need to prove the following two statements.
 \begin{enumerate}[{\rm (i)}]
   \item If $\s\in\tilde\cL_{m,k}\setminus(\w\bullet\cL_k)$, then for any $\tilde{\s}\in \w\bullet\cL_k$ we have $\tilde{\s}\succ\s$;
   \item If $\a\in\tilde{\P}_{m,k}\setminus(\w\bullet\P_k)$, then for any $\tilde{\a}\in \w\bullet\P_k$ we have $\tilde{\a}\prec\a$. 
 \end{enumerate} 
 
 First we prove (i). By (\ref{eq:Phi'}) and (\ref{eq:dec20-1}) it follows that 
 \begin{equation}\label{eq:dec20-5}
 \Phi_q'(\w\bullet\cL_k)={\set{01^{i_1}21^{j_1}\cdots 01^{i_\ell}21^{j_\ell}:  \sum_{s=1}^\ell(i_s+j_s+2)=k \textrm{ with }i_s, j_s\geq 0}}.
 \end{equation}
 Take $\s=\s_1\ldots\s_k\in\tilde{\cL}_{m,k} \setminus(\w\bullet\cL_k)$ with each $\s_i\in\{0^{q+1}1,0^q1,0^{q-1}1\}$. Then by (\ref{eq:dec20-5}) $\Phi_q'(\s)$ must contain  the block $01^i0$ or $21^i2$ for some $i\ge 0$. If $\Phi_q'(\s)$ contains $01^i0$ for some $i\ge 0$, then by (\ref{eq:dec20-5}) it follows that for any $\tilde\s\in\w\bullet\cL_k$ we have
 \[
 \Phi_q'(\s)\lle 01^i0\b\prec \Phi_q'(\tilde\s)
 \]
 for some $\b\in\set{0,1,2}^*$. Since $\Phi_q'$ is strictly increasing by Lemma \ref{lem:monotonicity-Phi'-Psi'}, we obtain $\s\prec\tilde\s$.
 Now suppose $\Phi_q'(\s)$ contains $21^i2$ for some $i\ge 0$. Note that $m=k(q+1)$ and  $\s\in\tilde\cL_{m,k}$ is $2$-balanced. Then $\Phi_q'(\s)$ contains the same number of $0$s and $2$s. This implies that $\Phi_q'(\s)$ also contains $01^j0$ for some $j\ge 0$, and then by the same argument as above we can deduce that $\s\prec\tilde\s$ for any $\tilde\s\in\w\bullet\cL_k$. This establishes (i).
 
 Next we prove (ii). Note by (\ref{eq:Psi'}) and (\ref{eq:dec20-3}) that 
 \begin{equation}\label{eq:dec20-6}
   \Psi_q'(\w\bullet\P_k)={\set{21^{i_1}01^{j_1}\ldots21^{i_\ell}01^{j_\ell}: \sum_{s=1}^\ell(i_s+j_s+2)=k \textrm{ with }i_s, j_s\geq 0}}.
 \end{equation}
 Take $\a=\a_1\ldots\a_k\in\tilde{\P}_{m,k} \setminus(\w\bullet\P_k)$ with each $\a_i\in\{10^{q+1},10^q,10^{q-1}\}$. Then by (\ref{eq:dec20-6}) $\Psi_q'(\a)$ must contain  the block $21^i2$ or $01^i0$ for some $i\ge 0$. If $\Psi_q'(\a)$ contains $21^i2$ for some $i\ge 0$, {then by (\ref{eq:dec20-6}) it follows that for any $\tilde\a\in\w\bullet\P_k$} we have
 \[
 \Psi_q'(\a)\lge 21^i2\c\succ \Psi_q'(\tilde\a)
 \]
 for some $\c\in\set{0,1,2}^*$. Since $\Psi_q'$ is strictly increasing by Lemma \ref{lem:monotonicity-Phi'-Psi'}, we obtain $\a\succ\tilde\a$.
 Now suppose $\Psi_q'(\a)$ contains $01^i0$ for some $i\ge 0$. Note that $m=k(q+1)$ and $\a\in\tilde\P_{m,k}$ is $2$-balanced. Then $\Psi_q'(\a)$ contains the same number of $0$s and $2$s. This implies that $\Psi_q'(\a)$ also contains $21^j2$ for some $j\ge 0$, and then by the same argument as above we can deduce that $\a\succ\tilde\a$ for any {$\tilde\a\in\w\bullet\P_k$}. This proves (ii).
\end{proof}

{Next we consider $\gcd(m,k)=k_*\in(1,k)$.
 For $m,k\in\N_{\ge 2}$ with $1<k<m$, we recall from (\ref{eq:balanced-Lyndon-Perron}) and (\ref{eq:quasi-balanced-Lyndon-Perron}) that 
 \[
 \tilde {\cL}_{m,k} =\left\{\begin{array}{ccc}
                          {\cL}_{m,k}\cap Y_{\lfloor\frac{m}{k}\rfloor-1} & \textrm{if} & k\nmid m, \\
                          {\cL}_{m,k}\cap Y'_{\lfloor\frac{m}{k}\rfloor-1} & \textrm{if} & k\mid m;
                        \end{array}\right.\quad\textrm{and}\quad 
                        \tilde \P_{m,k} =\left\{\begin{array}{ccc}
                          \P_{m,k}\cap Z_{\lfloor\frac{m}{k}\rfloor-1} & \textrm{if} & k\nmid m, \\
                          \P_{m,k}\cap Z'_{\lfloor\frac{m}{k}\rfloor-1} & \textrm{if} & k\mid m.
                        \end{array}\right.
 \]
 Recall $\Phi_q$ and $\Psi_q$ in (\ref{eq:Phi}) and (\ref{eq:Psi}), respectively.
 Motivated by the proof of Theorem \ref{main-2:L-P-mk} (i), the extremal words $\l_{m,k}$ and $\p_{m,k}$ can be reduced to $\l_{m_*,k_*}$ and $\p_{m_*,k_*}$ {with $k_*|m_*$ respectively}, and then we can apply Proposition \ref{prop:extrem-Lyndon-Perron-dividable} to determine $\l_{m,k}$ and $\p_{m,k}$.

 \begin{lemma}\label{lem:simplification-w}
   Let $m,k\in\N_{\ge 2}$ with $1<k<m$ and $k_*:=\gcd(m,k)$. If $1<k_*<k$, then 
   there exist $q_0, q_1,\ldots, q_\ell\in\N_{0}$ and $m_*>k_*$ with $k_*\mid m_*$ such that  
     \[
     \Phi_{q_\ell}\circ\Phi_{q_{\ell-1}}\circ\cdots \circ\Phi_{q_0}(\l_{m,k})=\max \tilde {\cL}_{m_*, k_*},
     \]
 and
     \[
     \Psi_{q_\ell}\circ\Psi_{q_{\ell-1}}\circ\cdots \circ\Psi_{q_0}(\p_{m,k})=\min\tilde \P_{m_*, k_*}.
     \]
 
 \end{lemma}
 \begin{proof}
 We will use Euclid's algorithm on long division. 
     Let $k_{-1}=m$ and $k_0=k$. Since $\gcd(k_{-1}, k_0)=\gcd(m,k)=k_*$ and $k_*<k_0<k_{-1}$, there exist $q_0\in\N_0$ and $r_0\in\set{1,\ldots, k_0-1}$ such that 
   \[
   k_{-1}=(q_0+1)k_0+r_0.
   \]
   By Lemmas \ref{lem:lyndon-balanced} and \ref{lem:property-Phi-Psi} it follows that 
   \begin{equation}\label{eq:oct-6-1}
     \Phi_{q_0}(\l_{k_{-1}, k_0})=\Phi_{q_0}(\max\tilde {\cL}_{k_{-1}, k_0})=\max\Phi_{q_0}(\tilde {\cL}_{k_{-1},k_0})=\max {\cL}_{k_0,k_1},
   \end{equation}
   where $k_1:=k_0-r_0$. Similarly, by Lemmas \ref{lem:perron-balanced} and \ref{lem:property-Phi-Psi} we obtain that 
   \begin{equation}\label{eq:oct-6-2}
     \Psi_{q_0}(\p_{k_{-1}, k_0})=\Psi_{q_0}(\min\tilde \P_{k_{-1}, k_0})=\min\Psi_{q_0}(\tilde \P_{k_{-1},k_0})=\min \P_{k_0,k_1}.
   \end{equation}
   Note that $\gcd(k_0,k_1)=\gcd(k_0, r_0)=\gcd(k_{-1},k_0)=k_*$. If $k_1=k_*$, then we are done by taking $m_*=k_0$ and using Lemmas \ref{lem:lyndon-quasi-balanced} and \ref{lem:perron-quasi-balanced} in (\ref{eq:oct-6-1}) and (\ref{eq:oct-6-2}) that $\max {\cL}_{k_0,k_1}=\max\tilde {\cL}_{k_0, k_1}$ and $\min \P_{k_0,k_1}=\min\tilde \P_{k_0,k_1}$.
   
   If $k_1\ne k_*$, then $k_*<k_1$ and we can find $q_1\in\N_0$ and $r_1\in\set{1,\ldots, k_1-1}$ such that 
   \[
   k_0=(q_1+1)k_1+r_1.
   \]By Lemmas \ref{lem:lyndon-balanced}, \ref{lem:property-Phi-Psi} and (\ref{eq:oct-6-1}) it follows that 
   \begin{equation}\label{eq:oct-6-3}
     \Phi_{q_1}\circ\Phi_{q_0}(\l_{k_{-1}, k_0})=\Phi_{q_1}(\max \tilde{\cL}_{k_{0}, k_1})=\max\Phi_{q_1}(\tilde {\cL}_{k_0,k_1})=\max {\cL}_{k_1,k_2},
   \end{equation}
   where $k_2:=k_1-r_1$. Similarly, by Lemmas \ref{lem:perron-balanced}, \ref{lem:property-Phi-Psi} and (\ref{eq:oct-6-2}) we obtain that 
   \begin{equation}\label{eq:oct-6-4}
   { \Psi_{q_1}\circ\Psi_{q_0}(\p_{k_{-1}, k_0})}=\Psi_{q_1}(\min  \tilde\P_{k_0, k_1})=\min\Psi_{q_1}(\tilde \P_{k_0,k_1})=\min \P_{k_1,k_2}.
   \end{equation}
   Note that $\gcd(k_1,k_2)=\gcd(k_1,r_1)=\gcd(k_0,k_1)=k_*$. If $k_2=k_*$, then we are done by taking $m_*=k_1$ and using Lemmas \ref{lem:lyndon-quasi-balanced} and \ref{lem:perron-quasi-balanced} in (\ref{eq:oct-6-3}) and (\ref{eq:oct-6-4}) respectively. 
   
   If $k_2\ne k_*$, then $k_*<k_2$, and we can continue the above argument. After finitely many steps, we can find $(q_i,k_i,r_i)\in\N_0^3$ with $i=0,1,\ldots, \ell$ such that 
   \[
   k_{i-1}=(q_i+1)k_i+r_i, \quad i=0,1,\ldots, \ell,
   \]
   where $k_{i+1}=k_i-r_i$ and $0<r_i<k_i-k_*$ for all $0\le i<\ell$, and $r_\ell=k_\ell-k_*$.
   So,
   \[
   \Phi_{q_\ell}\circ\Phi_{q_{\ell-1}}\circ\cdots\circ\Phi_{q_0}(\l_{k_{-1},k_0})=\max {\cL}_{k_\ell,k_\ell-r_\ell}=\max {\cL}_{k_\ell,k_*}.
   \]
   Since $\gcd(k_\ell,k_*)=\gcd(k_{\ell-1},k_\ell)=\cdots=\gcd(k_{-1},k_0)=k_*$, we have $k_*\mid k_{\ell}$. By Lemma \ref{lem:lyndon-quasi-balanced} and taking $m_*=k_\ell$ we obtain that 
   \[\Phi_{q_\ell}\circ\cdots\circ\Phi_{q_0}(\l_{m,k})=\max\tilde {\cL}_{m_*,k_*}\] as desired. Similarly, we can deduce that $\Psi_{q_\ell}\circ\cdots\circ\Psi_{q_0}(\p_{m,k})=\min\tilde \P_{m_*,k_*}$.
 \end{proof}}

\begin{proof}[Proof of Theorem \ref{main-2:L-P-mk} (ii)]
  Note by Lemma \ref{lem:simplification-w} that
  \begin{equation}\label{eq:29-1}
    \Phi_{q_\ell}\circ\Phi_{q_{\ell-1}}\circ\cdots\circ\Phi_{q_0}(\l_{m,k})=\max\tilde {\cL}_{m_*,k_*},\quad\Psi_{q_\ell}\circ\Psi_{q_{\ell-1}}\circ\cdots\circ\Psi_{q_0}(\p_{m,k})=\min\tilde \P_{m_*,k_*}
  \end{equation}
%
  for some $q_0,q_1,\ldots, q_\ell\in\N_0$ and $m_*>k_*$ with $k_*\mid m_*$. Write $q_*=m_*/k_*$. Then $q_*>1$, and  by Proposition \ref{prop:extrem-Lyndon-Perron-dividable} and Lemma \ref{lem:lyndon-quasi-balanced} it follows that 
  \begin{equation}\label{eq:29-3}
    \max\tilde {\cL}_{m_*,k_*}=\max {\cL}_{m_*,k_*} =\w_{1/q_*}\bullet 01^{k_*-1}.
  \end{equation}
  Similarly, by Lemma \ref{lem:perron-quasi-balanced} and Proposition \ref{prop:extrem-Lyndon-Perron-dividable} we have 
  \begin{equation}\label{eq:29-4}
    \min\tilde \P_{m_*,k_*}=\min \P_{m_*,k_*}= \w_{1/q_*}\bullet 10^{k_*-1}.
  \end{equation}
  Note  by Lemma \ref{lem:U-Phi-Farey-word}  that $\Phi_q^{-1}=U_{0^q 1}$. So, by (\ref{eq:29-1}), (\ref{eq:29-3}) and Lemma \ref{lem:commute-U-bullet} we obtain that 
  \begin{equation}\label{eq:29-5}
  \begin{split}
    \l_{m,k} & =\Phi_{q_0}^{-1}\circ\Phi_{q_1}^{-1}\circ\cdots\circ\Phi_{q_\ell}^{-1}(\w_{1/q_*}\bullet {01^{k_*-1}}) \\
     &= \Phi_{q_0}^{-1}\circ\Phi_{q_1}^{-1}\circ\cdots\circ\Phi_{q_\ell}^{-1}(\w_{1/q_*})\bullet {01^{k_*-1}}.
  \end{split}
  \end{equation}
  Note by the proof of Lemma \ref{lem:simplification-w} that $\w_{1/q_*}=\w_{k_*/m_*}=\w_{\frac{k_\ell-r_\ell}{k_\ell}}$, and
  \[
  k_{-1}=m,\quad k_0=k,\quad\quad\textrm{and}\quad k_{i-1}=(q_i+1)k_i+r_i  \quad\textrm{for}\quad i=0,1,\ldots, \ell.
   \]
 Then by Lemma \ref{lem:U-Phi-Farey-word} and (\ref{eq:29-5}) it follows that 
   \begin{align*}
     \l_{m,k}&=\Phi_{q_0}^{-1}\circ\cdots\circ\Phi_{q_{\ell-1}}^{-1}\circ\Phi_{q_\ell}^{-1}(\w_{\frac{k_\ell-r_\ell}{k_\ell}})\bullet {01^{k_*-1}}\\
     &=\Phi_{q_0}^{-1}\circ\cdots\circ\Phi_{q_{\ell-1}}^{-1}(\w_{\frac{k_\ell}{k_{\ell-1}}})\bullet {01^{k_*-1}}\\
     &=\cdots\\
     &=\w_{\frac{k_0}{k_{-1}}} \bullet 10^{k_*-1}=\w_{\frac{k}{m}}\bullet {01^{k_*-1}}.
   \end{align*}

 Next we consider $\p_{m,k}$. Note by (\ref{eq:29-1}) and (\ref{eq:29-4}) that 
  \begin{equation}\label{eq:29-1'}
  \begin{split}
    \p_{m,k} & =\Psi_{q_0}^{-1}\circ\Psi_{q_1}^{-1}\circ\cdots\circ\Psi_{q_\ell}^{-1}(\w_{1/q_*}\bullet 10^{k_*-1}) \\
      & =\Psi_{q_0}^{-1}\circ\Psi_{q_1}^{-1}\circ\cdots\circ\Psi_{q_\ell}^{-1}\left({\L}(\w_{1/q_*}\bullet 0^{k_*-1}1)\right),
  \end{split}
  \end{equation}
  where the second equality follows by Lemma \ref{lem:property-bullet} and $10^{k_*-1}=\L(0^{k_*-1}1)$. Observe by Lemma \ref{lem:commute-S-L-Phi-Psi} that for any $\w\in\cL^*$ we have $
  {\L}(\w)={\L}(\Phi_q(\Phi_q^{-1}(\w)))=\Psi_q({\L}(\Phi_q^{-1}(\w))),
  $
  which gives 
  \[\Psi_q^{-1}({\L}(\w))={\L}(\Phi_q^{-1}(\w))\quad\forall ~\w\in\cL^*.\]
  Applying this successively to (\ref{eq:29-1'}) and the above argument we obtain that 
  \begin{align*}
    \p_{m,k} & ={\L}\left(\Phi_{q_0}^{-1}\circ\cdots \circ\Phi_{q_\ell}^{-1}(\w_{1/q_*}\bullet 0^{k_*-1}1)\right) \\
     & = {\L}\left(\Phi_{q_0}^{-1}\circ\cdots \circ\Phi_{q_\ell}^{-1}(\w_{1/q_*})\bullet 0^{k_*-1}1\right) \\
     &=\L\left(\w_{\frac{k}{m}}\bullet 0^{k_*-1}1\right) =\w_{\frac{k}{m}}\bullet 10^{k_*-1},
  \end{align*}
  where the last equality follows by Lemma \ref{lem:property-bullet}. This completes the proof.
\end{proof}

 \section{Critical value $\tau_m$}\label{sec:critical-value-tau}
 In this section we will prove the main result Theorem \ref{main-3:critical-value-details}. Recall that for $m\in\N_{\ge 2}$, a  vector $(m, k_1, k_2, \ldots, k_j)\in\N^*$ is   an admissible $m$-chain if $\gcd(m,k_1,k_2,\ldots, k_j)=1$ and 
 \[
  k_i<m_i:=\gcd(m,k_1,\ldots, k_{i-1})\quad\forall 1\le i\le j,
 \]
 where   $m_1=m$. Furthermore, recall by (\ref{eq:partition-base}) that for an admissible $m$-chain $(m,k_1,\ldots,k_j)$ the unique base $\beta_{m,k_1,\ldots, k_j}\in(1,2]$ satisfies 
 $
 \delta(\beta_{m,k_1,\ldots, k_j})=\L(\w_{k_1/m_1}\bullet\w_{k_2/m_2}\bullet\cdots\bullet\w_{k_j/m_j})^\f.
$
 First we determine  the critical value $\tau_m(\beta)$ for $\beta$ in the two extremal $m$-partition intervals $(1,\beta_{m,1}]$ and $(\beta_{m,m-1},2]$.
 \begin{lemma}
   \label{lem:critical-value-first-last-interval}
   Let $m\in\N_{\ge 2}$. The following statements hold.
   \begin{enumerate}[{\rm(i)}]
     \item If $\beta\in(1,\beta_{m,1}]$, then $\tau_m(\beta)=0$.
     \item  If $\beta\in(\beta_{m,m-1},2]$, then $\tau_m(\beta)=((01^{m-1})^\f)_\beta=\frac{\beta^{m-1}-1}{(\beta-1)(\beta^m-1)}$. 
   \end{enumerate}
 \end{lemma}
 \begin{proof}
   Note that 
   \begin{equation}\label{eq:dec29-1}
   \de(\beta_{m,1})=(10^{m-1})^\f\quad \textrm{and}\quad \de(\beta_{m,m-1})=(1^{m-1}0)^\f.
    \end{equation}
    First we prove (i). Take $\beta\in(1,\beta_{m,1}]$. Then by (\ref{eq:dec29-1}) and Lemma \ref{lem:delta-beta} we have $\de(\beta)\lle (10^{m-1})^\f$. It suffices to prove that the following set 
   \[
   \K_\beta(0)=\set{(d_i)\in\set{0,1}^\N: 0^\f\lle\si^n((d_i))\prec \de(\beta)~\forall n\ge 0}
   \]
   contains no periodic sequences of smallest period $m$. Suppose on the contrary that $(d_1\ldots d_m)^\f\in\K_\beta(0)$ for some non-periodic block $d_1\ldots d_m$. Then $d_1\ldots d_m$ contains at least one   digit $1$, and thus we can find some $n\ge 0$ such that $\de(\beta)\succ \si^n((d_1\ldots d_m)^\f)\lge (10^{m-1})^\f$, leading to a contradiction with our assumption. So, $\tau_m(\beta)=0$.
   
   Next we prove (ii). Take $\beta\in(\beta_{m,m-1},2]$. Then by (\ref{eq:dec29-1}) and Lemma \ref{lem:delta-beta} we have $\de(\beta)\succ(1^{m-1}0)^\f$. Clearly, $(01^{m-1})^\f\lle\si^n((01^{m-1})^\f)\prec \de(\beta)$ for all $n\ge 0$. This implies $\tau_m(\beta)\ge((01^{m-1})^\f)_\beta$. To prove the reverse inequality we take an arbitrary $t>((01^{m-1})^\f)_\beta$ with its greedy $\beta$-expansion $(t_i)$, and it suffices to prove $\tau_m(\beta)\le t$.
   
   Suppose on the contrary that $\tau_m(\beta)>t$, and then there exists a non-periodic block $d_1\ldots d_m\in\set{0,1}^m$ such that 
   \[
   (t_i)\lle \si^n((d_1\ldots d_m)^\f)\prec \de(\beta)\quad\forall n\ge 0.
   \]
   Thus, we can find $n\ge 0$ such that $(01^{m-1})^\f\lge \si^n((d_1\ldots d_m)^\f)\lge (t_i)$. By Lemma \ref{lem:greedy-expansion} this implies  $t=((t_i))_\beta\le ((01^{m-1})^\f)_\beta$, leading to a contradiction with our assumption $t>((01^{m-1})^\f)_\beta$. This completes the proof.
 \end{proof}
 
In terms of the partition in (\ref{eq:partition-intervals}), we consider the critical value $\tau_m(\beta)$ for $\beta$ in {any other non-extremal   $m$-partition interval $I_{m,k_1,\ldots,k_j}$, where  $(m,k_1,\ldots,k_j)\in\A_m\setminus\set{(m,m-1)}$ is an admissible $m$-chain with $\N\ni k_i<m_i=\gcd(m,k_1,\ldots,k_{i-1})$ for $1\le i\le j$.} In view of the butterfly tree $\mathcal T_m$ described in Section \ref{sec: Introduction}, we split these   {non-extremal $m$-partition} intervals into the following four types (see Example \ref{ex:four-types-partition-intervals}):
  \begin{enumerate}[{\rm(i)}]
    \item  Type A partition intervals have the form 
    \[I_{m,k_1,\ldots,k_j}^A=(\beta_{m,k_1,\ldots,k_j}, \beta_{m,k_1,\ldots,k_j^+}]\quad\textrm{with}\quad \gcd(k_j,m_j)=1,\quad \gcd(k_j^+,m_j)=1,\]
    where $k^+:=k+1$ for any $k\in\N$.
    \item  Type B partition intervals have the form 
    \[I_{m,k_1,\ldots,k_j}^B=(\beta_{m,k_1,\ldots,k_j},  \beta_{m,k_1,\ldots, k_j^+,1}]\quad\textrm{with}\quad \gcd(k_j,m_j)=1,\quad\gcd(k_j^+,m_j)>1.\]
    
    \item  Type C partition intervals have the form 
    \[I_{m,k_1,\ldots,k_j}^C=(\beta_{m,k_1,\ldots,k_{j}}, \beta_{m,k_1,\ldots,k_{j-1}^+}]\quad\textrm{with}\quad k_j=m_j-1,\quad\gcd(k_{j-1}^+,m_{j-1})=1.\]
    
    \item Type D partition intervals have the form 
    \[I_{m,k_1,\ldots,k_j}^D=(\beta_{m,k_1,\ldots,k_j}, \beta_{m,k_1,\ldots,k_{j-1}^+,1}]\quad\textrm{with}\quad k_j=m_j-1,\quad\gcd(k_{j-1}^+, m_{j-1})>1.\]
  \end{enumerate}
    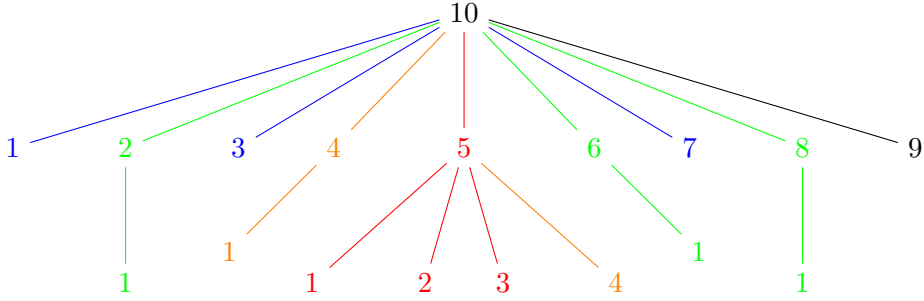
\begin{figure}
   [h!]
   \begin{tikzpicture}[level distance=18mm]
   \node{$10$}
   child[blue]{node[blue]{$1$}}
   child[green]{node[green]{$2$}
 child[green]{node[green]{$1$}}
 }
 child[blue]{node[blue]{$3$}}
child[orange,left]{node[orange]{$4$}
child[orange,rotate=-40]{node[orange]{$1$}}
}
child[red]{node[red]{$5$}
child[red,right]{node[red]{$1$}}
child[red,right]{node[red]{$2$}}
child[red,left]{node[red]{$3$}}
child[orange,left]{node[orange]{$4$}}
}
child[green,right]{node[green]{$6$}
child[green,rotate=40]{node[green]{$1$}}
}
child[blue]{node[blue]{$7$}}
child[green]{node[green]{$8$}
child[green]{node[green]{$1$}}
}
child{node{$9$}}
;

\end{tikzpicture}
 
\caption{The butterfly tree $\mathcal T_{10}$. The red paths correspond to Type $A$ partition intervals; the blue paths correspond to Type $B$ partition intervals; the green paths correspond to Type $C$ partition intervals; and the {orange} paths correspond to Type $D$ partition intervals.}\label{fig:tree-10}
\end{figure}
  \begin{example}\label{ex:four-types-partition-intervals}
   Let $m=10$.  In view of Figure \ref{fig:tree-10}, the non-extremal partition intervals are classified as follows. \\
   (i) The Type A partition intervals are 
    \[
    I_{10,5,1}^A=(\beta_{10,5,1}, \beta_{10,5,2}],\quad I_{10,5,2}^A=(\beta_{10,5,2}, \beta_{10,5,3}],\quad I_{10,5,3}^A=(\beta_{10,5,3}, \beta_{10,5,4}].
    \](ii) The Type B partition intervals are 
    \[
    I_{10,1}^B=(\beta_{10,1}, \beta_{10,2,1}],\quad I_{10,3}^B=(\beta_{10,3}, \beta_{10,4,1}],\quad I_{10,7}^B=(\beta_{10,7}, \beta_{10,8,1}].
    \] (iii) The Type C partition intervals are 
    \[
    I_{10,2,1}^C=(\beta_{10,2,1},\beta_{10,3}],\quad I_{10,6,1}^C=(\beta_{10,6,1}, \beta_{10,7}],\quad I_{10,8,1}^C=(\beta_{10,8,1}, \beta_{10,9}].
    \] (iv) The Type D partition intervals are 
    \[
    I_{10,4,1}^D=(\beta_{10,4,1}, \beta_{10,5,1}],\quad I_{10,5,4}^D=(\beta_{10,5,4}, \beta_{10,6,1}].
    \]
  \end{example}
  
   We point out that for any Type A and Type B partition intervals  $I^A_{m,k_1,\ldots,k_j}, I^B_{m,k_1,\ldots,k_j}$ we have  $k_j<m_j-1$.   Furthermore, for any Type C and Type D partition intervals $I^C_{m,k_1,\ldots,k_j}, I^D_{m,k_1,\ldots,k_j}$ we have  $k_{j-1}<m_{j-1}-1$, since otherwise $m_j=1$ and thus $(m,k_1,\ldots,k_j)$ can not be an admissible $m$-chain. {We will determine $\tau_m(\beta)$ for $\beta$ in the following two classes separately.} (I)   $\beta$ in Type A and Type B partition intervals; (II) $\beta$ in Type C and Type D partition intervals.

  \subsection{Critical value $\tau_m$ in Type A and Type B partition intervals} Suppose that $I^A_{m,k_1,\ldots,k_j}=(\beta_{m,k_1,\ldots,k_j}, \beta_{m,k_1,\ldots,k_j^+}]$ is a Type A partition interval. Then $\gcd(m_j,k_j)=\gcd(m_j,k_j^+)=1$, which implies $k_j<m_j-1$. Note that $(m,k_1,\ldots,k_j)$ and $(m,k_1,\ldots,k_j^+)$ are two neighboring admissible $m$-chains, which means there is no admissible $m$-chains between them in the lexicographical ordering. Similarly, let $I^B_{m,k_1,\ldots, k_j}=(\beta_{m,k_1,\ldots,k_j},  \beta_{m,k_1,\ldots, k_j^+,1}]$ be a Type B partition interval. Then $\gcd(m_j,k_j)=1$ and $\gcd(m_j,k_j^+)>1$. In this case, $(m,k_1,\ldots,k_j)$ and $(m,k_1,\ldots,k_j^+,1)$ are   neighboring admissible $m$-chains.
  
 \begin{lemma}
   \label{lem:partition-base-order}
   The Type A and Type B partition intervals are well-defined. 
 \end{lemma}
 \begin{proof}
    Let $I_{m,k_1,\ldots,k_j}^A=(\beta_{m,k_1,\ldots,k_j}, \beta_{m,k_1,\ldots,k_j^+}]$ be a Type A partition interval. Then $\gcd(k_j,m_j)=\gcd(k_j^+,m_j)=1$. By (\ref{eq:partition-base}) and Lemma \ref{lem:property-bullet} it follows that 
 \begin{align*}
   \de(\beta_{m,k_1,\ldots,k_j}) & =\L(\w_{k_1/m_1}\bullet\cdots\bullet\w_{k_{j-1}/m_{j-1}}\bullet\w_{k_j/m_j})^\f\\
   &=\left(\w_{k_1/m_1}\bullet\cdots\bullet\w_{k_{j-1}/m_{j-1}}\bullet\L(\w_{k_j/m_j})\right)^\f,
   \end{align*}
   and
   \begin{align*}
    \de(\beta_{m,k_1,\ldots,k_j^+}) & =\L(\w_{k_1/m_1}\bullet\cdots\bullet\w_{k_{j-1}/m_{j-1}}\bullet{\w_{k_j^+/m_j}})^\f\\
    &=\left(\w_{k_1/m_1}\bullet\cdots\bullet\w_{k_{j-1}/m_{j-1}}\bullet\L({\w_{k_j^+/m_j}})\right)^\f.
 \end{align*}
 By Lemmas \ref{lem:delta-beta} and \ref{lem:property-bullet}, to prove $\beta_{m,k_1,\ldots,k_j}<\beta_{m,k_1,\ldots,k_j^+}$ it suffices to prove $\L(\w_{k_j/m_j})\prec \L(\w_{k_j^+/m_j})$. Since $\gcd(k_j,m_j)=1=\gcd(k_j^+,m_j)$, by Lemma \ref{lem:monotonicity-lyndon-perron-m-k} and Theorem \ref{main-2:L-P-mk} it follows that 
 \[
 \L(\w_{k_j/m_j})=\p_{m_j,k_j}\prec\p_{m_j,k_j^+}=\L(\w_{k_j^+/m_j}).
 \]
 This proves $\beta_{m,k_1,\ldots,k_j}<\beta_{m,k_1,\ldots,k_j^+}$, and so Type A partition intervals are well-defined. 
 
 Next, let $(\beta_{m,k_1,\ldots,k_j},  \beta_{m,k_1,\ldots, k_j^+,1}]$ be a Type B partition interval. Then $\gcd(m_j,k_j)=1$ and $\gcd(m_j,k_j^+)=d>1$. By the same argument as above we obtain that
 \begin{align*}
   \de(\beta_{m,k_1,\ldots,k_j}) & =(\w_{k_1/m_1}\bullet\cdots\bullet\w_{k_{j-1}/m_{j-1}}\bullet\L(\w_{k_j/m_j}))^\f,\\
    \de(\beta_{m,k_1,\ldots,k_j^+,1}) & =(\w_{k_1/m_1}\bullet\cdots\bullet\w_{k_{j-1}/m_{j-1}}\bullet\L(\w_{k_j^+/m_j}\bullet\w_{1/d}))^\f,
 \end{align*}
 where we set $\w_{k_j^+/m_j}=\w_{\frac{k_j^+/d}{m_j/d}}$.
 By Lemma \ref{lem:delta-beta} and Lemma \ref{lem:property-bullet} it suffices to prove $\L(\w_{k_j/m_j})\prec \L(\w_{k_j^+/m_j}\bullet\w_{1/d})$. Since $\gcd(k_j,m_j)=1$ and $\gcd(k_j^+,m_j)=d>1$, by Theorem \ref{main-2:L-P-mk} and Lemma \ref{lem:monotonicity-lyndon-perron-m-k} it follows that \[
 \L(\w_{k_j/m_j})=\p_{m_j,k_j}\prec \p_{m_j,k_j^+}=\L(\w_{k_j^+/m_j}\bullet\w_{1/d}),
 \]
 which implies $\L(\w_{k_j/m_j})\prec \L(\w_{k_j^+/m_j}\bullet\w_{1/d})$. This proves that Type B partition intervals are well defined. 
 \end{proof}

 \begin{proposition}
   \label{prop:critical-value-AB}
   \begin{enumerate}[{\rm(i)}]
   \item Let $I^A_{m,k_1,\ldots,k_j}=(\beta_{m,k_1,\ldots,k_j}, \beta_{m,k_1,\ldots,k_j^+}]$ be a Type A partition interval. Then for any $\beta\in I^A_{m,k_1,\ldots, k_j}$ we have
   \[
   \tau_m(\beta)=((\w_{k_1/m_1}\bullet\w_{k_2/m_2}\bullet\cdots\bullet\w_{k_j/m_j})^\f)_\beta,
   \]
   where $m_i=\gcd(m,k_1,\ldots,k_{i-1})$ for $1\le i\le j$. 
   
   \item Similarly, let $I^B_{m,k_1,\ldots, k_j}=(\beta_{m,k_1,\ldots,k_j},  \beta_{m,k_1,\ldots, k_j^+,1}]$ be a Type B partition interval. Then for any $\beta\in I^B_{m,k_1,\ldots,k_j}$ we have 
       \[
   \tau_m(\beta)=((\w_{k_1/m_1}\bullet\w_{k_2/m_2}\bullet\cdots\bullet\w_{k_j/m_j})^\f)_\beta.
   \]
   \end{enumerate}
 \end{proposition}
 The proof of Proposition \ref{prop:critical-value-AB} will be split into several lemmas. 
 First, we recall the following result on Farey words (see \cite[Proposition 4.4]{Kalle-Kong-Langeveld-Li-18}).
\begin{lemma} 
  \label{lem:finite-set}
  Let $\s=s_1\ldots s_m$ be a Farey word. Then  the following set
  \[
  \Ga_\s:=\set{(d_i): \s^\f\lle \si^n((d_i))\lle {\L}(\s)^\f~\forall n\ge 0}
  \]
  is finite. In fact, we have  $\Ga_\s=\set{\si^n(\s^\f): n=0,1,\ldots, m-1}$. 
\end{lemma}
The following renormalization result was essentially proved in \cite[Lemma 7.2]{Allaart-Kong-2026}. For completeness and the convenience of   readers we include a proof.
\begin{lemma}
  \label{lem:renormalization}
  Let $\s\in\cL^*$ and $l\in\N_{\ge 2}$. If $(d_i)\in\set{0,1}^\N$ contains either $\s^-$ or $\L(\s)^+$, and it satisfies
  \[
  (\s\bullet 0^{l-1} 1)^\f\lle\si^n((d_i))\lle (\s\bullet 1^{l-1}0)^\f\quad\forall n\ge 0,
  \]
 then $(d_i)$ ends with $\s\bullet(z_i)$ for some $(z_i)\in\set{0,1}^\N$.
\end{lemma}
\begin{proof}
Let $q=|\s|$. Suppose without loss of generality that  $d_{n_0+1}\ldots d_{n_0+q}=\s^-$ for some $n_0\ge 0$. 
Note by (\ref{eq:bullet-operator})   that 
$\s\bullet 0^{l-1}1=\s^-{\L}(\s)^{l-2}{\L}(\s)^+$ and $\s\bullet1^{l-1}0={\L}(\s)^+\s^{l-2}\s^-.$ Then the sequence $(d_i)$ satisfies 
  \begin{equation}\label{eq:nov29-1}
    (\s^-{\L}(\s)^{l-2}{\L}(\s)^+)^\f\lle\si^n((d_i))\lle ({\L}(\s)^+\s^{l-2}\s^-)^\f\quad\forall n\ge 0.
  \end{equation}
Since   $d_{n_0+1}\ldots d_{n_0+q}=\s^-$, by (\ref{eq:nov29-1}) it follows that 
\begin{equation}\label{eq:nov29-2}
  ({\L}(\s)^{l-2}{\L}(\s)^+\s^-)^\f\lle d_{n_0+q+1}d_{n_0+q+2}\ldots \lle ({\L}(\s)^+\s^{l-2}\s^-)^\f,
\end{equation}
which implies $d_{n_0+q+1}\ldots d_{n_0+2q}\in\set{{\L}(\s), {\L}(\s)^+}$. If $d_{n_0+q+1}\ldots d_{n_0+2q}={\L}(\s)$, then $l\ge 3$, and by (\ref{eq:nov29-1}) and (\ref{eq:nov29-2}) we have 
\[
 {{\L}(\s)^{l-3}({\L}(\s)^+\s^-{\L}(\s)^{l-2})^\f}\lle d_{n_0+2q+1}d_{n_0+2q+2}\ldots \lle ({\L}(\s)^+\s^{l-2}\s^-)^\f,
\]
which implies $d_{n_0+2q+1}\ldots d_{n_0+3q}\in\set{{\L}(\s), {\L}(\s)^+}$.
Continuing  this procedure, and by (\ref{eq:nov29-1}) we can find $i_1\in\set{0,1,\ldots, l-2}$ such that $d_{n_0+1}\ldots d_{n_0+(i_1+2)q}=\s^-{\L}(\s)^{i_1}{\L}(\s)^+$.

Write $n_1:=n_0+(i_1+1)q$. Since $d_{n_1+1}\ldots d_{n_1+q}={\L}(\s)^+$, by (\ref{eq:nov29-1}) it follows that 
\begin{equation}\label{eq:nov29-3}
  (\s^-{\L}(\s)^{l-2}{\L}(\s)^+)^\f\lle d_{n_1+q+1}d_{n_1+q+2}\ldots\lle (\s^{l-2}\s^-{\L}(\s)^+)^\f.
\end{equation}
This implies $d_{n_1+q+1}\ldots d_{n_1+2q}\in\set{\s^-, \s}$. If $d_{n_1+q+1}\ldots d_{n_1+2q}=\s$, then $l\ge 3$, and by (\ref{eq:nov29-1}) and (\ref{eq:nov29-3}) we have 
\[
(\s^-{\L}(\s)^{l-2}{\L}(\s)^+)^\f\lle d_{n_1+2q+1}d_{n_1+2q+2}\ldots\lle {\s^{l-3}(\s^-{\L}(\s)^+\s^{l-2})^\f,}
\]
which implies $d_{n_1+2q+1}\ldots d_{n_1+3q}\in\set{\s^-, \s}$. Continuing this procedure and by (\ref{eq:nov29-1}), we can find $i_2\in\set{0,1,\ldots, l-2}$ such that $d_{n_1+1}\ldots d_{n_1+(i_2+2)q}={\L}(\s)^+\s^{i_2}\s^-$.

Proceeding the above discussion we can deduce that 
\[
d_{n_0+1}d_{n_0+2}\ldots =\s^-{\L}(\s)^{i_1}{\L}(\s)^+\s^{i_2}\s^-{\L}(\s)^{i_3}\ldots, 
\] 
where each $i_n\in\set{0,1,\ldots, l-2}$.
By (\ref{eq:bullet-operator}) this implies that $d_{n_0+1}d_{n_0+2}\ldots=\s\bullet(z_i)$ for some $(z_i)\in\set{0,1}^\N$, completing the proof.
\end{proof}
{For a periodic sequence $\c=(c_1\ldots c_m)^\f$, note that $km$ is   a period of $\c$ for any $k\in\N$.}
Recall that $\mathcal F^*$ consists of all Farey words of length at least two. Then for $j\in\N$ let 
   \[
   \Lambda_j:=\{\s_1\bullet\s_2\bullet\cdots\bullet\s_j:\s_i \in\mathcal F^*~~\forall 1\le i\le j\}.
    \]
    Motivated by \cite[Lemma 7.3]{Allaart-Kong-2026} we prove the following result.
     \begin{lemma}
  \label{lem:periodic-Gamma-S}
 If $\cs=\s_1\bullet\s_2\bullet\cdots\bullet\s_j\in\Lambda_j$ for some $j\in\N$, then any periodic sequence in 
  \[
  \Gamma(\cs):=\set{(d_i): \cs^\f\lle\si^n((d_i))\lle {\L}(\cs)^\f~\forall n\ge 0}
  \]
  has a period $|\cs|$. 
\end{lemma}
 
\begin{proof}     
 We proceed by induction on  $j$. The case for $j=1$ follows from Lemma \ref{lem:finite-set}. Now  suppose the statement holds for all $\cs\in\Lambda_{j-1}$ for some  $j\geq 2$, and we consider $\Lambda_j$. Take $\cs\in\Lambda_j$, then we can write $\cs=\s\bullet\r$, where $\s$ is a Farey word and $\r\in\Lambda_{j-1}$. Let $(d_i)\in\Gamma(\cs)$ be a periodic sequence. Since $\s\bullet\r\lge \s\bullet 0^{l-1}1$ and  ${\L}(\s\bullet\r)=\s\bullet\L(\r)\lle \s\bullet 1^{l-1}0$ where $l=|\r|$,   it follows that
 \begin{equation}\label{eq:dec24-3}
  (\s\bullet 0^{l-1}1)^\f\lle(\s\bullet\r)^\f\lle\si^n((d_i))\lle{\L}(\s\bullet\r)^\f\lle(\s\bullet 1^{l-1}0)^\f\quad\forall n\ge 0.
   \end{equation}
If  neither $\s^-$ nor ${\L}(\s)^+$ occurs in $(d_i)$, then by (\ref{eq:dec24-3}) we have $(d_i)\in\Gamma(\s)$. By Lemma \ref{lem:finite-set} it follows that $(d_i)$ has a period $|\s|$, and so it has a period $|\cs|=l|\s|$. 

If $(d_i)$ contains $\s^-$ or ${\L}(\s)^+$, then by Lemma \ref{lem:renormalization} there exists $n_0\in\N_0$ such that $\si^{n_0}((d_i))  =\s\bullet (z_i)$ for some $(z_i)\in\{0,1\}^\N$.  Since $(d_i)\in\Gamma(\s\bullet\r)$ is periodic, it follows that  $\s\bullet(z_i)=\si^{n_0}((d_i))\in\Gamma(\s\bullet\r)$ is periodic.  {We claim that $(z_i)\in\Gamma(\r)$ and $(z_i)$ is periodic. 

Since $\s\bullet(z_i)\in\Gamma(\s\bullet\r)$, we have 
\begin{equation}\label{eq:1}
(\s\bullet\r)^\f\lle\si^n(\s\bullet(z_i))\lle(\L(\s\bullet\r))^\f\quad \forall n\geq 0.
\end{equation}
Note that   $\r$ begins with $0$ and ends with $1$, and   $\L(\r)$ begins with $1$ and ends with $0$. Then 
by Lemma \ref{lem:property-bullet}  we have $(\s\bullet\r)^\f=\s\bullet\r^\f$ and $(\L(\s\bullet\r))^\f=(\s\bullet\L(\r))^\f=\s\bullet\L(\r)^\f$. Observe that $\s\bullet(z_i)\in\set{\s^-,\s, \L(\s), \L(\s)^+}^\N$. Since $\s\bullet\r^\f$ begins with $\s^-$ and $\s\bullet\L(\r)^\f$ begins with $\L(\s)^+$, by (\ref{eq:1}) it follows that
\[
\s\bullet\r^\f\lle\s\bullet\si^n((z_i))\lle\s\bullet\L(\r)^\f \quad \forall n\geq 0,
\]which together with  Lemma \ref{lem:property-bullet} (ii) implies 
\[
\r^\f\lle\si^n((z_i))\lle\L(\r)^\f\quad \forall n\ge 0.\]
 So, $(z_i)\in\Gamma(\r)$. Note that $\s\bullet(z_i)\in\set{\s^-,\s,\L(\s),\L(\s^+)}^\N$. Since $\s$ is a Farey word of length at least $2$, the four blocks $\s^-,\s,\L(\s),\L(\s)^+$ are all distinct. So the  periodicity of $\s\bullet(z_i)$ implies the periodicity of $(z_i)$. This establishes the claim.
} 

Therefore, by the induction hypothesis $(z_i)$ has a period $|\r|$. This implies that $\si^{n_0}((d_i))=\s\bullet(z_i)$ has a period $|\s|\cdot|\r|=|\cs|$. Note that $(d_i)$ is a {periodic} sequence. We conclude that  $|\cs|$ is also a period of  $(d_i)$. 
Hence, by induction this completes the proof.
\end{proof}
{Recall that $\cL_m$ consists of all length $m$ Lyndon words.}
The following lemma plays an important role in the proof of Proposition \ref{prop:critical-value-AB}.
\begin{lemma}\label{lem:admissible-word-cs-bullet-r}
   Let $(m, k_1,\ldots,k_j)$ be an admissible $m$-chain with   $m_i=\gcd(m,k_1,\ldots,k_{i-1})$ for $1\le i\le j$. {Let $\ell\in\set{2,\ldots,j}$.} If $d_1\ldots d_m\in\cL_m$ satisfies  
{  \[
  \left(\w_{\frac{k_1}{m_1}}\bullet\cdots\bullet\w_{\frac{k_{\ell-1}}{m_{\ell-1}}}\bullet 0^{m_\ell-1}1\right)^\f\lle \si^n((d_1\ldots d_m)^\f)\lle \left(\w_{\frac{k_1}{m_1}}\bullet\cdots\bullet\w_{\frac{k_{\ell-1}}{m_{\ell-1}}}\bullet 1^{m_\ell-1}0\right)^\f
  \]for all $n\ge 0$, 
  then there exists $\r\in\cL_{m_\ell}$ such that
  \[d_1\ldots d_m=\w_{\frac{k_1}{m_1}}\bullet\cdots\bullet\w_{\frac{k_{\ell-1}}{m_{\ell-1}}}\bullet\r.\] }
\end{lemma}
 
\begin{proof}
Let $\cs:=\w_{\frac{k_1}{m_1}}\bullet\cdots\bullet\w_{\frac{k_{\ell-1}}{m_{\ell-1}}}$, and write $(d_i)=(d_1\ldots d_m)^\f$.
Note that $m_\ell>k_\ell\ge 1$. Furthermore, by   (\ref{eq:bullet-operator}) we have $\cs\bullet 0^{m_\ell-1}1=\cs^-{\L}(\cs)^{m_\ell-2}{\L}(\cs)^+$ and $\cs\bullet 1^{m_\ell-1}0={\L}(\cs)^+\cs^{m_\ell-2}\cs^-$. Then  
   \begin{equation}\label{eq:nov29-11}
    (\cs^-{\L}(\cs)^{m_\ell-2}{\L}(\cs)^+)^\f \lle\si^n((d_i))\lle ({\L}(\cs)^+\cs^{m_\ell-2}\cs^-)^\f\quad\forall n\ge 0.
  \end{equation}
If $(d_i)$ contains neither $\cs^-$ nor $\L(\cs)^+$, then by (\ref{eq:nov29-11}) we have $(d_i)\in\Gamma(\cs)$. Since $(d_i)=(d_1\ldots d_m)^\f$ is periodic, by Lemma \ref{lem:periodic-Gamma-S} it follows that $(d_1\ldots d_m)^\f$ has a period $|\cs|=\frac{m}{m_\ell}$, leading to a contradiction with $d_1\ldots d_m\in\cL_m$.

So, $(d_i)$ contains either $\cs^-$ or $\L(\cs)^+$. By (\ref{eq:nov29-11}) and Lemma \ref{lem:renormalization} there exists $n_0\ge 0$ such that 
$
d_{n_0+1}d_{n_0+2}\ldots =\cs\bullet(z_i)  
$ 
for some $(z_i)\in\set{0,1}^\N$.
Note that $|\cs|=\frac{m}{m_\ell}$ and $(d_i)=(d_1\ldots d_m)^\f$ is periodic with $d_1\ldots d_m\in\cL_m$. By {Lemma \ref{lem:property-bullet} (iii)}  this implies that 
\[
(d_1\ldots d_m)^\f=\cs\bullet\r^\f=(\cs\bullet\r)^\f 
\]
 for some  $\r\in\cL_{m_\ell}$. Hence, $d_1\ldots d_m=\cs\bullet\r$ as required.  
\end{proof}

 \begin{proof}
   [Proof of Proposition \ref{prop:critical-value-AB}]
   First we prove (i). Take $\beta\in(\beta_{m,k_1,\ldots,k_j}, \beta_{m,k_1,\ldots,k_j^+}]$. Then by (\ref{eq:partition-base}),  Lemmas \ref{lem:delta-beta} and \ref{lem:property-bullet}  it follows that 
\begin{equation}
\label{eq:Nov30-1}
(\cs \bullet\L(\w_{k_j/m_j}))^\f\prec\de(\beta)\lle(\cs \bullet\L(\w_{k_j^+/m_j}))^\f,
\end{equation}
where $\cs=\w_{k_1/m_1}\bullet\cdots\bullet\w_{k_{j-1}/m_{j-1}}$. We will prove that 
  $\tau_m(\beta)=((\cs \bullet\w_{k_j/m_j})^\f)_\beta$. Note by (\ref{eq:Nov30-1}) and Lemma \ref{lem:property-bullet} that for any $n\ge 0$,
    \begin{align*}
    (\cs\bullet\w_{k_j/m_j})^\f\lle\si^n(\cs\bullet\w_{k_j/m_j})^\f&\lle{\L}(\cs\bullet\w_{k_j/m_j})^\f 
     =(\cs\bullet{\L}(\w_{k_j/m_j}))^\f \prec\de(\beta).
     \end{align*}
  Since $\cs\bullet\w_{k_j/m_j}$ is not periodic, by the definition of $\tau_m(\beta)$ it follows that $\tau_m(\beta)\geq((\cs\bullet\w_{k_j/m_j})^\f)_\beta=:t_*$. Furthermore, by Lemma \ref{lem:greedy-expansion} the sequence $(\cs\bullet\w_{k_j/m_j})^\f$ is the greedy $\beta$-expansion of $t_*$.

To prove $\tau_m(\beta)\leq t_*$, we take an arbitary $t>t_*$, and it suffices to prove $\tau_m(\beta)\leq t$. Suppose on the contrary that $\tau_m(\beta)>t$. Then there exists $d_1\ldots d_m\in\cL_m$ such that 
 \begin{equation}
  \label{eq:Nov30-2}
  (t_i)\lle\si^n(d_1\ldots d_m)^\f\prec\de(\beta) \quad \forall n\geq0,
 \end{equation}
 where $(t_i)$ is the greedy $\beta$-expansion of $t$. Since $t>t_*=((\cs\bullet\w_{k_j/m_j})^\f)_\beta$, we have $(t_i)\succ(\cs\bullet\w_{k_j/m_j})^\f$. Then by (\ref{eq:Nov30-1}) and (\ref{eq:Nov30-2}) it follows that
  \begin{equation}
  \label{eq:Nov30-3}
  (\cs\bullet0^{m_j-1}1)^\f\lle(\cs\bullet\w_{k_j/m_j})^\f\prec\si^n(d_1\ldots d_m)^\f\prec(\cs\bullet\L(\w_{k_j^+/m_j}))^\f\lle(\cs \bullet1^{m_j-1}0)^\f
  \end{equation}
  for all $n\geq 0$. By (\ref{eq:Nov30-3}) and Lemma \ref{lem:admissible-word-cs-bullet-r} we can find $\r\in\cL_{m_j}$ such that $d_1\ldots d_m=\cs\bullet\r$. In fact, by (\ref{eq:Nov30-3}) we also have
   \[
     \cs\bullet\r=d_1\ldots d_m\succ \cs\bullet\w_{k_j/m_j},
     \]
which together with Lemma \ref{lem:monotonicity-lyndon-perron-m-k} implies that $\r\succ\w_{k_j/m_j}=\max\bigcup_{\ell=1}^{k_j} {\cL}_{m_j,\ell}$. Since $\r\in\cL_{m_j}=\bigcup_{\ell=1}^{m_j-1}\cL_{m_j,\ell}$, we have  $\r\in\bigcup_{\ell=k_j^+}^{m_j-1}{\cL}_{m_j,\ell}$,  which implies $\L(\r) \lge \p_{m_j,k_j^+}$ by Lemma \ref{lem:monotonicity-lyndon-perron-m-k}. So, we can find $n_1\in\N$ such that 
 \begin{align*}
  \si^{n_1}(d_1\ldots d_m)^\f&={\L}(d_1\ldots d_m)^\f={\L}(\cs\bullet\r)^\f\\
  &=(\cs\bullet{\L}(\r))^\f  \succcurlyeq (\cs\bullet\p_{m_j,k_j^+})^\f =(\cs\bullet{\L(\w_{k_j^+/m_j})})^\f\succcurlyeq\de(\beta),
   \end{align*}
   where the last equality follows by Theorem \ref{main-2:L-P-mk} (i) and $\gcd(k_j^+,m_j)=1$, and the last inequality holds by (\ref{eq:Nov30-1}). This again leads to a contradiction with (\ref{eq:Nov30-2}). So, $\tau_m(\beta)\le t$. Since $t>t_*$ was arbitrary, we conclude that $\tau_m(\beta)\le t_*$.

 Next we prove (ii).   Take $\beta\in(\beta_{m,k_1,\ldots,k_j}, \beta_{m,k_1,\ldots,k_j^+,1}]$. Then $\gcd(k_j,m_j)=1$ and $\gcd(k_j^+,m_j)=m_{*}>1$. By  (\ref{eq:partition-base}),  Lemmas \ref{lem:delta-beta} and \ref{lem:property-bullet}  it follows that 
\begin{equation*}
{(\cs \bullet\L(\w_{k_j/m_j}))^\f}\prec\de(\beta)\lle(\cs \bullet\L(\w_{k_j^+/m_j}\bullet\w_{1/m_{*}}))^\f,
\end{equation*}
where $\cs=\w_{k_1/m_1}\bullet\cdots\bullet\w_{k_{j-1}/m_{j-1}}$. Note by Theorem \ref{main-2:L-P-mk} (ii) that \[\L(\w_{k_j^+/m_j}\bullet\w_{1/m_{*}})=\w_{k_j^+/m_j}\bullet 10^{m_{*}-1}=\p_{m_j,k_j^+}.\] Then by the same argument as in (i) we can show that $\tau_m(\beta)=((\cs\bullet\w_{k_j/m_j})^\f)_\beta$, completing the proof.
 \end{proof}

 \subsection{Critical value $\tau_m$ in Type C and Type D partition intervals}
 
 Recall that a Type C partition interval has the form $I^C_{m,k_1,\ldots,k_j}=(\beta_{m,k_1,\ldots,k_j}, \beta_{m,k_1,\ldots,k_{j-1}^+}]$, where $k_j=m_j-1$ and $\gcd(k_{j-1}^+,m_{j-1})=1$. Furthermore, a Type D partition interval has the form $I^D_{m,k_1,\ldots,k_j}=(\beta_{m,k_1,\ldots,k_j}, \beta_{m,k_1,\ldots, k_{j-1}^+,1}]$, where $k_j=m_j-1$ and $\gcd(k_{j-1}^+,m_{j-1})>1$. Note that in this case we have  $k_{j-1}<m_{j-1}-1$, since otherwise $m_j=1$ and thus $(m,k_1,\ldots,k_j)$ can not be an admissible $m$-chain.
 \begin{lemma}
   \label{lem:partition-base-order-CD}
   The Type C and Type D partition intervals are well-defined.
 \end{lemma}
 \begin{proof}
 Let $I_{m,k_1,\ldots,k_j}^C=(\beta_{m,k_1,\ldots,k_j},\beta_{m,k_1,\ldots,k_{j-1}^+}]$ be a Type C partition interval. Then $k_j=m_j-1$ and $\gcd(k_{j-1}^+, m_{j-1})=1$. Note by (\ref{eq:partition-base})  and Lemma \ref{lem:property-bullet} that 
 \[
 \de(\beta_{m,k_1,\ldots,k_j})=\w_{k_1/m_1}\bullet\cdots\bullet\w_{k_{j-2}/m_{j-2}}\bullet\L(\w_{k_{j-1}/m_{j-1}}\bullet\w_{k_j/m_j})^\f
 \]
 and
 \[
 \de(\beta_{m,k_1,\ldots,k_{j-1}^+})=\w_{k_1/m_1}\bullet\cdots\bullet\w_{k_{j-2}/m_{j-2}}\bullet\L(\w_{k_{j-1}^+/m_{j-1}})^\f.
 \]
 Since $k_j=m_j-1$, by Lemmas \ref{lem:delta-beta} and \ref{lem:property-bullet}, to prove $\beta_{m,k_1,\ldots,k_j}<\beta_{m,k_1,\ldots,k_{j-1}^+}$, it suffices to prove 
 \begin{equation}\label{eq:type-C-1}
   \w_{k_{j-1}/m_{j-1}}\bullet 1^{m_j-1}0\prec\L(\w_{k_{j-1}^+/m_{j-1}}).
 \end{equation}
 
 Let $\s:=\w_{k_{j-1}/m_{j-1}}=s_1\ldots s_{|\s|}$ and $\r:=\w_{k_{j-1}^+/m_{j-1}}=r_1\ldots r_{|\r|}$. Note that $\gcd(k_{j-1}, m_{j-1})=m_j$ and $\gcd(k_{j-1}^+, m_{j-1})=1$. Then $|\s|=\frac{m_{j-1}}{m_j}$ and $|\r|=m_{j-1}$. Thus, $|\r|=m_j|\s|$. Observe by Lemma \ref{lem:Farey-words-characterization} that $\L(\s)^+$ can be obtained from $\s$ by changing its first digit to $1$, and similarly, $\L(\r)$ can be obtained from $\r^-$ by changing its first digit to $1$. Note by (\ref{eq:bullet-operator}) that 
 $
 \s\bullet 1^{m_j-1}0=\L(\s)^+\s^{m_j-2}\s^-$. Then (\ref{eq:type-C-1}) is equivalent to
 \begin{equation}\label{eq:type-C-2}
 \s^{m_j}\prec \r. 
 \end{equation}
 Since $\s=\w_{k_{j-1}/m_{j-1}}$ and $\r=\w_{k_{j-1}^+/m_{j-1}}$, by using $\frac{k_{j-1}}{m_{j-1}}<\frac{k_{j-1}^+}{m_{j-1}}$ and Lemma \ref{lem:bijection-Fareyword-Rational} it follows that  
 \[
 \s=s_1\ldots s_{|\s|}\lle r_1\ldots r_{|\s|}.
 \]
 If $\s\prec r_1\ldots r_{|\s|}$, then (\ref{eq:type-C-2}) follows directly. Otherwise, we assume $\s=r_1\ldots r_{|\s|}$. Note that $\r$ is a Farey word, which is also a Lyndon word. Then by $|\r|=m_j|\s|$ and Lemma \ref{lem:Lyndon-words} we conclude that 
 \[
 \s\lle r_{i|\s|+1}\ldots r_{(i+1)|\s|} \quad\forall 0\le i<m_j-1;\quad\textrm{and}\quad \s\prec r_{(m_j-1)|\s|+1}\ldots r_{m_j|\s|}.
 \]
 This again proves (\ref{eq:type-C-2}). So, $\beta_{m,k_1,\ldots,k_j}<\beta_{m,k_1,\ldots,k_{j-1}^+}$. This proves that Type C partition intervals are well-defined.
 
 Next we consider Type D partition intervals. Take a Type D partition interval $I_{m,k_1,\ldots,k_j}^D=(\beta_{m,k_1,\ldots,k_j}, \beta_{m,k_1,\ldots,k_{j-1}^+,1}]$. Then $m_j=\gcd(k_{j-1},m_{j-1})>1$, $k_j=m_j-1$ and $\tilde m_j:=\gcd(k_{j-1}^+,m_{j-1})>1$. Note by (\ref{eq:partition-base}) and Lemma \ref{lem:property-bullet} that
 \[
 \de(\beta_{m,k_1,\ldots,k_j})=\w_{k_1/m_1}\bullet\cdots\bullet\w_{k_{j-2}/m_{j-2}}\bullet\L(\w_{k_{j-1}/m_{j-1}}\bullet \w_{k_j/m_j})^\f
 \] 
 and
 \[
 \de(\beta_{m,k_1,\ldots,k_{j-1}^+,1})=\w_{k_1/m_1}\bullet\cdots\bullet\w_{k_{j-2}/m_{j-2}}\bullet\L(\w_{k_{j-1}^+/m_{j-1}}\bullet\w_{1/\tilde m_j})^\f.
 \]
Observe that $k_j=m_j-1$. By Lemmas \ref{lem:delta-beta} and \ref{lem:property-bullet}, to prove $\beta_{m,k_1,\ldots,k_j}<\beta_{m,k_1,\ldots,k_{j-1}^+,1}$ it suffices to prove 
 \begin{equation}\label{eq:type-D-1}
 \w_{k_{j-1}/m_{j-1}}\bullet 1^{m_j-1}0\prec \w_{k_{j-1}^+/m_{j-1}}\bullet 10^{\tilde m_j-1}.
 \end{equation}
 
 Write 
 \[\s:=\w_{k_{j-1}/m_{j-1}}=s_1\ldots s_{|\s|}\quad\textrm{and}\quad\r:=\w_{k_{j-1}^+/m_{j-1}}=r_1\ldots r_{|\r|}.\] Note that $\gcd(k_{j-1},m_{j-1})=m_j$ and $\gcd(k_{j-1}^+, m_{j-1})=\tilde m_j$. Then $|\s|=\frac{m_{j-1}}{m_j}$ and $|\r|=\frac{m_{j-1}}{\tilde m_j}$. Observe by (\ref{eq:bullet-operator}) that $\s\bullet 1^{m_j-1}0=\L(\s)^+\s^{m_j-2}\s^-$ and $\r\bullet 10^{\tilde m_j-1}=\L(\r)^+\r^-\L(\r)^{\tilde m_j-2}$. Furthermore, $\L(\s)^+$ can be obtained from $\s$ by changing its first digit from zero to one, and similarly, $\L(\r)^+$ is obtained from $\r$ by changing its first digit from zero to one. Therefore, (\ref{eq:type-D-1}) is equivalent to 
 \begin{equation}\label{eq:type-D-2}
   \s^{m_j-1}\s^-\prec \r\r^-\L(\r)^{\tilde m_j-2}.
 \end{equation}
Note that $\s=\w_{k_{j-1}/m_{j-1}}$ and $\r=\w_{k_{j-1}^+/m_{j-1}}$ are Farey words. Then by using $\frac{k_{j-1}}{m_{j-1}}<\frac{k_{j-1}^+}{m_{j-1}}$ and Lemma \ref{lem:bijection-Fareyword-Rational} it follows that $s_1\ldots s_{n}\lle r_1\ldots r_n$ for all $n\le\min\set{|\s|, |\r|}$.  
 If $|\s|\ge |\r|$, by the construction of Farey words from rational numbers (see (\ref{eq:farey-word})) we obtain that $s_1\ldots s_{|\r|}\prec \r$, which proves (\ref{eq:type-D-2}).
 Next we assume $|\s|<|\r|$. write $|\r|=p|\s|+q$ with $q\in\set{0,1,\ldots,|\s|-1}$. Note by Lemma \ref{lem:bijection-Fareyword-Rational} that $\s=s_1\ldots s_{|\s|}\lle r_1\ldots r_{|\s|}$. If $\s\prec r_1\ldots r_{|\s|}$, then we are done and prove (\ref{eq:type-D-2}). Otherwise, $\s=r_1\ldots r_{|\s|}$. Since $\r=r_1\ldots r_{|\r|}$ is a Lyndon word, by Lemma \ref{lem:Lyndon-words} we can conclude that  
 \[\s^{p}s_1\ldots s_{q}=(s_1\ldots s_{|\s|})^p s_1\ldots s_q\prec r_1\ldots r_{|\r|}=\r,\]
 which also implies (\ref{eq:type-D-2}). So, $\beta_{m,k_1,\ldots,k_j}<\beta_{m,k_1,\ldots,k_{j-1}^+,1}$. This proves that Type D partition intervals are well-defined. 
 \end{proof}

Now we determine the critical value $\tau_m(\beta)$ for $\beta$ in any Type C any Type D partition intervals. 
 \begin{proposition}
   \label{prop:critical-value-CD}
   \begin{enumerate}[{\rm(i)}]
   \item Let $I^C_{m,k_1,\ldots,k_j}=(\beta_{m,k_1,\ldots,k_j}, \beta_{m,k_1,\ldots,k_{j-1}^+}]$ be a Type C partition interval. Then for any $\beta\in I^C_{m,k_1,\ldots, k_j}$ we have
   \[
   \tau_m(\beta)=((\w_{k_1/m_1}\bullet\w_{k_2/m_2}\bullet\cdots\bullet\w_{k_j/m_j})^\f)_\beta,
   \]
   where $m_i=\gcd(m,k_1,\ldots,k_{i-1})$ for $1\le i\le j$. 
   
   \item Similarly, let $I^D_{m,k_1,\ldots, k_j}=(\beta_{m,k_1,\ldots,k_j},  \beta_{m,k_1,\ldots, k_{j-1}^+,1}]$ be a Type D partition interval. Then for any $\beta\in I^D_{m,k_1,\ldots,k_j}$ we have 
       \[
   \tau_m(\beta)=((\w_{k_1/m_1}\bullet\w_{k_2/m_2}\bullet\cdots\bullet\w_{k_j/m_j})^\f)_\beta.
   \]
   \end{enumerate}
 \end{proposition}
 \begin{proof}
   First we prove (i). Take $\beta\in I^C_{m,k_1,\ldots, k_j}=(\beta_{m,k_1,\ldots,k_j}, \beta_{m,k_1,\ldots,k_{j-1}^+}]$. Then $k_j=m_j-1$ and $\gcd(k_{j-1}^+, m_{j-1})=1$.  By (\ref{eq:partition-base}),  Lemmas \ref{lem:delta-beta} and \ref{lem:property-bullet}  it follows that 
\begin{equation}
\label{eq:dec30-1}
(\cs \bullet\L(\w_{k_{j-1}/m_{j-1}}\bullet\w_{k_j/m_j}))^\f\prec\de(\beta)\lle(\cs \bullet\L(\w_{k_{j-1}^+/m_{j-1}}))^\f,
\end{equation}
where $\cs=\w_{k_1/m_1}\bullet\cdots\bullet\w_{k_{j-2}/m_{j-2}}$. We will prove that 
  $\tau_m(\beta)=((\cs \bullet\w_{k_{j-1}/m_{j-1}}\bullet\w_{k_j/m_j})^\f)_\beta$. Note by (\ref{eq:dec30-1}) and Lemma \ref{lem:property-bullet} that for any $n\ge 0$,
    \begin{align*}
    (\cs\bullet\w_{k_{j-1}/m_{j-1}}\bullet\w_{k_j/m_j})^\f&\lle\si^n(\cs\bullet\w_{k_{j-1}/m_{j-1}}\bullet\w_{k_j/m_j})^\f \\ &\lle{\L}(\cs\bullet\w_{k_{j-1}/m_{j-1}}\bullet\w_{k_j/m_j})^\f \\
     &=(\cs\bullet{\L}(\w_{k_{j-1}/m_{j-1}}\bullet\w_{k_j/m_j}))^\f \prec\de(\beta).
     \end{align*}
  Since $\cs\bullet\w_{k_{j-1}/m_{j-1}}\bullet\w_{k_j/m_j}$ is not periodic, by the definition of $\tau_m(\beta)$ it follows that $\tau_m(\beta)\geq((\cs\bullet\w_{k_{j-1}/m_{j-1}}\bullet\w_{k_j/m_j})^\f)_\beta=:t_*$. Furthermore, by Lemma \ref{lem:greedy-expansion} the sequence $(\cs\bullet\w_{k_{j-1}/m_{j-1}}\bullet\w_{k_j/m_j})^\f$ is the greedy $\beta$-expansion of $t_*$.

To prove $\tau_m(\beta)\leq t_*$, we take an arbitary $t>t_*$, and it suffices to prove $\tau_m(\beta)\leq t$. Suppose on the contrary that $\tau_m(\beta)>t$. Then there exists $d_1\ldots d_m\in\cL_m$ such that 
 \begin{equation}
  \label{eq:dec30-2}
  (t_i)\lle\si^n(d_1\ldots d_m)^\f\prec\de(\beta) \quad \forall n\geq0,
 \end{equation}
 where $(t_i)$ is the greedy $\beta$-expansion of $t$. Since $t>t_*=((\cs\bullet\w_{k_{j-1}/m_{j-1}}\bullet\w_{k_j/m_j})^\f)_\beta$, we have $(t_i)\succ(\cs\bullet\w_{k_{j-1}/m_{j-1}}\bullet\w_{k_j/m_j})^\f$. Then by (\ref{eq:dec30-1}) and (\ref{eq:dec30-2}) it follows that
  \begin{equation}
  \label{eq:dec30-3}
  \begin{split}
  (\cs\bullet0^{m_{j-1}-1}1)^\f&\lle(\cs\bullet\w_{k_{j-1}/m_{j-1}}\bullet\w_{k_j/m_j})^\f\\
  &\prec\si^n(d_1\ldots d_m)^\f \prec(\cs\bullet\L(\w_{k_{j-1}^+/m_{j-1}}))^\f\lle(\cs\bullet1^{m_{j-1}-1}0)^\f
  \end{split}
  \end{equation}
  for all $n\geq 0$. By (\ref{eq:dec30-3}) and Lemma \ref{lem:admissible-word-cs-bullet-r} we can find $\r\in\cL_{m_{j-1}}$ such that $d_1\ldots d_m=\cs\bullet\r$. In fact, by (\ref{eq:dec30-3}) we also have
   \[
     \cs\bullet\r=d_1\ldots d_m\succ \cs\bullet\w_{k_{j-1}/m_{j-1}}\bullet\w_{k_j/m_j},
     \]
which, together with $k_j=m_j-1$ and Theorem \ref{main-2:L-P-mk} (ii), implies that 
\[\r\succ\w_{k_{j-1}/m_{j-1}}\bullet\w_{k_j/m_j}=\w_{k_{j-1}/m_{j-1}}\bullet 01^{m_j-1}=\l_{m_{j-1},k_{j-1}}=\max\bigcup_{\ell=1}^{k_{j-1}} {\cL}_{m_{j-1},\ell},\]
where the last equality follows by Lemma \ref{lem:monotonicity-lyndon-perron-m-k}.  Since $\r\in\cL_{m_{j-1}}=\bigcup_{\ell=1}^{m_{j-1}-1}\cL_{m_{j-1},\ell}$, we have  $\r\in\bigcup_{\ell=k_{j-1}^+}^{m_{j-1}-1}{\cL}_{m_{j-1},\ell}$,  which implies $\L(\r) \lge \p_{m_{j-1},k_{j-1}^+}$ by Lemma \ref{lem:monotonicity-lyndon-perron-m-k}. So, we can find $n_1\in\N$ such that 
 \begin{align*}
  \si^{n_1}(d_1\ldots d_m)^\f&={\L}(d_1\ldots d_m)^\f={\L}(\cs\bullet\r)^\f\\
  &=(\cs\bullet{\L}(\r))^\f  \succcurlyeq (\cs\bullet\p_{m_{j-1},k_{j-1}^+})^\f =(\cs\bullet {\L(\w_{k_{j-1}^+/m_{j-1}})})^\f\succcurlyeq\de(\beta),
   \end{align*}
   where the last equality follows by Theorem \ref{main-2:L-P-mk} (i) and $\gcd(k_{j-1}^+,m_{j-1})=1$, and the last inequality holds by (\ref{eq:dec30-1}). This again leads to a contradiction with (\ref{eq:dec30-2}). So, $\tau_m(\beta)\le t$, and hence $\tau_m(\beta)\le t_*$ as required.

 Next we prove (ii).   Take $\beta\in I^D_{m,k_1,\ldots, k_j}=(\beta_{m,k_1,\ldots,k_j}, \beta_{m,k_1,\ldots,k_{j-1}^+,1}]$. Then $k_j=m_j-1$ and $\gcd(k_{j-1}^+,m_{j-1})=m_{*}>1$. By  (\ref{eq:partition-base}),  Lemmas \ref{lem:delta-beta} and \ref{lem:property-bullet}  it follows that 
\begin{equation*}
(\cs\bullet \L(\w_{k_{j-1}/m_{j-1}}\bullet\w_{k_j/m_j}))^\f\prec\de(\beta)\lle(\cs \bullet\L(\w_{k_{j-1}^+/m_{j-1}}\bullet\w_{1/m_{*}}))^\f,
\end{equation*}
where $\cs=\w_{k_1/m_1}\bullet\cdots\bullet\w_{k_{j-2}/m_{j-2}}$. Note by Theorem \ref{main-2:L-P-mk} (ii) that \[\L(\w_{k_{j-1}^+/m_{j-1}}\bullet\w_{1/m_{*}})=\w_{k_{j-1}^+/m_{j-1}}\bullet 10^{m_{*}-1}=\p_{m_{j-1},k_{j-1}^+}.\] Then by the same argument as in (i) we can show that $\tau_m(\beta)=((\cs\bullet\w_{k_{j-1}/m_{j-1}}\bullet\w_{k_j/m_j})^\f)_\beta$, completing the proof.
 \end{proof}
 \begin{proof}
   [Proof of Theorem \ref{main-3:critical-value-details}] The result follows by Lemma \ref{lem:critical-value-first-last-interval}, Propositions \ref{prop:critical-value-AB} and \ref{prop:critical-value-CD}.
 \end{proof}

 \section{Final remarks on the critical value $\tau_m$}\label{sec:final-remarks}
 In this section we make some further remarks on Theorem \ref{main-3:critical-value-details}.
Given $\beta\in(1,2]$ and $m\in\N_{\ge 2}$, first  we give an algorithm to determine $\tau_m(\beta)$. {
\begin{itemize}
  \item[{\em Step} 1.] Put $m_1=m$. Choose the largest $k_1\in\set{1,2,\ldots,m_1-1}$   such that 
  \begin{equation}\label{eq:step-1} 
  (\L(\w_{k_1/m_1})^\f)_\beta{\leq} 1.
  \end{equation}
   {If no such $k_1$ exists, then $\beta<\beta_{m,1}$ and we are done with $\tau_m(\beta)=0$.} Let $m_2=\gcd(m_1,k_1)$, and suppose $m_2=1$. If the equality holds in (\ref{eq:step-1}), then $\beta=\beta_{m,k_1}$ and we are done by Theorem \ref{main-3:critical-value-details}. If the strictly inequality holds in (\ref{eq:step-1}), then $\beta\in I_{m,k_1}$ and we are done with $\tau_m(\beta)=(\w_{k_1/m_1}^\f)_\beta$.  So it remains to consider $m_2>1$, which will be discussed  in the next step.
  \item  [{\em Step} 2.] Choose the largest $k_2\in\set{1,\ldots,m_2-1}$ such that
  \begin{equation}\label{eq:step-2}
  (\L(\w_{k_1/m_1}\bullet\w_{k_2/m_2})^\f)_\beta\leq 1.
  \end{equation}
  If no such $k_2$ exists, then $\beta<\beta_{m,k_1,1}$ and $\beta$ belongs to the partition interval with its right endpoint $\beta_{m,k_1,1}$; in this case we  have $\tau_m(\beta)=(\d)_\beta$, where $\d$ is the greedy $\beta_{m,k_1,1}$-expansion of  $\tau_m(\beta_{m,k_1,1})$.  Let $m_3=\gcd(m_1,k_1,k_2)$, and suppose $m_3=1$. If the equality holds in (\ref{eq:step-2}), then $\beta=\beta_{m,k_1,k_2}$ and we are done by Theorem \ref{main-3:critical-value-details}. If the strictly inequality holds in (\ref{eq:step-1}), then $\beta\in I_{m,k_1,k_2}$ and we are done with $\tau_m(\beta)=((\w_{k_1/m_1}\bullet\w_{k_2/m_2})^\f)_\beta$. Therefore, it remains to consider $m_3>1$, which will be discussed in the next step.
       
    \item[{\em Step} j.] Continuing this process finitely many times,     we  will always find a {lexicographically largest} admissible $m$-chain $(m_1,k_1,\ldots,k_j)$ such that 
    \begin{equation}\label{eq:step-j}
    (\L(\w_{k_1/m_1}\bullet\cdots\w_{k_j/m_j})^\f)_\beta{\leq} 1,
   \end{equation} 
   where $m_i=\gcd(m_1,k_1,\ldots,k_{i-1})$ for $1\le i\le j$. If the equality holds in (\ref{eq:step-j}), then $\beta=\beta_{m,k_1,\ldots,k_j}$  and we are done by Theorem \ref{main-3:critical-value-details}. If the strictly inequality holds in (\ref{eq:step-j}), then $\beta\in I_{m,k_1,\ldots,k_j}$ and we  are done with $\tau_m(\beta)=((\w_{k_1/m_1}\bullet\w_{k_2/m_2}\bullet\cdots\bullet\w_{k_j/m_j})^\f)_\beta$.
\end{itemize}}
\begin{proposition}
  \label{prop:algorithm}
 Given $\beta\in(1,2]$ and $m\in\N_{\ge 2}$, the above algorithm   determines $\tau_m(\beta)$.
\end{proposition}
\begin{proof}
{In Step 1, suppose $(\L(\w_{k_1/m_1})^\f)_\beta\le 1$ fails for any $k_1\in\set{1,2,\ldots,m_1-1}$, then $(\L(\w_{1/m})^\f)_\beta> 1$, which implies $\de(\beta_{m,1})=\L(\w_{1/m})^\f\succ \de(\beta)$ by Lemma \ref{lem:delta-beta}. So, $\beta< \beta_{m,1}$, and then $\tau_m(\beta)=0$ by Theorem \ref{main-3:critical-value-details}. If $(\L(\s_{k_1/m_1})^\f)_\beta\leq 1$ for some largest $k_1\in\set{1,2,\ldots,m_1-1}$  and $m_2=\gcd(m_1,k_1)=1$, then by Lemma  \ref{lem:delta-beta}  it follows that \[
   \de(\beta_{m,k_1}) =\L(\w_{k_1/m_1})^\f\lle \de(\beta).
  \] 
  Again, by Lemma \ref{lem:delta-beta} we have $\beta\ge \beta_{m,k_1}$. So by the maximum of $k_1$ it follows that  $\beta=\beta_{m,k_1}$ or $\beta\in I_{m,k_1}$. 
  
  In Step 2, suppose  $m_2=\gcd(m_1,k_1)>1$. If $(\L(\w_{k_1/m_1}\bullet\w_{k_2/m_2})^\f)_\beta\le 1$ fails for any $k_2\in\set{1,2,\ldots,m_2-1}$, then $(\L(\w_{k_1/m_1}\bullet\w_{1/m_2})^\f)_\beta>1$, which implies $\de(\beta_{m,k_1,1})=\L(\w_{k_1/m_1}\bullet\w_{1/m_2})^\f\succ \de(\beta)$ by Lemma \ref{lem:delta-beta}. So, $\beta< \beta_{m,k_1,1}$, and by the maximum of $k_1$ it follows that $\beta$ belongs to the partition interval with its right endpoint $\beta_{m,k_1,1}$, and thus  $\tau_m(\beta)=(\d)_\beta$, where $\d$ is the greedy $\beta_{m,k_1,1}$-expansion of  $\tau_m(\beta_{m,k_1,1})$. If $(\L(\s_{k_1/m_1}\bullet\w_{k_2/m_2})^\f)_\beta\leq 1$ for some largest $k_2\in\set{1,2,\ldots, m_2-1}$  and $m_3= \gcd(m_1,k_1,k_2)=1$, then by Lemma  \ref{lem:delta-beta}  it follows that \[
   \de(\beta_{m,k_1,k_2}) =\L(\w_{k_1/m_1}\bullet\w_{k_2/m_2})^\f\lle \de(\beta).
  \] 
  Again, by Lemma \ref{lem:delta-beta} we have $\beta\ge \beta_{m,k_1,k_2}$. So by the maximum of $k_2$ it follows that  $\beta=\beta_{m,k_1,k_2}$ or $\beta\in I_{m,k_1,k_2}$. 
  
  Proceeding this algorithm it determines a unique admissible $m$-chain $(m, k_1,\ldots,k_j)$ such that   $(\L(\w_{k_1/m_1}\bullet\cdots\bullet\w_{k_j/m_j})^\f)_\beta\le 1$. Then by {Lemma  \ref{lem:delta-beta} it follows that \[
   \de(\beta_{m,k_1,\ldots, k_j}) =\L(\w_{k_1/m_1}\bullet\cdots\bullet\w_{k_j/m_j})^\f\lle \de(\beta).
  \] 
 Again, by Lemma \ref{lem:delta-beta} we have $\beta\ge \beta_{m,k_1,\ldots,k_j}$.} Observe that in each step we choose $k_i\in\set{1,2,\ldots,m_i-1}$ as large as possible.   So, $\beta=\beta_{m,k_1,\ldots,k_j}$ or $\beta\in I_{m,k_1,\ldots,k_j}$ as desired.
 }
\end{proof}


\begin{example}\label{ex:komornik-Loreti}
  Let {$\beta=\beta_{KL}\approx 1.78723$ be the Komornik-Loreti constant defined in \cite{Komornik-Loreti-1998}. It is known that $\beta_{KL}$ is transcendental (cf.~\cite{Allouche_Cosnard_2000}).} By using the above algorithm we   {determine $\tau_m(\beta)$  for $m=2,3,\ldots, 10$ as in the following table.} 
 \begin{center}
\begin{longtable}{p{1.3cm}|p{1.2cm}|p{1.2cm}|p{1.2cm}|p{1.2cm}|p{1.2cm}|p{1.2cm}|p{1.2cm}|p{1.2cm}|p{1.2cm}}
    \hline
    \centering$m$ & $2$ & $3$ & $4$ & $5$ & $6$ & $7$ & $8$ & $9$ & $10$\\
    \hline
    \centering$ I_{m,k_1,\ldots,k_j}$&$I_{2,1}$&$I_{3,1}$&$I_{4,2,1}$&$I_{5,2}$&$I_{6,3,1}$&$I_{7,3}$&$I_{8,4,2,1}$&$I_{9,4}$&$I_{10,5,2}$\\
    \hline
    \centering$\tau_m(\beta)\approx$ & $0.45574$&$0.21236$&$0.30286$&$0.24335$&$0.26894$&$0.25149$&$0.27292$&$0.25391$&$0.26988$ \\
    \hline
\end{longtable}
\end{center}
\end{example}
{ Given $\beta\in(1,2]$ and $m,n\in\N_{\ge 2}$, we want to compare $\tau_m(\beta)$ and $\tau_n(\beta)$. 
 In terms of Theorem \ref{main-3:critical-value-details}, excluding  the trivial case $\beta\leq\max\set{\beta_{m,1},\beta_{n,1}}$} we can find a unique admissible $m$-chain $(m,k_1,\ldots, k_j)$ and a unique admissible $n$-chain $(n,k_1',\ldots, k_l')$ such that 
  $
  \beta\in I_{m,k_1,\ldots,k_j}\cap I_{n,k_1',\ldots, k_\ell'},
  $
  where $I_{m,k_1,\ldots,k_j}$ and $I_{n,k_1',\ldots,k_\ell'}$ are corresponding partition intervals.
Then 
 \begin{equation}\label{eq:critical-value-m-n}
  \tau_m(\beta)=((\w_{k_1/m_1}\bullet\cdots\bullet\w_{k_j/m_j})^\f)_\beta, \quad
  \tau_n(\beta)=((\w_{k_1'/n_1}\bullet\cdots\bullet\w_{k_\ell'/n_\ell})^\f)_\beta,
  \end{equation}
  where {$m_i=\gcd(m,k_1,\ldots, k_{i-1})$ for $1\le i\le j$, and $n_i=\gcd(n,k_1,\ldots, k_{i-1}')$ for $1\le i\le \ell$}. For a vector $(r_1,\ldots, r_n)\in\R^n$ let $(r_1,r_2,\ldots, r_n)^\f=(r_1r_2\ldots r_n)^\f$ be a periodic sequence.

  \begin{proposition}\label{prop:critical-m>n}
  For any $\beta\in I_{m,k_1,\ldots,k_j}\cap I_{n,k_1',\ldots, k_\ell'}$ we have
    \[\tau_m(\beta)>\tau_n(\beta)\quad\textrm{if and only if}\quad \left(\frac{k_1}{m_1},\frac{k_2}{m_2},\ldots,\frac{k_j}{m_j}\right)^\f\succ\left(\frac{k_1'}{n_1},\frac{k_2'}{n_2},\ldots,\frac{k_\ell'}{n_\ell}\right)^\f.\]
  \end{proposition}
  \begin{proof}
  {Note by 
  (\ref{eq:critical-value-m-n}) that $(\w_{k_1/m_1}\bullet\cdots\bullet\w_{k_j/m_j})^\f$ is the greedy $\beta$-expansion of $\tau_m(\beta)$, and $(\w_{k_1'/n_1}\bullet\cdots\bullet\w_{k_l'/n_l})^\f$ is the greedy $\beta$-expansion of $\tau_n(\beta)$. Then by Lemma \ref{lem:greedy-expansion}}  
  it follows that $\tau_m(\beta)>\tau_n(\beta)$ if and only if $(\w_{{k_1}/{m_1}}\bullet\cdots\bullet\w_{{k_j}/{m_j}})^\f\succ(\w_{{k_1'}/{n_1}}\bullet\cdots\bullet\w_{{k_\ell'}/{n_\ell}})^\f$. So, it suffices to show that this is equivalent to $({k_1}/{m_1},\ldots,{k_j}/{m_j})^\f\succ({k_1'}/{n_1}, \ldots, {k_\ell'}/{n_\ell})^\f$. Observe by Lemma \ref{lem:bijection-Fareyword-Rational} that the map ${p}/{q}\mapsto \w_{{p}/{q}}$ is  strictly increasing on $\mathbb Q\cap(0,1)$. Then by Lemma \ref{lem:property-bullet} it follows that 
  \[\left(\w_{\frac{k_1}{m_1}}\bullet\cdots\bullet\w_{\frac{k_j}{m_j}}\right)^\f\succ\left(\w_{\frac{k_1'}{n_1}}\bullet\cdots\bullet
  \w_{\frac{k_\ell'}{n_\ell}}\right)^\f\quad\Longleftrightarrow\quad
  \left(\frac{k_1}{m_1}, \ldots,\frac{k_j}{m_j}\right)^\f\succ\left(\frac{k_1'}{n_1}, \ldots,\frac{k_\ell'}{n_\ell}\right)^\f, \]completing the proof.
\end{proof}
{Note that for any $m\in\N_{\ge 2}$, the $m$-partition intervals form a partition of $(1,2]$. Then the collection of all $m$-partition intervals for $m\in\N_{\ge 2}$ forms a nest of $(1,2]$. {More precisely, let $\mathcal I_m=\set{(1,\beta_{m,1}], (\beta_{m,m-1},2]}\cup\set{I_{m,k_1,\ldots,k_j}: (m,k_1,\ldots,k_j)\in\A_m\setminus\set{(m,m-1)}}$ be the partition of $(1,2]$ defined as in (\ref{eq:partition-intervals}). Then set $\mathcal N_n:=\bigvee_{m=1}^n\mathcal I_m$ forms a finer partition of $(1,2]$. So, $\set{\mathcal N_n}_{n=1}^\f$ is a nest of $(1,2]$.} Observe that each non-extremal $m$-partition interval $I_{m,k_1,\ldots,k_j}$ is determined by an admissible $m$-chain $(m,k_1,\ldots,k_j)$. This is also   determined by a rational vector $(k_1/m_1, k_2/m_2,\ldots,k_j/m_j)$ with $m_i=\gcd(m,k_1,\ldots,k_{i-1})$ for $1
\le i\le j$. So, we also write the partition interval $I_{m,k_1,\ldots,k_j}$ as $I_{k_1/m_1,\ldots,k_j/m_j}$. In the following we show that for any 
  sequence  $(r_i)\in\mathbb Q\cap(0,1)$, there is a unique $\beta\in(1,2]$ belonging to 
$
I_{r_1,r_2,\ldots, r_n}$
for all $n\in\N$.} 
\begin{proposition}\label{prop:beta-infinite-intervals}
  For any sequence $(r_i)\in\mathbb Q\cap(0,1)$ there exists a unique $\beta\in\bigcap_{n=1}^\f I_{r_1,r_2,\ldots, r_n}$ satisfying
  \begin{equation}\label{eq:beta-inifinite}
  \de(\beta)=\lim_{n\to\f}\L(\w_{r_1}\bullet\w_{r_2}\bullet\cdots\bullet\w_{r_n})^\f.
  \end{equation}
  {So, if $r_i=\frac{p_i}{q_i}$ with $\gcd(p_i,q_i)=1$ for all $i\in\N$, then for any $m=\prod_{i=1}^n q_i$ with $n\in\N$ we have 
  \[
  \tau_m(\beta)=((\w_{r_1}\bullet\w_{r_2}\bullet\cdots\bullet\w_{r_n})^\f)_\beta.
  \]}
\end{proposition}
\begin{proof}
  Take $n\in\N,$ and let 
\[P_i=p_i\prod_{k=i+1}^{n}q_k\quad \textrm{and}\quad Q_i=\prod_{k=i}^{n}q_k\quad\forall 1\le i\le n.\]
 Then  $(Q_1, P_1, P_2,\ldots, P_n)$ is an admissible $Q_1$-chain satisfying
 \[\frac{P_i}{Q_i}=\frac{p_i}{q_i} \quad\textrm{with} \quad Q_i=\gcd(Q_1, P_1,\ldots, P_{i-1})\quad\forall 1\le i\le n.\] This implies that  
 \[
  I_{r_1,r_2,\ldots, r_n}=I_{Q_1, P_1,\ldots, P_n}=:(\beta_n,\beta_n'],
  \]
where its left endpoint $\beta_n$ satisfies 
\begin{equation}\label{eq:beta-n}
  \de(\beta_n)=\L(\w_{P_1/Q_1}\bullet\w_{P_2/Q_2}\bullet\cdots\bullet\w_{P_n/Q_n})^\f=\L(\w_{r_1}\bullet\w_{r_2}\bullet\cdots\bullet\w_{r_n})^\f.
  \end{equation}
  So, {by Theorem \ref{main-3:critical-value-details}} it suffices to prove $\beta\in(\beta_n, \beta_n']$.
  
  Note by (\ref{eq:beta-inifinite}) and Lemma \ref{lem:property-bullet} that $\de(\beta)$ begins with {$\L(\w_{r_1}\bullet\cdots\bullet\w_{r_n})^+$}, which implies $\de(\beta)\succ\de(\beta_n)$ by (\ref{eq:beta-n}). Then by Lemma \ref{lem:delta-beta} we have $\beta>\beta_n$. To prove $\beta\le\beta_n'$ we consider the four types of partition intervals. Without loss of generality we assume $I_{r_1,\ldots r_n}$ is a Type A partition interval. Then $\beta_n'$ satisfies
  \begin{equation}\label{eq:beta-n'}
  \de(\beta_n')=\L(\w_{r_1}\bullet\cdots\bullet\w_{r_{n-1}}\bullet\w_{p_n^+/q_n})^\f.
  \end{equation}
  Note that $p_n^+/q_n>r_n$. Then $(r_1,\ldots r_{n-1}, \frac{p_n^+}{q_n})\succ (r_1,\ldots r_k)$ for all $k\ge n$. By (\ref{eq:beta-n'}), {Lemmas \ref{lem:property-bullet} and \ref{lem:bijection-Fareyword-Rational}} it follows that 
  \[
  \de(\beta_n')=\L(\w_{r_1}\bullet\cdots\bullet\w_{r_{n-1}}\bullet\w_{p_n^+/q_n})^\f\succ{\L(\w_{r_1}\bullet\cdots\bullet\w_{r_k})^\f}=\de(\beta_k)\quad\forall k\ge n.
  \]
  Since $\lim_{k\to\f}\beta_k=\beta$, by Lemma \ref{lem:delta-beta} we conclude that $\beta_n'\ge \beta$. This completes the proof. 
\end{proof}

 {If we take all $r_i=1/2$, then Proposition \ref{prop:beta-infinite-intervals} yields the unique base $\beta=\beta_{KL}$, which is the Komornik-Loreti constant considered in Example \ref{ex:komornik-Loreti}. So, by Proposition \ref{prop:beta-infinite-intervals} it follows that for any $n\ge 1$ we  have $\beta_{KL}\in I_{2^n, 2^{n-1},\ldots,2,1}$, and 
 \[
 \tau_{2^n}(\beta_{KL})=((\underbrace{\w_{1/2}\bullet\cdots\bullet\w_{1/2}}_{n})^\f)_\beta=((\underbrace{01\bullet\cdots\bullet 01}_{n})^\f)_\beta.
 \]
 It might be interesting to determine  $\tau_m(\beta_{KL})$ for all $m\in\N\setminus\set{2^n: n\in\N}$.
 }

 \section*{Acknowledgements} 
 {The authors thank the anonymous referee for many useful comments which greatly improve the paper.}
The authors also thank Professor Shishuo Fu for providing the online sequence A006874-OEIS. The first author was supported by the Chongqing Natural Science Foundation: CQYC20220511052 and the Scientific Research Innovation Capacity Support Project for Young Faculty No. ZYGXQNISKYCXNLZCXM-P2P.


\end{document}